\documentclass{article}
\usepackage{CJKutf8}
\usepackage{definitions}
\newcommand{\tr}{\mathrm{tr}}
\DeclareMathOperator*{\argmin}{arg\,min}

\title{New Understandings and Computation on Augmented Lagrangian Methods for Low-Rank Semidefinite Programming}
\author{Lijun Ding\thanks{University of California San Diego, Department of Mathematics (l2ding@ucsd.edu).} \and Haihao Lu\thanks{MIT, Sloan School of Management (haihao@mit.edu).} \and Jinwen Yang\thanks{Department of Industrial Engineering and Operations Research, Columbia University (jy3711@columbia.edu).}}
\date{}

\begin{document}

\maketitle

\begin{abstract}
    The Augmented Lagrangian Method (ALM) combined with the Burer-Monteiro (BM) factorization, dubbed ALM-BM, provides a powerful approach for solving large-scale low-rank semidefinite programs (SDPs). Despite its empirical success, the theoretical understanding of the resulting nonconvex ALM-BM subproblems, particularly concerning their structural properties and efficient subproblem solvability by first-order methods, still remains limited. This work addresses these notable gaps by providing a rigorous theoretical analysis. We demonstrate that, under appropriate regularity of the original SDP, termed as primal simplicity, ALM subproblems inherit crucial properties such as low-rankness, strict complementarity, and quadratic growth when the dual variable is localized. Building on these properties, we prove that the nonconvex ALM-BM subproblems can be solved to global optimality by gradient descent, achieving linear convergence under local initialization and proximity of the dual variable, with exact or over-parametrized factor rank. Through illustrative examples, we further establish the necessity of these assumptions, revealing them as intrinsic characteristics of the problem structure. Motivated by these theoretical insights, we propose ALORA, a rank-adaptive augmented Lagrangian method that builds upon the ALM-BM framework. The method dynamically adjusts the rank using spectral information and explores negative curvature directions to navigate the nonconvex landscape. Exploiting modern GPU computing architectures, ALORA exhibits strong numerical performance, solving SDPs with side dimension on the order of tens of millions in a few minutes.
\end{abstract}

\section{Introduction}
Semidefinite programming (SDP) is a central class of convex optimization problems in which the decision variable is a symmetric matrix $X\in \real^{n\times n}$, the objective is a linear function of $X$, and the feasible set is defined by linear equality/inequality constraints together with a \emph{positive-semidefinite (PSD)} cone constraint on $X$~\cite{vandenberghe1996semidefinite,todd2001semidefinite,wolkowicz2012handbook}.\footnote{We defer the mathematical formulation of SDP, as well as ALM and ALM-BM to be introduced in later paragraphs, to Section \ref{sec: Preliminary_SDP_AL_AL-BM}. Recall that a matrix is PSD if it is symmetric with all eigenvalues nonnegative. } 
The PSD constraint enables SDP
to capture a wide range of convex constraints and to have remarkable modeling power. Indeed, SDP 
not only generalizes the classical linear programming (LP), 
\footnote{If the variable matrix $X$ is restricted to be diagonal, then the PSD constraint collapses to simple non-negativity on each diagonal entry, and one recovers LP.
} but also 
has found extensive applications across diverse fields in capturing nonlinear relationships, including control and robotics~\cite{majumdar2020recent,martin2023guarantees}, power systems~\cite{lavaei2011zero,low2014convex1,low2014convex2,madani2014convex}, quantum information~\cite{mazziotti2011large,cavalcanti2016quantum,skrzypczyk2023semidefinite}, and combinatorial optimization~\cite{alizadeh1995interior,helmberg2000semidefinite,goemans2000combinatorial,lovasz2003semidefinite,benson2000solving}, among numerous others. 
%

%
%
%
%

Because of its remarkable modeling power, a wide range of algorithms have been developed for solving SDP, including classical  methods such as interior-point methods (IPM)~\cite{alizadeh1995interior,helmberg1996interior,todd1998nesterov} and numerous modern first-order methods, e.g., \cite{o2016conic,o2021operator,garstka2021cosmo,souto2022exploiting,wang2025accelerated}.
In this work, we focus on another classic algorithm for SDP, the \emph{augmented Lagrangian method (ALM)}, originally developed in the early 70s \cite{hestenes1969multiplier,powell1969method,rockafellar1976augmented}. 
In a nutshell, ALM reformulates the linearly constrained SDP by augmenting the standard Lagrangian with a quadratic penalty term that discourages linear constraints violations. Then, at each iteration, ALM approximately solves \emph{an augmented Lagrangian subproblem}, i.e., it approximately minimizes the augmented Lagrangian (with a fixed dual variable) over the primal PSD variable $X$, followed by an explicit dual update based on the current primal residual in terms of the linear constraint.

The above description of ALM reveals its two major advantageous features over other methods: \emph{flexible algorithm designs and 
robust convergence behavior}, making ALM a fertile ground for the development of recent scalable algorithms for solving SDPs~\cite{burer2003nonlinear,burer2005local,burer2006computational,zhao2010newton,yang2015sdpnal+,sun2020sdpnal+,yurtsever2021scalable}.
Indeed, the subproblem, which has no linear constraints and only a PSD variable $X$, is relatively simple; hence, it allows for \emph{flexible algorithm designs} to avoid dense matrix operations, such as those in IPMs. Moreover, thanks to the quadratic penalty term and its connection to the proximal point method \cite{rockafellar1976augmented}, ALM usually enjoys \emph{robust and stable empirical convergence behavior}, making it more favorable than some first-order methods. 

Despite its flexibility and robustness, a fundamental and general challenge in solving large-scale SDPs remains for ALM: the computational cost associated with the high-dimensional matrix variables $X$. Specifically for ALM, the challenge appears in solving the ALM subproblem with the variable $X$. Standard methods, such as the projected gradient, typically operate on the full $n\times n$ matrix variable $X$, requiring $O(n^2)$ memory and incurring per-iteration costs that quickly become prohibitive as $n$ grows. 
To address this general challenge, motivated by the low-rankness of the optimal solution to many SDPs \cite{barvinok1995problems,pataki1998rank,recht2010guaranteed,chen2015fast,chen2018harnessing,ding2021simplicity}, 
Burer and Monteiro~\cite{burer2003nonlinear} reparametrize the PSD matrix variable $X$  as a low-rank matrix product $FF^\top$ for a factor $F\in \real^{n\times r}$. The Burer-Monteiro (BM) factorization reduces the number of variables from $O(n^2)$ to $O(nr)$, enabling scalable algorithms while preserving the expressive power of the underlying SDP problems for appropriately chosen ranks $r$. Indeed, the \emph{Augmented Lagrangian Method with Burer-Monteiro} (ALM-BM) framework applies BM reformulation on each ALM subproblem, and has become one of the most robust and powerful algorithms for solving large-scale low-rank SDPs~\cite{burer2003nonlinear,burer2005local,burer2006computational,wang2023decomposition,wang2023solving,monteiro2024low}.

\noindent\textbf{Challenges of ALM-BM framework: subproblem regularity and solvability.} However, despite the empirical success and growing interest in the ALM–BM framework, the understanding of its subproblems' structure, regularity, and solvability to global optimality, remains notably limited and leaves fundamental questions for the ALM-BM framework unanswered.
Specifically, although many SDPs admit low-rank solutions and may exhibit a strong form of regularities \cite{ding2021simplicity,ding2024sharpnesswellconditioningnonsmoothconvex,garber2025low}, it is unclear whether these structures carry over to the ALM subproblems: the presence of penalty terms and dual variables may distort the problem structure, and there is no guarantee of a low-rank solution or regularities across iterations. Understanding this behavior is critical for justifying low-rank parameterizations and the empirical strong performance of ALM in general:
\begin{center}
    \textit{Do the ALM subproblems inherit structural properties of the original SDP, such as regularity and the existence of low-rank solutions?}
\end{center}
\vspace{0.1cm}
Moreover, the nonconvexity introduced by the Burer-Monteiro factorization raises a second key question regarding the solvability (in terms of global optimality) of ALM-BM subproblems by first-order methods: 
\begin{center}
    \textit{Do simple algorithms such as gradient descent have a provably fast convergence to  {global optimal} solutions of nonconvex ALM-BM subproblems?}
\end{center}
\vspace{0.1cm}

\noindent \textbf{Our contribution.} In this work, we establish a set of theoretical results that offer new understandings of the structure of ALM subproblems and their solvability when combined with the Burer-Monteiro approach. We show that, under appropriate regularity conditions, the favorable low-rank and local growth geometry of an SDP is inherited uniformly by nearby ALM subproblems. This inherited structure, in turn, yields a uniform local linear convergence guarantee for standard gradient descent applied to the associated nonconvex ALM-BM subproblems, even under over-parameterization. Specifically, under the condition \emph{primal simplicity} (detailed in Section \ref{sec: sdp_p_ps}) for the original SDP problem that possesses low-rank solutions, our theoretical contributions are summarized as follows:
\begin{itemize}
    \item 
    \textbf{Inheritance of regularity and low-rankness:} Formalized in Theorem \ref{thm: ALmin_primal_simple_qg} in Section \ref{sec: analysis_AL-P}, 
    we demonstrate that the ALM subproblems inherit the desirable regularity properties of the original SDP, such as strong duality and strict complementarity. Additionally, these subproblems are shown to admit low-rank optimal solutions, mirroring the structure of the original problem, and also enjoy a quadratic growth condition. The growth condition and the preservation of regularity and low-rankness are established under the condition that the dual variable remains close to a regular dual optimal solution of the original SDP. 
   This result is not only fundamental to ALM-BM, as it partially explains the validity of BM for the subproblem, but also explains the empirical fast convergence of standard methods and opens the door for other low-rank methods to solve the subproblems \cite{ding2020spectral,allen2017linear,garber2025low}.
    \item 
    \textbf{Linear convergence of gradient descent for ALM-BM subproblems:}
    Formalized in Theorem \ref{cor: LCBM} in Section \ref{sec:gd}, we prove that the nonconvex ALM-BM subproblems can be solved to global optimality using the standard gradient descent method at a uniform linear rate, for \emph{any 
    parametrized rank  $r$ that is greater than or equal to the true solution rank}. 
    These strong convergence guarantees are contingent upon a local initialization of the GD algorithm and, as before, require the dual variable to be within a local region around a dual optimal solution. Our result underpins the practical viability of applying gradient-based methods to tackle these nonconvex reformulations.
    \item \textbf{Necessity of regularity and localization assumptions:} To underscore the tightness and necessity of our theoretical assumptions, we provide a comprehensive range of examples in Section \ref{sec: ALM-nonexample} and \ref{sec: app_ALM_sbprblm_non_exmpl}. These examples serve to illustrate scenarios where the conditions of our theorems are violated, leading to cases where the desirable properties of low-rankness or quadratic growth do not hold, or where the subproblems might exhibit unique high-rank solutions. Through these examples, we argue that the local nature and regularity of the aforementioned results concerning both ALM and ALM-BM subproblems are not merely a theoretical limitation, but rather an essential feature that reflects the inherent structure and behavior of these optimization problems. 
\end{itemize}

Motivated by our analysis of the regularity of augmented Lagrangian subproblems and the demonstrated effectiveness of first-order methods in the local regime, and inspired by the algorithm development based on ALM-BM \cite{burer2003nonlinear,burer2005local,wang2023decomposition,wang2023solving,monteiro2024low}, we design an algorithm and implement a GPU-accelerated low-rank SDP solver, ALORA, that explicitly leverages these favorable properties to achieve both scalability and reliability in practice:
\begin{itemize}
    \item \textbf{Theory inspired method -- ALORA:} In Section \ref{sec:alora}, we design and implement ALORA (Augmented Lagrangian Optimizer with Rank Adaptation) for solving low-rank SDPs.
    ALORA enhances the classical ALM–BM approach with two key innovations: (1) adaptive rank updates guided by \emph{multiple} spectral eigenvalues of the augmented Lagrangian gradient, and (2) a small auxiliary SDP subproblem to explore along directions of negative curvature. These mechanisms allow ALORA to dynamically adjust model complexity and improve global convergence in practice.
    
    \item \textbf{GPU-based implementation and numerical testing:} In the remaining part of Section \ref{sec:alora}, we present a practical, GPU-based implementation of ALORA and demonstrate its scalability through extensive numerical experiments. The solver exploits modern GPU architectures to accelerate core linear algebra operations. As a result, ALORA is capable of solving SDP problems with $n>10^7$ in hundreds of seconds.
\end{itemize}
Overall, we believe ALORA highlights the potential of the ALM-BM framework to scale further through the parallelism of modern GPU architectures. 

\textbf{Paper organization.} The remainder of this paper is organized as follows. 
In Section \ref{sec: rl}, we discuss the related work. In Section \ref{sec: preliminaries}, we first provide the mathematical formulations, \eqref{eq: sdp.p}, \eqref{eq: ALM}, and \eqref{eq: ALmin-BM} for SDP, ALM, and ALM-BM, respectively, and the necessary notations. We then introduce key concepts, including duality and a set of regularities called primal simplicity, for both the original SDP \eqref{eq: sdp.p} and its corresponding ALM subproblems \eqref{eq: ALmin}. Section \ref{sec: analysis_AL-P} then delves into the low-rankness, primal simplicity, and quadratic growth properties of these ALM subproblems, complemented by examples detailed in Section \ref{sec: ALM-nonexample} that demonstrate the necessity of assumptions. In Section \ref{sec:gd}, we prove the convergence of gradient descent on the Burer-Monteiro approach, establishing its linear convergence for ALM-BM subproblems \eqref{eq: ALmin-BM} under appropriate localization assumptions. We provide the proof for both the exact parametrized case ($r$ equals the solution rank) and the general case ($r$ is greater than or equal to the solution rank). Though the second case covers the first, the first one admits a simpler and more broadly applicable proof, so we retain it. Furthermore, in Section \ref{sec: app_ALM_sbprblm_non_exmpl}, we address the necessity of the localization assumption for linear convergence by providing examples where violations of these assumptions lead to the failure of the desired properties. Building upon these theoretical understandings, our proposed rank-adaptive augmented Lagrangian optimizer, ALORA, is introduced in Section \ref{sec:alora}. In the same section, we further discuss the GPU implementation of ALORA and present numerical experiments on problems such as MaxCut and matrix completion, showcasing its practical performance and scalability. Finally, we conclude this paper with a concise summary in Section \ref{sec: conlusion}.

\subsection{Related work}
\label{sec: rl}
We discuss related work on low-rank structure, Burer-Monteiro factorization, and methods that utilize the ALM-BM framework. Additional related work is in Appendix \ref{sec: add_rel_work}.

\textbf{Low-rank structure of SDP.}
A key structural property widely observed in practical semidefinite programs is the existence of low-rank optimal solutions. Empirical studies across applications, ranging from the famous Max Cut \cite{goemans1995improved} to matrix completion problems \cite{candes2012exact} and SDPs arising in engineering, signal processing, and statistics, consistently show that solutions to large-scale SDPs often exhibit low rank \cite{recht2010guaranteed,chen2018harnessing,ding2021simplicity}. On the theoretical side, a foundational result by Barvinok and Pataki~\cite{barvinok1995problems,pataki1998rank} established that any SDP with $m$ constraints admits an optimal solution with rank $\rstar$ satisfying $\frac{\rstar(\rstar+1)}{2}\leq m$.

\textbf{Burer-Monteiro factorization} Among the methods exploiting low-rank structures, a particularly influential approach is the Burer–Monteiro factorization~\cite{burer2003nonlinear,burer2005local}, which reparameterizes the semidefinite variable $X$ as $X=FF^\top$ for $F\in\mathbb {R} ^ {n\times r}$. While this introduces nonconvexity, it has been shown that, on the positive side, for sufficiently large rank estimate $r$ (above the Barvinok-Pataki bound) and other technical or generic assumptions, all local minima of the BM formulation of the SDP are globally optimal~\cite{burer2005local,journee2010low,boumal2016non}. However, on the negative side, there do exist examples showing that spurious local minima can exist when the rank is too small~\cite{waldspurger2020rank}. 
\textcolor{black}{Moreover, 
if some of the technical or generic assumptions fail, even if $r$ is on the order $n$, spurious local minima can still exist \cite{o2022burer}. For comprehensive discussions on the choice of $r$ and its relationship with the global optimality of local minimizers of the BM formulation of the SDP, we refer the reader to \cite[Section~1.2]{bhojanapalli2018smoothed}, \cite[pp.~2-3]{waldspurger2020rank}, and \cite[p.~2, Appendix~A]{o2022burer}. 
}

\textbf{ALM-BM for solving SDP.} A growing body of research has continued exploring the integration of the BM factorization with the augmented Lagrangian method for solving large-scale SDPs, since the seminal work~\cite{burer2003nonlinear}. In~\cite{burer2003nonlinear}, an augmented Lagrangian algorithm with a BM reformulation was proposed, showing promising empirical performance on large-scale instances. The subsequent work~\cite{burer2005local} advanced the theoretical foundation by analyzing the convergence properties of sequences generated by ALM-BM. Notably, it was shown that ALM can converge to globally optimal solutions and produce valid dual certificates under certain assumptions, despite the inherent nonconvexity of the problem. However, a key limitation of this analysis is its reliance on direct assumptions of ALM iterates, 
\textcolor{black}{such as the convergence of the linear feasibility in the original SDP and the first-order optimality of the ALM-BM formulation}, that are not guaranteed to hold,  or cannot be verified a priori, or cannot be enforced algorithmically. More recently, \cite{wang2023decomposition} extended the ALM-BM framework by addressing a broader class of semidefinite programs that involve nonlinear and nonsmooth objective functions, and a Riemannian semismooth Newton method is developed for solving the resulting ALM subproblems on a smooth manifold. The developed overall method in \cite{wang2023decomposition} has convergence to global optimality of the SDP under assumptions that might not be enforced easily algorithmically, e.g., \textcolor{black}{the nonconvex ALM-BM subproblems are solved to \emph{near global optimality}.} In~\cite{wang2023solving}, the ALM–BM subproblem was also formulated as a Riemannian optimization problem and solved using a Riemannian trust-region method. Global convergence was established under direct assumptions on the ALM iterates, similar to those in~\cite{burer2005local}, which are not guaranteed to hold or be verifiable a priori, e.g., the convergence of the iterate. Furthermore, when the SDP is equipped with a bounded trace constraint, a hybrid low-rank method was introduced that solves each ALM subproblem by alternating between an adaptive inexact proximal-point method for the BM formulation of the subproblem and Frank–Wolfe steps~\cite{monteiro2024low,aguirre2025cuhallar} for the original subproblem.
This hybrid strategy ensures convergence to global optimality (at a sublinear rate in theory) of each ALM subproblem while helping escape spurious stationary points that could arise in the nonconvex BM formulation. 

\section{Preliminaries: Basic setup, duality, and primal simplicity}
\label{sec: preliminaries}
In this section, we first describe the basic setup on SDP, ALM, ALM-BM, and the notations. We then present the duality concepts of \eqref{eq: sdp.p} and \eqref{eq: ALmin} and the primal simplicity, a set of regularity conditions, for these two problems. These concepts will be central to our main results in the next section. 

\subsection{Basic setup on SDP, ALM, ALM-BM, and notation}
\label{sec: Preliminary_SDP_AL_AL-BM}
In this section, we present the precise mathematical formulations of SDP, ALM, and ALM-BM and conclude with a paragraph of notations used in this paper. 

Denote $\symMat{n}$ the set of symmetric matrices in $\mathbb{R}^{n\times n}$.
The standard primal form of an SDP is given by:
\begin{equation}\label{eq: sdp.p}\tag{P}
    \begin{aligned}
        \min_{X\in\symMat{n}} &\quad  \langle C, X\rangle \\ 
        \text{s.t.} & \quad  \mathcal AX=b,\quad X\succeq 0.
    \end{aligned}
\end{equation}
where the decision variable is $X \in \symMat{n}$ and the problem data consist of a linear map $\mathcal A:\mathbb S^{n}\rightarrow \mathbb R^m$, a cost matrix $C\in \mathbb{S}^n$, and a right-hand-side vector $b\in \mathbb{R}^m$. 
The linear map $\Amap$ can be expressed explicitly using $m$ many matrices $A_i\in \symMat{n}$ by 
$[\Amap(X)]_i= \inprd{A_i}{X}$ for $i=1$, $\dots$, $m$ and any dummy variable $X\in \symMat{n}$. We equip $\mathbb{S}^n$ with the standard trace inner product and $\mathbb{R}^m$ with the standard dot product. The augmented Lagrangian of \eqref{eq: sdp.p} for any $\rho>0$ is a function $\mathcal{L}_\rho: \mathbb{S}^n\times \mathbb{R}^m \rightarrow \mathbb{R}$ with 
\begin{equation}\label{eq: AL}\tag{AL}
    \mathcal L_\rho (X,y) = \inprd{C}{X} + \inprd{y}{b- \mathcal{A}X} + \frac{\rho}{2} \twonorm{\mathcal A X - b}^2\ . 
\end{equation} Then, the ALM subproblems are
\begin{equation}\label{eq: ALmin}\tag{AL-P}
    \text{min}_{X\succeq 0} \quad \mathcal L_\rho (X,y) = \inprd{C}{X} + \inprd{y}{b- \mathcal{A}X} + \frac{\rho}{2} \twonorm{\mathcal A X - b}^2 \ . 
\end{equation}
With the above, ALM admits update rules as follows: 
\begin{equation}\label{eq: ALM}\tag{ALM}
    \begin{aligned}
        & X^{t+1}\leftarrow\argmin_{X\succeq0} \mathcal L_\rho (X,y^t)\\
        & y^{t+1}\leftarrow y^t+\rho(b-\mathcal AX^{t+1}) \ .
    \end{aligned}
\end{equation}

This iterative structure enables gradual improvement in both primal feasibility and dual optimality. 
\textcolor{black}{Note that the iterate $X^{t+1}$ does not need to be the exact subproblem solutions to ensure the convergence of ALM \cite{rockafellar1976augmented,liu2019nonergodic,xu2021iteration,liao2024inexact}.
}

The Burer-Monteiro approach factors $X\succeq 0$ as a low-rank product $X=FF^\top$, where $F\in\mathbb R^{n\times r}$ with rank $r$ much smaller than dimension $n$. The ALM-BM approach applies BM factorization to ALM subproblems, instead of iteratively solving convex ALM subproblems \eqref{eq: ALmin} in traditional ALM. Specifically, ALM-BM solves the following unconstrained but in general nonconvex subproblems:
\begin{equation}\label{eq: ALmin-BM}\tag{AL-BM}
    \text{min}_{F\in\mathbb R^{n\times r}} \quad \mathcal \Lf_\rho(F,y):\,=L_\rho (FF^\top,y) = \inprd{C}{FF^\top} + \inprd{y}{b- \mathcal{A}(FF^\top)} + \frac{\rho}{2} \twonorm{\mathcal A(FF^\top) - b}^2 \ . 
\end{equation}


\paragraph{Notation} 
We denote the optimal values of 
\eqref{eq: sdp.p} and \eqref{eq: sdp.d} as $\psol$ and $\dsol$ respectively. The dual slack map $Z: \mathbb{R}^m \rightarrow \symMat{n}$ is $Z(y) = C- \Amap^*(y)$.
For a matrix $A$, we denote its 
Frobenius norm, spectral norm, nuclear norm, largest singular value, and smallest nonzero singular value as $\fronorm{A}$, $\opnorm{A}$, $\nucnorm{A}$, $\sigma_{\max}(A)$, and $\sigma_{\min >0}(A)$, respectively. If $A$ is symmetric, we denote $\lambda_{\operatorname{min}}(A)$ the minimum eigenvalue of  $A$. 
We equip $\mathbb{R}^m$ and $\mathbb{S}^n$ with the standard dot product and trace inner product, respectively. 
For a linear map $\mathcal{B}:\symMat{d}\rightarrow \mathbb{R}^m$, we denote its smallest singular value as $
\sigma_{\min}(\mathcal{B}) = \min_{S\in \symMat{d},S\not =0} \frac{\twonorm{\mathcal{B}(S)}}{\fronorm{S}}$ and the operator norm as $\opnorm{\mathcal{B}} = \sigma_{\max}(\mathcal{B})= \sup_{X \in \symMat{d},X\not=0} \frac{\twonorm{\mathcal{B} X}}{\fronorm{X}}$. Similarly, the smallest and largest singular value adjoint map $\mathcal{B}^*$ are $
\sigma_{\min}(\mathcal{B}^*) = \min_{v\in \mathbb{R}^m,v\not =0} \frac{\fronorm{\mathcal{B}^*(v)}}{\twonorm{v}}=\sigma_{\min}(\mathcal{B})$ and 
$\sigma_{\max}(\mathcal{B}^*) = \max_{v\in \mathbb{R}^m,v\not =0} \frac{\fronorm{\mathcal{B}^*(v)}}{\twonorm{v}}=\sigma_{\max}(\mathcal{B})$, respectively.
Given a matrix $V\in \mathbb{R}^{d\times r}$ and a linear map $\mathcal{B}:\symMat{d}\rightarrow \mathbb{R}^m$, we define the restricted linear map of $\mathcal{B}$ with respect to $V$ as $\mathcal{B}_{V}: \symMat{r}\rightarrow \RR^{m}$ with 
$\mathcal{B}_{V}(S) = \mathcal{B}(VSV^\top).$ We shall frequently use the fact that for $X\succeq 0$, we have $\tr(X) = \nucnorm{X}$. We denote $I_k$ to be the identity matrix in $\real^{k\times k}$.  The notation $\mathbf{1}_s$ is a vector of length $s$ with its entries being all $1$. The notation  $0_{s\times t}$ is a matrix of size $s\times t$ with entries being all $0$. We also denote $0_{s}:= 0_{s\times 1}$. The operator $\diag: \mathbb{R}^n \rightarrow \mathbb{R}^{n\times n}$ puts the vector on the matrix diagonal. 
%
%


\subsection{Duality and primal simplicity of \eqref{eq: sdp.p}}\label{sec: sdp_p_ps}
In this subsection, we first introduce the dual problems of \eqref{eq: sdp.p} and the dual Slater's condition. We start with the dual problems of \eqref{eq: sdp.p}, and then introduce strong duality, strict complementarity, and the primal simplicity.

\paragraph{Dual of \eqref{eq: sdp.p}} The Fenchel dual problem of \eqref{eq: sdp.p} is the following:
\begin{equation}\label{eq: sdp.d}\tag{D}
    \begin{aligned}
        \max_y &\quad  \langle b, y\rangle \\ 
        \text{s.t.} & \quad  C- \mathcal A^* y \succeq 0 \ ,
    \end{aligned}
\end{equation}
where the variable $y\in \mathbb R ^m$ and the map $\mathcal A^*$ is the adjoint map of $\mathcal A$ (we equip $\mathbb{R}^n$ with the standard dot product). We denote the dual slack map $Z(y) = C-\Amap^*(y)$. 

Next, to ensure the validity of the duality framework, we introduce the dual Slater's condition. 
\paragraph{Dual Slater's condition} The dual Slater's condition states the following:
\begin{equation}\label{eq: DSlater's}\tag{D-Slater's}
\text{there exists a $y\in\mathbb{R}^m$ such that $Z(y)\succ 0$.}
\end{equation}
Note this covers the situation when $C\succ 0$ or $\Amap (X) = b \implies \tr(X)\leq \alpha$ for some $\alpha>0.$ Two common situations considered in many previous works \cite{ding2021optimal,ding2023revisiting,yurtsever2021scalable,monteiro2024low,helmberg2000spectral}. It is well-known that if the dual Slater's condition holds, then the primal solution to \eqref{eq: sdp.p} exists and the two problems \eqref{eq: sdp.p} and \eqref{eq: sdp.d} match in terms of their optimal values. 

Let us now define strong duality and strict complementarity. 
\begin{mydef}[Strong duality of \eqref{eq: sdp.p} and \eqref{eq: sdp.d}]
The problems \eqref{eq: sdp.p} and \eqref{eq: sdp.d} satisfy strong duality if they admit an optimal primal-dual pair and the following equality holds for any such pair $(\Xsol,\ysol)$:
\begin{equation}\label{eq: sdp.sd}\tag{PD-SD}
    \psol = \inprd{C}{\Xsol}  = \inprd{b}{\ysol}  = \dsol\ .
\end{equation}
We also say \eqref{eq: sdp.p} (or \eqref{eq: sdp.d} resp.) satisfies strong duality if  \eqref{eq: sdp.p} and \eqref{eq: sdp.d} satisfy strong duality.
\end{mydef}
We note that in this paper, the strong duality requires the existence of primal and dual optimal solutions rather than merely the optimal values matching. It is well known that primal and dual Slater's conditions can ensure the existence of primal and dual optimal solutions. 

To define strict complementarity, 
let us first consider the following complementarity, which is well-known to be equivalent to \eqref{eq: sdp.sd}:
\footnote{Indeed, using the linear feasibility $\Amap \Xsol = b$ in the following step $(a)$, we have  
\begin{equation}\label{eq: cs_sd}
0 = \inprd{C}{\Xsol}  - \inprd{b}{\ysol} \overset{(a)}{=}
\inprd{C}{\Xsol}  - \inprd{\Amap \Xsol}{\ysol} = 
\inprd{C}{\Xsol}  - \inprd{\Amap^*\ysol}{\Xsol} = 
\inprd{Z(\ysol)}{\Xsol}\ .
\end{equation}
The complementarity \eqref{eq: sdp.cs} follows from the above by considering that $\Xsol\succeq 0$ and $Z(\ysol)\succeq 0$. The reverse implication is also true by noting 
$
\inprd{C}{\Xsol}  - \inprd{b}{\ysol} = 
\inprd{Z(\ysol)}{\Xsol} =0
$ 
from \eqref{eq: sdp.cs}.}
\begin{equation} \label{eq: sdp.cs}
     Z(\ysol)\Xsol = 0\ . 
\end{equation}
Note that 
the complementarity condition \eqref{eq: sdp.cs} is equivalent to 
\begin{equation}\label{eq: sdp.cs.containment}
\range(\Xsol) \subset \nullspace(Z(\ysol)).
\end{equation}

Strict complementarity strengthens the subset relationship to be an equality relationship:
\begin{equation}\label{eq: sdp.sc.containment}
\range(\Xsol) = \nullspace(Z(\ysol))\ .
\end{equation}
Here, we introduced a version described in \cite{alizadeh1997complementarity}, which is equivalent to \eqref{eq: sdp.sc.containment} by the Rank-Nullity theorem.
\begin{mydef}[Strict complementarity]
The problems \eqref{eq: sdp.p} and \eqref{eq: sdp.d} satisfy strict complementarity if there is a pair of optimal solutions $(\Xsol,\ysol)$ that satisfies the following equality in addition to \eqref{eq: sdp.sd}: 
\begin{equation}\label{eq: sdp.sc}\tag{PD-SC}
    \rank(\Xsol) + \rank(Z(\ysol)) = n\ . 
\end{equation}
If a primal optimal $\Xsol$ (or a dual optimal $\ysol$ resp.) satisfies \eqref{eq: sdp.sc}, we say $\Xsol$ (or $\ysol$ resp.) satisfies strict complementarity or it is a strict complementary solution. We also say \eqref{eq: sdp.p} (or \eqref{eq: sdp.d} resp.) satisfies strict complementarity if  \eqref{eq: sdp.p} and \eqref{eq: sdp.d} satisfy strict complementarity.
\end{mydef}

Lastly, let us define the primal simplicity condition introduced in \cite{ding2021simplicity}. It requires strong duality, strict complementarity, and the additional condition that the primal solution is \emph{unique}. As shown in \cite{ding2021simplicity,ding2024sharpnesswellconditioningnonsmoothconvex}, this set of conditions holds in many low-rank SDP applications, including maxcut and matrix completion. 
\begin{mydef}[Primal simplicity]
    The problem \eqref{eq: sdp.p} is primal simple if the following three conditions hold:
    \begin{itemize}
    \item \eqref{eq: sdp.p} satisfies strong duality;
    \item \eqref{eq: sdp.p} satisfies strict complementarity;
    \item \eqref{eq: sdp.p} has a unique primal solution. We denote it as $\Xsol$.
    \end{itemize}
\end{mydef}
As described in \cite{ding2021simplicity}, primal simplicity ensures that \eqref{eq: sdp.p} is robust to optimization errors, i.e., termination error in iterative optimization methods for solving \eqref{eq: sdp.p} which necessarily terminated in finite steps, and measurement errors, error in the problem data $\Amap$, $b$, $C$ of \eqref{eq: sdp.p}. Primal simplicity is also vital in both algorithm design and analysis, see examples in \cite[Section 1]{ding2021simplicity}. In summary, primal simplicity can be considered a set of regularity conditions that ensure \eqref{eq: sdp.p} is well-behaved. The relationship between the primal simplicity and other regularities introduced in the literature is described in \cite[Section 1.1.1]{ding2021simplicity}, and we omit a detailed discussion here.

\subsection{Duality and primal simplicity of \eqref{eq: ALmin}}

In this subsection, we first introduce the dual of the subproblem \eqref{eq: ALmin}. We then define the strong duality, strict complementarity, and primal simplicity for \eqref{eq: ALmin}. Particularly, we show that the dual Slater's condition of the original SDP guarantees the unique solution and strong duality of \eqref{eq: ALmin}. 

Recall the augmented Lagrangian defined in \eqref{eq: AL} and the primal problem \eqref{eq: ALmin}. We introduce the dual problem below. 

\paragraph{Dual of \eqref{eq: ALmin}} The dual problem of the augmented Lagrangian problem \eqref{eq: ALmin} is 
\begin{equation}\label{eq: ALminD}\tag{AL-D}
    \max_{z\in \mathbb{R}^m} \quad \dyp(z) : = \inprd{b}{z} - \frac{1}{2\rho} \twonorm{z-y}^2 \quad \text{s.t.} \quad  \;C-\Amap^* z\succeq 0\ .
\end{equation}
Note that \eqref{eq: ALminD} is the proximal problem of \eqref{eq: sdp.d}. Hence, ALM on the primal side is simply the proximal point method for the problem \eqref{eq: sdp.d}, a critical observation used heavily in the literature \cite{rockafellar1976augmented,xu2021iteration,liu2019nonergodic}.

Let us now define strong duality of \eqref{eq: ALmin} and \eqref{eq: ALminD}.
\begin{mydef}[Strong duality and KKT of \eqref{eq: ALmin} and \eqref{eq: ALminD}]
The problems \eqref{eq: ALmin} and \eqref{eq: ALminD} satisfy strong duality if there is an optimal primal-dual pair, and for any  optimal primal-dual pair $(\Xyp,\zyp) \in \symMat{n}\times \RR^m$, the following KKT condition holds:
\begin{subequations}\label{eq: ALmin.KKT}
\begin{align}
  \Amap (\Xyp) =\frac{1}{\rho}(y - \zyp) +b & & (\text{first-order condition} ), \label{eq: ALmin.KKT.fo}\\
    Z(\zyp) \Xyp = 0  & & (\text{complementarity}), \label{eq: ALmin.KKT.cs}\\
    \Xyp \succeq 0, Z(\zyp) \succeq 0 & & (\text{primal-dual feasibility}) .\label{eq: ALmin.KKT.pdf}
\end{align}
\end{subequations}
We also say \eqref{eq: ALmin} (or \eqref{eq: ALminD} resp.) satisfies strong duality if  \eqref{eq: ALmin} and \eqref{eq: ALminD}  satisfy strong duality.
\end{mydef}

Thanks to the dual Slater's condition \eqref{eq: DSlater's}, strong duality in the above sense always holds for  \eqref{eq: ALmin} and \eqref{eq: ALminD} as shown by the following proposition. 
\begin{prop}\label{prop: existence_ALmin_P_D_sd}
Suppose the dual Slater's condition \eqref{eq: DSlater's} holds. Then strong duality holds for \eqref{eq: ALmin} and \eqref{eq: ALminD}, and the dual solution is unique, and they satisfy the KKT condition \eqref{eq: ALmin.KKT}.
\end{prop}
\begin{proof}
    Since the objective of \eqref{eq: ALminD} is continuous and admits strong concavity, we know \eqref{eq: ALminD} admits a unique solution $\zyp$. Furthermore, as Slater's condition \eqref{eq: DSlater's} holds for the dual problem \eqref{eq: ALminD} and the dual problem has an optimal solution, we know the strong duality holds for \eqref{eq: ALmin} and \eqref{eq: ALminD}, and the KKT condition holds,  due to standard Lagrangian duality theory.
\end{proof}

Next, we present two properties of the solutions to the \eqref{eq: ALmin} and  \eqref{eq: ALminD}: Nonexpansiveness of dual solution and the unique image under $\Amap$ for primal solutions. The former allows us to control the behavior of $z_{y,\rho}$, which is an implicit function of $y, \rho$, and the problem data. The latter result will be used in Section \ref{sec: ALM-nonexample} for properties of examples showing the necessity of the assumptions to be made. 

\begin{lem}[Nonexpansiveness of dual solution]\label{lem: zypyysol}
   Suppose \eqref{eq: ALminD} is feasible, then for any optimal solution $\ysol$ of \eqref{eq: sdp.d}, we have the following inequalities: 
   \begin{subequations}\label{eq: zypyysol}
       \begin{align}
           &\twonorm{\zyp - \ysol} \leq \twonorm{y -\ysol} \label{eq: zypysol}\\ 
           &\twonorm{\zyp - y} \leq \twonorm{y -\ysol}\label{eq: zypy} \ .
       \end{align}
   \end{subequations}
\end{lem}
\begin{proof}
From the strong concavity of $\dyp$ (defined in \eqref{eq: ALminD}) and the optimality of $\zyp$, it holds for any $z$ that:
\begin{equation}\label{eq: three_point_AL_dual}
    \dyp(\zyp) - \frac{1}{2\rho} \twonorm{z - \zyp}^2 \geq \dyp(z)\ .  
\end{equation}
Then the inequalities in \eqref{eq: zypyysol} follow from \eqref{eq: three_point_AL_dual} by setting $z = \ysol$ and the optimality of $\ysol$ in \eqref{eq: sdp.d}. 
\end{proof}


\begin{lem}[Uniqueness of image of primal solutions]\label{lem: AL_same_image}
Suppose $X_{y,\rho,1}$ and $X_{y,\rho,2}$ are two optimal solutions to \eqref{eq: ALmin} and strong duality holds for \eqref{eq: ALmin}, then it holds that $\Amap(X_{y,\rho,1}) = \Amap(X_{y,\rho,2})$.
\end{lem}
\begin{proof}
The result can be obtained by the KKT condition of \eqref{eq: ALmin.KKT.fo} and the uniqueness of $z_{y,\rho}$ by Proposition \ref{prop: existence_ALmin_P_D_sd}
\end{proof}

We introduce strict complementarity for \eqref{eq: ALmin} below.

\begin{mydef}[Strict complementarity of \eqref{eq: ALmin} and \eqref{eq: ALminD}]
The problems \eqref{eq: ALmin} and \eqref{eq: ALminD} satisfy strict complementarity if there is a pair of optimal solutions $(\Xyp,\zyp)$ satisfying the following equality in addition to \eqref{eq: ALmin.KKT}: 
\begin{equation}\label{eq: ALmin.sc}\tag{ALmin-SC}
    \rank(\Xyp) + \rank(Z(\zyp)) = n\ .
\end{equation}
Equivalently, 
\begin{equation}\label{eq: ALmin.sc.LA}
    \range(\Xyp) = \nullspace(Z(\zyp))\ .
\end{equation}
If a primal solution $\Xyp$ (or a dual solution $\zyp$ resp.) satisfies \eqref{eq: ALmin.sc}, we say $\Xyp$ (or $\zyp$ resp.) satisfies strict complementarity or it is a strict complementary solution. We also say \eqref{eq: ALmin} (or \eqref{eq: ALminD} resp.) satisfies strict complementarity if  \eqref{eq: ALmin} and \eqref{eq: ALminD} satisfy strict complementarity.
\end{mydef}

Finally, we introduce primal simplicity for \eqref{eq: ALmin}.  
\begin{mydef}[Primal simplicity of \eqref{eq: ALmin}]
    The problem \eqref{eq: ALmin} is primal simple if the following three conditions hold:
    \begin{itemize}
    \item \eqref{eq: ALmin}  satisfies strong duality;
    \item \eqref{eq: ALmin}  satisfies strict complementarity;
    \item \eqref{eq: ALmin}  has a unique primal solution. We denote it as $\Xyp$.
    \end{itemize}
\end{mydef}

As we shall see in the Section \ref{sec: analysis_AL-P}, if \eqref{eq: sdp.p} is primal simple, the dual Slater's condition holds, and the dual vector $y$ in \eqref{eq: ALmin} is near an optimal strict complementary $\ysol$, then \eqref{eq: ALmin} inherits the primal simplicity and other desired properties.

\section{Low-rankness, primal simplicity, and growth of \eqref{eq: ALmin}}\label{sec: analysis_AL-P}
This section details one of our primary contributions: the ALM subproblems \eqref{eq: ALmin} exhibit desirable properties, including low-rank solutions, primal simplicity, and a nonvanishing quadratic growth condition. These properties are shown to hold when the original problem \eqref{eq: sdp.p} is primal simple and admits a low-rank solution, and crucially, when the dual vector $y$ lies within a neighborhood of a strictly complementary dual optimal solution $\ysol$. Subsequently, in Section \ref{sec: ALM-nonexample}, we further provide examples illustrating that this localization assumption on the dual vector $y$ is indeed necessary. Without it, the subproblem \eqref{eq: ALmin} may admit a unique high-rank solution or fail to satisfy the quadratic growth property.

\paragraph{Necessary notions for main theorem} Before presenting our main results, we first introduce some necessary notations for an optimal strictly complementary pair $(\Xsol,\ysol)$ under the primal simplicity of \eqref{eq: sdp.p}. We define the rank of the optimal primal solution $\Xsol$ as:
\begin{equation}\label{eq: Xsol_Zysol_rank_r_star}
\rstar :\,= \rank(\Xsol) \overset{(a)}{=} \dim(\nullspace(Z(\ysol)))\ ,
\end{equation}
where the step $(a)$ is a direct consequence of strict complementarity \eqref{eq: sdp.sc}. Also, let $\Vstar\in \mathbb R^{n\times\rstar}$ be a matrix whose orthonormal columns span the null space of $Z(\ysol)$. Due to the strict complementarity of \eqref{eq: sdp.p}, we know that $\range(\Vstar)=\nullspace(Z(\ysol))=\range(\Xsol)$. Furthermore, for the unique solution $\zyp$ of \eqref{eq: ALminD}, let $\bar{V}\in \mathbb R^{n\times\rstar}$ denote the matrix whose orthonormal columns are the eigenvectors of $Z(\zyp)$ corresponding to its smallest $\rstar$ eigenvalues. In the following, we define $V = \bar{V} \bar{O}$ where $\bar{O} \in \arg\min_{OO^\top = I_{\rstar},O\in \mathbb{R}{\rstar \times \rstar}} \fronorm{\Vstar - VO}$. 

With these notions in place, we are now ready to state our main theorem, whose proof can be found in Section \ref{sec: proof_of_main_theorem}.

\begin{thm}\label{thm: ALmin_primal_simple_qg}
Consider the primal-dual SDP pair \eqref{eq: sdp.p} and \eqref{eq: sdp.d}. Suppose \eqref{eq: sdp.p} is primal simple and the dual Slater's condition \eqref{eq: DSlater's} holds.
Let $(\Xsol,\ysol)$ be an optimal strict complementary primal-dual solution pair to \eqref{eq: sdp.p} and \eqref{eq: sdp.d}. Then, for any $y$ with $\twonorm{y - \ysol}\leq c_{\star} := 
\min\{
\frac{\sigma_{\min>0}(\Xsol)}{\fronorm{\Xsol}}c_{\text{aux}} ,\, c_\rho\}$ where $c_{\text{aux}} =\frac{\sigma_{\min>0}(Z(\ysol))\sigma_{\min}(\mathcal{A}_{\Vstar})}{40\sqrt{2}\opnorm{\mathcal{A}}^2 }$ and 
$c_\rho =  \frac{\rho \sigma_{\min>0}(\Xsol)\sigma_{\min}(\mathcal{A}_{\Vstar})}{8}$, we have

(1) Primal simplicity and low-rank: \eqref{eq: ALmin} is primal simple, and the following equalities on rank hold for  any primal-dual optimal pair $(\Xyp, \zyp)$ for  \eqref{eq: ALmin} and \eqref{eq: ALminD}:
\begin{equation}\label{eq: XypZyp.rank}
\rank(\Xyp) = \rank(\Xsol)\ , \quad \text{and} \quad \rank(Z(\zyp)) = \rank(Z(\ysol)) \ .
\end{equation}

(2) Nondegraded quadratic growth: For any $B>0$, \eqref{eq: ALmin} admits a local quadratic growth inequality for some constant $\gamma>0$: for any $X\succeq 0$ with $\nucnorm{X}\leq B$, 
\begin{equation}\label{eq: ALmin.qg}
    \mathcal{L}_\rho(X,y) - \mathcal L_\rho(\Xyp,y) = 
 \inprd{Z(\zyp)}{X} + \frac{\rho}{2}\twonorm{\Amap X - \Amap \Xyp}^2
    \geq \gamma \fronorm{X - \Xyp}^2\ ,
\end{equation}
where $\gamma$ admits the following lower estimate: 
\begin{equation}\label{eq: ALmin.gamma_formula}
\gamma  \geq \min\left \{ \frac{\lambda_{n-\rstar}(Z(\ysol))}{B\left(10 + 20 \frac{\sigma_{\max}^2(\Amap)}{\sigma_{\min}^2(\Amap_{\Vstar})}\right)}, \frac{\rho\sigma_{\min}^2(\Amap_{\Vstar})}{10} \right\}>0\ .    
\end{equation}
\end{thm} 

Here are a few remarks on our main theorem:
\begin{rem}[Rank stability and \textcolor{black}{algorithmic} consequences]
    Our rank stability result \eqref{eq: XypZyp.rank} ensures that the Burer-Monteiro  \eqref{eq: ALmin-BM} with the parameter rank $r \ge \rstar$ is a correct approach as a method for solving  \eqref{eq: ALmin}, in the sense that the global minimizer of \eqref{eq: ALmin-BM} with $r=\rstar$ indeed corresponds to the minimizer of \eqref{eq: ALmin}. Therefore, as long as \eqref{eq: sdp.p} admits a low-rank solution, we expect $k$ to stay small when the dual vector $y$ is near optimal. In Remark \ref{rem: weaker_primal_rank_bound}, we present a bound of $\rank(\Xyp)$ under a much weaker condition. 
\end{rem}

\begin{rem}[Primal simplicity and \textcolor{black}{algorithmic} consequences]
The primal simplicity and quadratic growth of \eqref{eq: ALmin} explain partially the empirical success of methods based on ALM. Specifically, since \eqref{eq: ALmin} is primal simple and satisfies quadratic growth, many classical and modern algorithms applied to \eqref{eq: ALmin} have better theoretical guarantees than their worst-case behavior. Hence, their overall performances are enhanced. For example, IPM enjoys superlinear convergence under primal simplicity \cite{luo1998superlinear} rather than the worst-case linear convergence \cite{nesterov2018lectures}. As another example, the projected gradient and its restarted accelerated version for \eqref{eq: ALmin} converge sublinearly in general. However, with primal simplicity and quadratic growth, the method converges linearly \cite{necoara2019linear}.
\end{rem}

\begin{rem}[Nonvanishing quadratic growth and \textcolor{black}{algorithmic} consequences]
    The quadratic growth constant, as stated in \eqref{eq: ALmin.gamma_formula} of the theorem, is independent of $y$. Such independence is critical to the algorithmic analysis of the subproblem \eqref{eq: ALmin} and the overall ALM. For example, if this constant diminishes as $y$ approaches $\ysol$, then algorithms for solving \eqref{eq: ALmin} in general need more iterations to achieve the same accuracy, increasing the total computation of ALM. Indeed, as we shall see in Theorem \ref{cor: LCBM}, the independence ensures that gradient descent for solving the BM formulation \ref{eq: ALmin-BM} has a convergence speed guarantee independent of $y$. 

Moreover, the rank stability, the primal simplicity, and the nonvanishing quadratic growth ensure the validity of other low-rank methods developed in \cite{allen2017linear, ding2020spectral, garber2025low}, which are based on the Frank-Wolfe method and differ from the BM approach. 
\end{rem}

\subsection{Proof of Theorem \ref{thm: ALmin_primal_simple_qg}}
\label{sec: proof_of_main_theorem}
In this section, we present the proof of Theorem \ref{thm: ALmin_primal_simple_qg}.  
We first present an auxiliary lemma. We then present the main proof.




\paragraph{An auxiliary lemma} The following lemma summarizes a bunch of useful auxiliary results that will be instrumental in our subsequent analysis, particularly in estimating the rank of 
primal and dual solutions of \eqref{eq: ALmin}. We defer the proof to Section \ref{sec: pf_lem_auxiliary} in the appendix. 

\begin{lem}\label{lem: auxiliary}
Suppose the assumptions of Theorem \ref{thm: ALmin_primal_simple_qg} hold. Then, if $\twonorm{z_{y,\rho}-\ysol}\leq c_{\text{aux}}: =\frac{\sigma_{\min>0}(Z(\ysol))\sigma_{\min}(\mathcal{A}_{\Vstar})}{40\sqrt{2}\opnorm{\mathcal{A}}^2 }$, 

(1) The following control of the smallest nonzero eigenvalue value of $Z(z_{y,\rho})$ and the relationship of $V$, any $X_{y,\rho}$, and $Z(z_{y,\rho})$ hold:
\begin{align} 
\sigma_{n-\rstar}(Z(\zyp))\geq \frac{9}{10} \sigma_{\min>0}(Z(\ysol)) >0\label{eq: Zmin_nonzero_better_bound}
\\ 
\range(\Xyp)\subset \nullspace(Z(\zyp))\subset \range(V) .\label{eq: null_space_V_range_Xyp_range}
\end{align}

(2)The following bounds of the difference of $\Amap_V$ and $\Amap_{\Vstar}$ and their smallest singular values hold:
\begin{subequations}
\label{eq: A_V_y_zrho}
\begin{align} 
\opnorm{\mathcal{A}_{\Vstar} - \mathcal{A}_V} & \leq \frac{4\sqrt{2}\opnorm{\Amap}^2\twonorm{\ysol -\zyp} }{\sigma_{\min>0}(Z(\ysol))} 
\label{eq: A_Vmin_A_Vstar_diff}\\ 
\sigma_{\min}(\Amap_{V})& \geq \frac{9}{10}\sigma_{\min}(\Amap_{\Vstar})>0. 
\label{eq: A_Vmin}
\end{align}
\end{subequations}
In particular, the map $\Amap_V$ is injective.
\end{lem}


\begin{proof}[Proof of Theorem \ref{thm: ALmin_primal_simple_qg}]
Thanks to Proposition \ref{prop: existence_ALmin_P_D_sd}, strong duality and the existence of primal and dual optimal solutions $\Xyp$ and $\zyp$ for \eqref{eq: ALmin} hold. Also, because of Lemma \ref{lem: zypyysol}, the dual nonexpansiveness, we know $\twonorm{z_{y,\rho}-\ysol}\leq \twonorm{y-\ysol}$. Thus, Lemma \ref{lem: auxiliary} is activated, as we assume that $\twonorm{y-\ysol}\leq c_\star$ in Theorem  \ref{thm: ALmin_primal_simple_qg}. Below, we first show that the primal solution $X_{y,\rho}$ is unique. We then show strict complementarity condition holds. Lastly, we prove the quadratic growth result. 


\underline{\textit{Uniqueness of $\Xyp$.}} Using \eqref{eq: null_space_V_range_Xyp_range} in Lemma \ref{lem: auxiliary}, we have that any primal solution to \eqref{eq: ALmin} satisfies $\Xyp = VS_{y,\rho}V^\top$ for some $S_{y,\rho}\in \symMat{\rstar}$. 
Next, consider the following linear equation in terms of $S$ (Recall that $\mathcal{A}_V: \mathbb{S}^{\rstar} \rightarrow\mathbb{R}^m$ with $\mathcal{A}_V(S) = \mathcal{A}(VSV^\top)$): 
\begin{equation}\label{eq: linear_pert_KKT_ALmin}
    \mathcal{A}_V(S) = b+\frac{1}{\rho}(y-z_{y,\rho}).
\end{equation}
Using the KKT condition \eqref{eq: ALmin.KKT.fo}, we know the matrix $S_{y,\rho}$ satisfies the equation \eqref{eq: linear_pert_KKT_ALmin}. Since $\sigma_{\min}(\Amap_V)>0$ as Lemma \ref{lem: auxiliary} is activated, we know $S_{y,\rho}$ is unique. Consequently, the primal solution $\Xyp$ is also unique. 

\underline{Strict complementarity.}
We claim that $S_{y,\rho} \in \symMat{\rstar}$ is positive definite. Under the claim,  
the relationship \eqref{eq: null_space_V_range_Xyp_range} in Lemma \ref{lem: auxiliary} becomes equality as $\dim(\range(V)) = \rstar$ and 
$\Xyp = VS_{y,\rho}V^\top$. Hence, the desired conclusion \eqref{eq: XypZyp.rank} follows by noting $\dim(\range(V)) = \rstar = n-\rank(Z(\ysol))$.

To show $S_{y,\rho}\succ 0$, we consider the following linear equations in terms of $S$:
\begin{align}
 \mathcal{A}_{\Vstar}(S)  & = b 
 \label{eq: A_V_star_b}.
 \end{align}
Under primal simplicity, we know that \eqref{eq: A_V_star_b} has a \emph{unique} solution $S_\star =\Vstar^\top \Xsol \Vstar \in \mathbb{S}^{\rstar}$ thanks to Lemma \ref{lem: kkt_like.qg} (proved in Appendix \ref{sec:lem} which asserts $\mathcal{A}_{\Vstar}$ is injective). 
Note that \eqref{eq: linear_pert_KKT_ALmin} is basically a perturbation version of \eqref{eq: A_V_star_b}. Let $\Delta b = 
b+\frac{1}{\rho}(y-z_{y,\rho})-b = \frac{1}{\rho}(y-z_{y,\rho})$ and $\Delta \mathcal{A} =\mathcal{A}_{\Vstar}-\mathcal{A}_{V}$. Standard linear algebra reveals that 
$S_{y,\rho} - S_\star = \Amap_V^\dagger (\Delta \Amap S_\star +\Delta b$, where $\Amap_V^\dagger $ is the pseudo-inverse of $\Amap_V$. Hence, we have
\begin{equation}\label{eq: SS_star}
    \fronorm{S_{y,\rho}-S_\star} \leq 
    \frac{\opnorm{\Delta \mathcal{A}}\fronorm{S_\star}+ \twonorm{\Delta b}}{\sigma_{\min}(\mathcal{A}_{V})} \overset{(a)} {\leq} 
\frac{\frac{4\sqrt{2}\opnorm{\Amap}^2\fronorm{S_\star}}{\sigma_{\min>0}(Z(\ysol))}\twonorm{\ysol-\zyp}+ \frac{1}{\rho} \twonorm{y-\zyp}}{\frac{9}{10}\sigma_{\min}(\mathcal{A}_{\Vstar})}    
    \overset{(b)}{\leq} \frac{1}{4} \sigma_{\min>0}(\Xsol). 
\end{equation}
Here, the step $(a)$ is due to \eqref{eq: A_V_y_zrho} of Lemma \ref{lem: auxiliary} and the step $(b)$ is due to Lemma \ref{lem: zypyysol} and  $\twonorm{y-\ysol}\leq c_\star$. 
Thus, due to Weyl's inequality and $\lambda_{\rstar}(S_\star)= \sigma_{\min>0}(\Xsol)$, we know all eigenvalues of $S_{y,\rho}$ are positive and strict complementarity is proved. 



\underline{\textit{Quadratic growth of \eqref{eq: ALmin}.}} 
Let us simplify the difference $\mathcal{L}_\rho (X,y) - \mathcal{L}_\rho(\Xyp,y)$: 
\begin{equation}
\begin{aligned}\label{eq: Al_diff}
\mathcal{L}_\rho (X,y) - \mathcal{L}_\rho(\Xyp,y) 
& = \inprd{C}{X-\Xyp} - \inprd{y}{\Amap( X-\Xyp)} + 
\frac{\rho}{2}\left(\twonorm{\Amap X-b}^2 - \twonorm{\Amap \Xyp -b}^2\right) \\ 
& \overset{(a)}{= } 
 \inprd{C-\Amap^* (y + \rho(b - \Amap \Xyp))}{X- \Xyp} + \frac{\rho}{2}\twonorm{\Amap X - \Amap \Xyp}^2\\
 &  \overset{(b)}{= } 
 \inprd{Z(\zyp)}{X} + \frac{\rho}{2}\twonorm{\Amap X - \Amap \Xyp}^2 \ ,
\end{aligned}
\end{equation}
where the step $(a)$ is a consequence of direct calculation and the step $(b)$ is due to the KKT condition $\zyp = y+ \rho(b-\Amap \Xyp)$ and $Z(\zyp)\Xyp =0$ in \eqref{eq: ALmin.KKT} for \eqref{eq: ALmin}.
Thus, we see that $X$ is optimal to \eqref{eq: ALmin} if and only if it satisfies the following system:
\begin{equation}\label{eq: ZypXAxXsucceq0}
\inprd{Z(\zyp)}{X}=0,\quad \Amap X = \Amap \Xyp,\quad \text{and}\quad X\succeq 0.
\end{equation}
Recall we have just shown that $\Xyp$ is a unique optimal solution. Hence, it is the unique solution to the system \eqref{eq: ZypXAxXsucceq0}. Moreover, since  $\rank(\Xyp) = \rank(\Xsol)$, we have 
$\range(V) = \range(\Xyp)$ due to \eqref{eq: null_space_V_range_Xyp_range} of Lemma \ref{lem: auxiliary}. Thus we can activate Lemma \ref{lem: kkt_like.qg} in Appendix \ref{sec:lem} and have that for any $X\succeq 0$, the following holds 
\begin{equation}
\begin{aligned}\label{eq: XXyp.kkt.qg}
    \fronorm{X-\Xyp}^{2}\;
    &\leq\tr(X)
    \left(4+8\frac{\sigma_{\max}^2(\Amap)}{\sigma_{\min}^2(\Amap_{V})}\right)\frac{\inprd{Z(\zyp)}{X}}{\lambda_{n-\rstar}(Z(\zyp))}
		+\frac{4}{\sigma_{\min}^{2}(\Amap_{V})}\twonorm{\Amap(X)-\Amap (\Xyp)}^{2} \\ 
    &\overset{(a)}{\leq} 
    \tr(X)
    \left(10+20\frac{\sigma_{\max}^2(\Amap)}{\sigma_{\min}^2(\Amap_{\Vstar})}\right)\frac{\inprd{Z(\zyp)}{X}}{\lambda_{n-\rstar}(Z(\ysol))}
		+\frac{5}{\sigma_{\min}^{2}(\Amap_{\Vstar})}\twonorm{\Amap(X)-\Amap (\Xyp)}^{2}.
\end{aligned}
\end{equation}
Here, in the step $(a)$, we use \eqref{eq: Zmin_nonzero_better_bound} and \eqref{eq: A_Vmin} in Lemma \ref{lem: auxiliary} for $\sigma_{\min>0}(Z(\zyp))$ and $\sigma_{\min}(\Amap_{\Vstar})$.
Combining \eqref{eq: Al_diff} and \eqref{eq: XXyp.kkt.qg}, we see the quadratic growth inequality \eqref{eq: ALmin.qg} holds for \eqref{eq: ALmin}.
\end{proof}

\begin{rem}[Review of the proof technique]
Theorem \ref{thm: ALmin_primal_simple_qg} states a stability result of the primal simplicity and quadratic growth of \eqref{eq: sdp.p} under primal simplicity of \eqref{eq: sdp.p}. However, we note that, since we work only with primal simplicity, prior results showing stability do not apply. 

Indeed, typical past results \cite{nayakkankuppam1999conditioning,chan2008constraint} consider a KKT-based equation $F(X,y,Z)=0$ for $(X,y,Z)$, the primal variable, the dual vector, and the dual slack variable, where $F$ can be smooth or nonsmooth. Then, many desired results are natural consequences from a version of inverse function theorem of $F$ at $(\Xsol,\ysol,Z(\ysol))$. Thus, these works study the invertibility of the (generalized) Jacobian of $F$. In \cite{nayakkankuppam1999conditioning}, the invertibility is equivalent to primal simplicity and dual uniqueness. In \cite{chan2008constraint}, the invertibility is equivalent to the so-called primal nondegeneracy and dual nondegeneracy. These conditions either strictly imply primal simplicity or are not equivalent to primal simplicity \cite[Section 1.1.1]{ding2021simplicity}. Thus, under primal simplicity alone, those past routes are blocked due to potential failure of invertibility. Henceforth, we must explore new paths under primal simplicity. 

Our proof for primal simplicity of \eqref{eq: ALmin}, on a high level, uses the particular properties of  \eqref{eq: ALmin} such as dual nonexpansiveness and the KKT condition of it. Specifically, our strategy is as follows: (1) we consider the KKT \eqref{eq: ALmin.KKT} but treating $z_{y,\rho}$ as a parameter. This greatly simplifies our analysis as the resulting condition is either linear in $X$ or just the PSD constraint on $X$; (2) we compare the linear equalities of the KKT of \eqref{eq: sdp.p} and \eqref{eq: ALmin}, resulting in \eqref{eq: linear_pert_KKT_ALmin} and \eqref{eq: A_V_star_b}. Controlling the perturbation would then allow us to show uniqueness and strict complementarity;  (3) however, directly controlling the perturbation is difficult as it depends on the parameter $z_{y,\rho}$ which depends on $y$ and problem data in an implicit way. Fortunately we observe that Lemma \ref{lem: zypyysol}, the dual nonexpansiveness, can be used to reduce the two relevant quantities, $\twonorm{z_{y,\rho}-y}$ and $\twonorm{z_{y,\rho}-\ysol}$, to $\twonorm{y-\ysol}$. 

To show nonvanishing quadratic growth, we utilize the (i) nonvanishing primal simplicity established: $\sigma_{\min}(\Amap_{V})$ and $\sigma_{\min>0}(Z(\zyp))$ are bounded away from $0$, and (ii) extend the argument in \cite{ding2020spectral}. The extension requires additional technical work as the optimization problem in \cite{ding2020spectral} assumes that all feasible points have the same trace which does not hold for \eqref{eq: ALmin}.

\end{rem}

Lastly, we record a quick consequence of the proof of Theorem \ref{thm: ALmin_primal_simple_qg} that is useful in the subsequent sections. A detailed proof can be found in Appendix \ref{sec: pf_lem_auxiliary}.
\begin{lem}\label{lem: X_yrho_X_star_diff_dual_diff}
Suppose the assumptions of Theorem \ref{thm: ALmin_primal_simple_qg} hold. Then, if $\twonorm{y-\ysol}\leq \min\{c_{\text{aux}}, c_\rho\}$ where $ c_{\text{aux}}=\frac{\sigma_{\min>0}(Z(\ysol))\sigma_{\min}(\mathcal{A}_{\Vstar})}{40\sqrt{2}\opnorm{\mathcal{A}}^2 }$ and 
$c_\rho=\frac{\rho \sigma_{\min>0}(\Xsol)\sigma_{\min}(\mathcal{A}_{\Vstar})}{8}$
, we have 
\begin{equation}\label{eq: Xyp_Xsol_dual_diff_bound}
    \fronorm{\Xyp - \Xsol}\leq c_{\text{pd}} \twonorm{y-\ysol}.
\end{equation}
where $ c_{\text{pd}} = \frac{1}{\sigma_{\min}(\mathcal{A}_{\Vstar})}\left(\frac{16\sqrt{2}\opnorm{\Amap}^2\fronorm{X_\star}}{\sigma_{\min>0}(Z(\ysol))} + \frac{2}{\rho} \right)$.
\end{lem}

\subsection{Sensitivity of Assumptions in Theorem \ref{thm: ALmin_primal_simple_qg}}
\label{sec: ALM-nonexample}

As demonstrated in Theorem \ref{thm: ALmin_primal_simple_qg}, the desirable properties of \eqref{eq: ALmin} (namely, the existence of a low-rank strictly complementary optimal solution and satisfaction of a quadratic growth condition with a non-vanishing constant) critically depend on the dual vector $y$ being sufficiently close to an optimal strictly complementary solution $\ysol$ of \eqref{eq: sdp.d}. In this section, we emphasize the necessity of this localization assumption by presenting examples where its violation leads to the failure of these desired properties. For primal simplicity, in Appendix \ref{sec: failure_ps}, we present two examples confirming the necessity of primal simplicity to the results in Theorem \ref{thm: ALmin_primal_simple_qg}, and that they cannot be replaced by some common regularities, e.g., the primal and dual nondegeneracy defined in \cite[Definition 5 and 8]{alizadeh1997complementarity}. \footnote{We defer the examples to the appendix since the primal simplicity holds in a wide range of SDP applications with low-rank solutions as shown in \cite{ding2021simplicity,ding2024sharpnesswellconditioningnonsmoothconvex}.} 

Before presenting these examples, we first establish a general bound on the solution rank. The following proposition ensures that the augmented Lagrangian problem always admits a solution with a rank no more than $\sqrt{2m}$, provided the dual Slater's condition \eqref{eq: DSlater's} holds.

\begin{prop}[Barvinok-Pataki bound of \eqref{eq: ALmin}]
    Consider the primal-dual SDP pair \eqref{eq: sdp.p} and \eqref{eq: sdp.d}. Suppose the dual Slater's condition \eqref{eq: DSlater's} holds. Then it holds for any $y$ in \eqref{eq: ALmin} that there is an optimal solution $X_{y,\rho}$ such that its rank $r_{y,\rho} = \rank(X_{y,\rho})$ satisfies the Barvinok-Pataki bound: $\frac{r_{y,\rho} (r_{y,\rho}+1)}{2} \leq m$. 
\end{prop}
\begin{proof}
    By Proposition \ref{prop: existence_ALmin_P_D_sd}, we know the KKT condition \eqref{eq: ALmin.KKT} holds for any optimal primal-dual pair ($X_{y,\rho},z_{y,\rho})$, and such a pair does exist. In particular, any point $X_{y,\rho}$ is optimal to \eqref{eq: ALmin} if and only if it is optimal to the following problem: 
    \begin{equation}\label{eq: reformulatedALmin}
    \begin{aligned}
        \text{minimize} &\quad  \langle Z(z_{y,\rho}), X\rangle \\ 
        \text{subject to} & \quad  \mathcal AX=b+ \frac{1}{\rho}(y-z_{y,\rho})\ ,\quad X\succeq 0\ .
    \end{aligned}
\end{equation}
Since \eqref{eq: reformulatedALmin} has $m$ many constraints and admits a primal optimal solution, from \cite{barvinok1995problems,pataki1998rank} and more concretely \cite[Theorem 2.1]{lemon2016low}, we know there is an optimal solution to \eqref{eq: reformulatedALmin}, which is also optimal to \eqref{eq: ALmin} and has rank $\frac{r_{y,\rho} (r_{y,\rho}+1)}{2} \leq m$.
\end{proof}
With this Barvinok-Pataki bound in mind, we now present a primal simple SDP \eqref{eq: sdp.p} that possesses a rank-one solution. However, for a specific feasible dual vector $y$, the augmented Lagrangian subproblem \eqref{eq: ALmin} will be shown to yield a unique primal solution whose rank precisely matches this Barvinok-Pataki bound.

\begin{prop}[Necessity of localization for low-rankness]\label{exam: high_rank_ALM_nonlocal_y}
For any $m\geq 3$ and $n\geq 2$ with $m\leq \frac{n(n+1)}{2}$, there exist $C$, $\Amap$, $b$, such that \eqref{eq: sdp.p} is primal simple with a rank one optimal solution. Moreover, the corresponding \eqref{eq: ALmin} with $\rho=1$ and a feasible $y$ admits a \emph{unique} solution $X_{y,\rho}$ whose rank $r$ satisfies $r = \max\left\{l\mid \frac{l(l+1)}{2} \leq m\right\}$. 
\end{prop}

\begin{proof}
  Let us first consider the following SDP for the simple case: $n=2$ and $m=3$.
\begin{equation}\label{eq: nonlocaly_SDP_simple}
    \begin{aligned}
        \text{minimize} &\quad 0\\ 
        \text{subject to} & \quad  X_{11} =1, \quad X_{12} = 0, \quad  X_{11}+X_{22} = 1\\ 
        & \quad X\succeq 0\ .
    \end{aligned}
\end{equation}
For \eqref{eq: nonlocaly_SDP_simple}, it can be verified that $\ysol = [\frac{1}{2},0,-\frac{1}{2}]^\top$ and  $\Xsol = \begin{bmatrix}
    1 & 0 \\ 
    0 & 0
\end{bmatrix}$ are dual and primal optimal. Strict complementarity can be verified for $Z(\ysol)$ and $\Xsol$. Uniqueness for $\Xsol$ follows directly from the structure of the constraints. Moreover,  for \eqref{eq: ALmin} with  $\rho=1$ and a feasible $y = [-\frac{3}{4}, 0, -\frac{1}{2}]^\top$ (as $Z(y) = \begin{bmatrix}
    \frac{5}{4} & 0 \\ 
    0 & \frac{1}{2}
\end{bmatrix}\succeq 0$), we see  $X_{y,\rho} = \begin{bmatrix}
    \frac{1}{4} & 0 \\
    0 & \frac{1}{4}
\end{bmatrix}$ is a unique optimal solution to \eqref{eq: ALmin} and has rank $2$ matching the Barvinok-Pataki bound. The uniqueness of $X_{y,\rho}$ follows from Lemma \ref{lem: AL_same_image}.

Generalizing the previous example to arbitrary $n>0$ and $m=\frac{r(r+1)}{2}$ for any integer $r\in [1,n]$ (We shall deal with the general case $m$ later.), we set \eqref{eq: sdp.p} with the following choice $C \in \symMat{n}$, $\Amap:\symMat{n}\rightarrow \RR^{\frac{r(r+1)}{2}}$,\footnote{Here, we index the set $\{1,\dots, \frac{r(r+1)}{2}\}$ using two indices $i$ and $j$ with $1\leq i \leq j\leq r$. For example, $(1,1)$ corresponds to $1$ and $(1,2)$ corresponds to $2$, while $(r,r)$ corresponds to $\frac{r(r+1)}{2}$. In general, the index $(i,j)$ corresponds to $\frac{(2r-i+2)(i-1)}{2} + j-i+1$.} and $b\in \RR^{\frac{r(r+1)}{2}}$:
\begin{equation}\label{eq: choice_CAb_nonlocal_y_high_rank_solution}
\begin{aligned}
    C      &= \diag(0_r,\mathbf{1}_{n-r})\ , 
    &A_{ij}  = \frac{e_i e_j^\top + e_j e_i^\top }{2}\ , \; 1\leq i<j\leq r\ , \\
    A_{ii} &= \begin{cases}
        e_ie_i^\top , &1\leq i <r\\ 
        \diag(\mathbf{1}_r,0_{n-r}), &i = r
    \end{cases}\ ,
    &b_{ij} = 
    \begin{cases}
        1 , & i=j=1 \;\text{or}\;i=j=r \\ 
        0, &\text{otherwise}
    \end{cases} \ .
\end{aligned}    
\end{equation}
For this choice of $C$, $\Amap$, and $b$, it is easy to verify that $\Xsol =e_1e_1^\top$ is the unique primal solution, and the following is a dual optimal solution:  
\begin{equation}\label{eq: optimal_y_nonlocal_SDP}
[\ysol]_{ij} =\begin{cases}
\frac{1}{2} & i=j= 1 \\
-\frac{1}{2} & i = j=r\\
0 & \text{otherwise}
\end{cases}\ .
\end{equation}
Moreover, the dual slack $Z(\ysol)=\mathbf{diag}(0,\frac{1}{2}\mathbf{1}_{n-(r-1)},\mathbf{1}_{n-r})$ is strictly complementary to $e_1e_1^\top$. 

Set $\rho$ and the dual variable $y \in \RR^{r(r+1)/2}$  in \eqref{eq: AL} as the following:
\begin{equation}\label{eq: y_nonlocal_SDP}
\rho =1,\quad     [y]_{ij} = \begin{cases}
       \frac{1}{2r}-1   , & i=j=1\\
       \frac{1}{2r} , & 1<i =j<r\\
       -\frac{1}{2}, & i=j=r\\
        0, & 1\leq i \not=j \leq r
    \end{cases} \ .
\end{equation}
This choice of $y$ is feasible with respect to \eqref{eq: ALminD} as $Z(y) = \mathbf{diag}(\frac{3}{2}-\frac{1}{2r}, \frac{1}{2}-\frac{1}{2r}, \dots, \frac{1}{2}, \mathbf{1}_{n-r})\succeq 0$. Moreover, the rank $r$ matrix $X_{y,\rho} = \mathbf{diag}(\frac{1}{2r} \mathbf{1}_{r}, 0_{n-r})$ is an optimal solution to \eqref{eq: ALmin} by verifying the KKT condition \eqref{eq: ALmin.KKT}. Using the KKT condition, one can show the dual optimal $z_{y,\rho}$ gives a dual slack $Z(z_{y,\rho})=\diag(0,\mathbf{1}_{n-r})$ and verify the strict complementarity. The strict complementarity enforces any primal optimal solution to \eqref{eq: ALmin} to be supported only on the left upper $r\times r$ block. By Lemma \ref{lem: AL_same_image} and the structure of $A_{ij}$, we see  $X_{y,\rho}=\mathbf{diag}(\frac{1}{2r} \mathbf{1}_{r}, 0_{n-r})$ is also unique.

For general $m\geq 3$, we pick $r=\max\{l\mid \frac{l(l+1)}{2}\leq m\}$. We consider \eqref{eq: sdp.p} with the choice of $C$, $\Amap$, and $b$ in \eqref{eq: choice_CAb_nonlocal_y_high_rank_solution}. Next, we add  $m-\frac{r(r+1)}{2}$ many redundant constraints, $\tr(0 \cdot X)=0$, to \eqref{eq: sdp.p}. By setting the additional dual variables in $\ysol$ and $y$ to $0$ and keeping the other entries as those in \eqref{eq: optimal_y_nonlocal_SDP} and \eqref{eq: y_nonlocal_SDP}, we see that the constructed \eqref{eq: sdp.p} is still primal simple with optimal solution $e_1e_1^\top$ and the corresponding \eqref{eq: ALmin} with the choice of $y$ also admits a unique rank $r$ optimal solution $X_{y,\rho} = \mathbf{diag}(\frac{1}{2r} \mathbf{1}_{r}, 0_{n-r})$.
\end{proof}

Next, we show the Barvinok-Pataki bound can still be matched if we only require the dual vector to be close to an arbitrary optimal dual solution for \eqref{eq: sdp.p}, rather than a strict complementary one.

\begin{prop}[Necessity of localization to a strict complementarity solution for low-rankness]\label{exam: high_rank_ALM_local_nonstrict_y}
    For any $m\geq 3$ and $n\geq 2$ with $m\leq \frac{n(n+1)}{2}$, there is some $C$, $\Amap$, $b$, such that \eqref{eq: sdp.p} is primal simple with a rank one optimal solution $\Xsol$. Then, for any $\epsilon>0$, there exists an optimal dual solution and a local $y$ with $\|y-\ysol\|_2\leq \epsilon$ such that the corresponding \eqref{eq: ALmin} with $\rho=1$ admits a \emph{unique} solution $X_{y,\rho}$ whose rank $r$ satisfies $r = \max\left\{l\mid \frac{l(l+1)}{2} \leq m\right\}$.
\end{prop}

\begin{proof}
Consider the same \eqref{eq: sdp.p} given in the Proposition \ref{exam: high_rank_ALM_nonlocal_y}. We shall consider the simpler setting that $n\geq 2$ and $m = \frac{r(r+1)}{2}$ for some integer $r\in [1,n]$. The general case follows the same way as the previous example. We construct a curve $y_\epsilon$ parameterized by $\epsilon>0$. This curve approaches an optimal but not a strict complementary dual solution for \eqref{eq: sdp.p} as $\epsilon\rightarrow 0$.  We show that the optimal solution to \eqref{eq: ALmin} with $y_\epsilon$ still has rank matching the Barvinok-Pataki bound for each small $\epsilon>0$. 

We set $\rho=1$. For any $\epsilon \in (0,\frac{1}{3(r-1)})$, we set $y_\epsilon$ in \eqref{eq: ALmin} as the following:
\begin{equation}\label{eq: y_value_close_to_optimal_non_strict_comp_higher_rank}
    y_{\epsilon,ij} =\begin{cases}
        -2(r-1)\epsilon, & i=j=1 \\
        \epsilon, & i=j, \;i\in[2,r-1]\\
        -(r-1)\epsilon, & i=j=r\\
        0, &\text{otherwise}
    \end{cases}\ . 
\end{equation}
The dual slack $Z(y_\epsilon) $ is 
$
Z(y_\epsilon) = 
\diag(3(r-1)\epsilon, (r-2)\epsilon\mathbf{1}_{r-2}, (r-1)\epsilon, \mathbf{1}_{n-r}) \succeq 0$, which is also dual feasible. 
The primal optimal solution to \eqref{eq: ALmin} with $y_\epsilon$ and $\rho=1$ is given by 
\begin{equation}\label{eq: X_value_close_to_optimal_non_strict_comp_higher_rank} 
    [X_{\rho,y_\epsilon}]_{ij} = \begin{cases}
        1-2(r-1)\epsilon & i=j=1\\
        \epsilon, &  i=j, \;i\in[2,r]\\
        0, & \text{otherwise}
    \end{cases} \ .
\end{equation}
 It is clearly of rank $r$. It is also unique and satisfies the strict complementarity condition by the same reasoning as the previous example. Moreover, note that as $\epsilon$ approaches zero, the dual vector $y$ approaches a dual optimal solution to \eqref{eq: sdp.p} whose slack matrix is not strictly complementary to $e_1e_1^\top$.
\end{proof}
\begin{rem}\label{rem: choice_rho_high_rank_localness}
In Proposition \ref{exam: high_rank_ALM_nonlocal_y} and Proposition \ref{exam: high_rank_ALM_local_nonstrict_y}, we set $\rho=1$. We can deal with the general $\rho>0$ by setting $C$ in  \eqref{eq: choice_CAb_nonlocal_y_high_rank_solution} as $C=\rho I$ and leave the rest untouched. Our conclusion that \eqref{eq: ALmin} has a unique primal solution matching the Barvinok-Pataki bound remains valid. 

We can also make $C$ independent of $\rho$ by relaxing the bound by one. We set $C=0$ and add a constraint $A_0 = \diag(0,\mathbf{1}_{n-r})$  and $b_0 =0$ to \eqref{eq: choice_CAb_nonlocal_y_high_rank_solution}. Then, we set the additional dual variable $y_0  = -\rho$ for every dual vector involved in the previous examples. In this case, we increase the number of constraints to $m+1$ and the rank of the primal optimal solution to \eqref{eq: ALmin} is still $r$ with $r= \max_\ell\{\ell\mid \frac{\ell(\ell+1)}{2}\leq m\}$, which almost matches the Barvinok-Pataki bound.
\end{rem}

Our previous propositions show that if the localization assumption on the dual vector $y$ fails, then the resulting \eqref{eq: ALmin} could admit a unique solution with the highest rank, violating the first desired result of Theorem \ref{thm: ALmin_primal_simple_qg}. In the following proposition, we show the second desired result of Theorem \ref{thm: ALmin_primal_simple_qg}, that the quadratic growth holds with a non-vanishing constant, fails if the localization assumption does not hold. 

\begin{prop}[Necessity of localization for quadratic growth condition]
There is some $C$, $\Amap$, $b$, such that the problem \eqref{eq: sdp.p} is primal simple with a rank one solution and for any $\rho>0$, there is a set of dual optimal vectors parameterized by $\epsilon \in [0,\rho)$, which gives augmented Lagrangian $\mathcal{L}$ with a diminishing quadratic growth constant as $\epsilon \rightarrow 0$ and the quadratic growth condition \emph{fails} for $\epsilon =0$.
\end{prop}

\begin{proof}
Consider the SDP \eqref{eq: sdp.p} with the following problem data: 
\begin{equation}
C = 0,\quad 
    A_1 = \begin{bmatrix}
        1 & 0 \\
        0 & 0
    \end{bmatrix},\quad A_2 = \begin{bmatrix}
        0 & 0 \\
        0 & 1
    \end{bmatrix},\quad \text{and}\quad b = \begin{bmatrix}
        0 \\ 1
    \end{bmatrix}\ .
\end{equation}
Note that $X_\star = \begin{bmatrix} 
0 & 0 \\ 
0 & 1\end{bmatrix}$ is a unique solution and strict complementarity is satisfied for $y_\star = [-1,0]^\top$. For any $\epsilon\in [0,\rho)$, consider $y_\epsilon = \begin{bmatrix}
    -\epsilon & 0
\end{bmatrix}^\top$. The matrix $X_\star$ continues to be the unique optimal solution for \eqref{eq: ALmin} with $\rho>0$ and the $y_\epsilon$. However, consider $X= \begin{bmatrix}
\delta & \sqrt{\delta}\\ 
\sqrt{\delta} & 1
\end{bmatrix}\succeq 0$. We have the following difference:
\begin{equation}
     \mathcal{L}_\rho (X,y) - \mathcal{L}_\rho (X_\star,y) 
     = \epsilon \delta + \frac{\rho}{2} \delta^2 \ .
\end{equation}
We also have $\fronorm{X - X_\star} = \sqrt{\delta^2 + 2\delta}$. Thus the quadratic growth parameter $\gamma \leq \frac{\epsilon}{2}$, which is shrinking as $\epsilon$ approaches $0$. Moreover, the quadratic growth condition \emph{fails} for $\epsilon =0$, and we only have a quartic growth. 
\end{proof}

\section{Linear convergence of gradient descent for \eqref{eq: ALmin-BM}}\label{sec:gd}
In this section, we consider gradient descent for \eqref{eq: ALmin-BM}:
\begin{equation}\label{eq:AL- BM- GD}\tag{\texttt{GD-AL-BM}}
    \text{initialize at}\;F_1\in \real^{n\times r},\quad \text{and iterate} \quad F_{t+1} = F_t - \eta \nabla_F \bar{\mathcal{L}}_\rho(F_t,y)\;\text{for}\;t=1,2,3,\dots
\end{equation}
We shall establish that under the same condition as Theorem \ref{thm: ALmin_primal_simple_qg} and the additional assumption that $F_1F_1^\top$ being close to $\Xsol$, \eqref{eq:AL- BM- GD} converges linearly to an optimal solution of  \eqref{eq: ALmin-BM}. We consider both the exact parametrized case $r=\rstar$ in Section \ref{sec: exact_case} and the general parametrized case $r\geq \rstar$ in Section \ref{sec: overparametrized ALG-BM}. Although the second case covers the first, the proof of the exact parametrized case is considerably simpler and applies to a broader class of functions, due to Proposition \ref{thm: f_g_A_GD_linear_convergence} below, which might be of independent interest.  In Section \ref{sec: app_ALM_sbprblm_non_exmpl}, we show that the localization conditions for such linear convergence results are necessary. Otherwise, there are examples of \eqref{eq: ALmin} where spurious local minimizers appear and forbid the convergence to global optimality.


\paragraph{Notation for the factor space} Because of the factorization $X = FF^\top$, \eqref{eq: ALmin-BM} necessarily possesses multiple solutions. To account for that in measuring the progress of \eqref{eq:AL- BM- GD}, we define the following quantities assuming \eqref{eq: ALmin} has a unique solution $\Xyp$. Given a matrix $F \in \mathbb{R}^{n\times r}$ for 
\eqref{eq: ALmin-BM}, we denote the difference of it to the closest optimal solution as $\Delta_F = F - F^*_{\pi}$ where $F^*_{\pi} \in \arg\min_{GG^\top =\Xyp} \fronorm{F- G}$. For an iterate $F_t$ of \eqref{eq:AL- BM- GD}, we add subscript $t$ to the relevant matrices, i.e., $F^*_{t,\pi}\in \arg\min_{GG^\top =\Xyp} \fronorm{F_t- G}$, and denote the difference $\Delta_t = F_t - F^*_{t,\pi}$.

\begin{thm}\label{cor: LCBM} Suppose the assumptions in Theorem \ref{thm: ALmin_primal_simple_qg} hold and fix $\rho>0$ and $r\geq \rstar$.  Then, 
there are constants $c_p,\;c_d,\;c_{\text{AL}},\;\text{and}\;c_{\text{step}}>0$, such that for any $y$ with $\twonorm{y-\ysol}\leq c_d$, and $\fronorm{F_1F_1^\top - \Xsol}\leq c_p$, gradient descent \eqref{eq: GD} for \eqref{eq: ALmin-BM} with a stepsize $\eta \in
[0,c_{\text{step}}]$  and initialized at $F_1$ satisfies the following inequality for any $t\geq 0$:
\begin{equation}\label{eq: ALM_GD_local_convergencen}
\fronorm{\Delta_{t+1}}^2\leq (1-\eta c_{\text{AL}})^t \fronorm{\Delta_1}^2.
\end{equation}
\end{thm}

A few remarks on our theorem are in order.

\begin{rem}[Constants dependence and algorithmic consequences]\label{rem: constant dependence}
    The constants $c_p,\;c_d,\;c_{\text{AL}},\;\text{and}\; c_{\text{step}}$, as stated in Theorem \ref{cor: LCBM}, only depends on the problem data $\Amap$, $C$, $b$, the parameter $\rho$, and the selected strictly complementary solution pair $(\Xsol,\ysol)$; in particular, they are {independent} of the particular local dual vector $y$
    \footnote{If $\rho$ is restricted to be in a finite interval, then the constants only depend on the lower and upper bounds, and $(\Amap,\;b,\;C)$.}. 
    Since ALM requires solving a series of subproblems \eqref{eq: ALmin} parametrized by $y$, it is critical for the analysis of ALM that these constants are independent of the exact position of $y$. Without such independence, even if each \eqref{eq: ALmin} satisfies strict complementarity and quadratic growth, the subproblems might become progressively harder as the method approaches the minimizer $\Xsol$. Specifically, the convergence rate of some subproblems may approach $1$, or their initial radius may shrink to $0$.

    Furthermore, it is noteworthy that we only require $F_1F_1^\top$ to be close to the primal optimal solution $\Xsol$, rather than requiring $F_1F_1^\top$ to be close to $X_{y,\rho}$, the minimizer of \eqref{eq: ALmin}, which varies with $y$. Such a requirement is particularly beneficial for analyzing the sequence of subproblems in \eqref{eq: ALM} because the initialization condition does not tighten with each iteration.
\end{rem}

\begin{rem}
[Localization of $y$ and $F_1$, and the overall analysis of ALM-BM] 
In Section \ref{sec: app_ALM_sbprblm_non_exmpl}, we shall show that the localization conditions on both the dual vector $y$ and the initialization $F_1$ are necessary for the desired linear convergence to global optimality. Apart from the necessity of these two assumptions, since we ultimately care about the overall convergence of ALM, it is also natural to wonder (i) how to achieve these two conditions initially, and (ii) whether these conditions continue to hold easily in the next round of ALM,  if the current subproblem is solved exactly or inexactly. 

For the first question, theoretically, one may employ and run the methods in \cite{ding2023revisiting, ding2021optimal, yurtsever2021scalable} with the trick in \cite[Section 6]{wang2025accelerated} for a while to achieve a provably good initialization while maintaining the computational cost as low as that of the ALM-BM framework. However, without knowing the constants that depend on the problem solution, it is difficult to provide a practical rule for stopping those methods to theoretically ensure a good initialization. 
In our experiments, we found that the warm-starting strategy seems not necessary when a rank-adaptive strategy is present.

For the second question, if the subproblem \eqref{eq: ALmin-BM} is solved exactly, then, according to Lemma \ref{lem: zypyysol} on nonexpansiveness, we again have that the new iterate $y^{t+1}$ stays in the neighborhood of $\ysol$. Moreover, thanks to Lemma \ref{lem: X_yrho_X_star_diff_dual_diff}, we know the optimal solution to the current \eqref{eq: ALmin-BM} is again close to $X_\star$. Hence, the linear convergence result here again applies to the next subproblem. The primal and dual iterates of ALM \eqref{eq: ALM} converge linearly by classical results in \cite{rockafellar1976augmented} and Lemma \ref{lem: X_yrho_X_star_diff_dual_diff}. If the subproblem \eqref{eq: ALmin-BM} is solved inexactly, the situation becomes much more complicated as the measure of inexactness can vary \cite{rockafellar1976augmented,liu2019nonergodic,xu2021iteration}, and we need to determine how accurately each subproblem needs to be solved. However, since we have linear convergence for each subproblem in terms of the distance, a very strong measure of convergence, the convergence and linear convergence rate of the ALM iterates, although subtle to establish, should still be achievable. We leave the detailed treatment for future work.
\end{rem}

\subsection{Analysis of the exact parametrized case: $r=\rstar$}
\label{sec: exact_case}

It turns out the proof is more transparent to consider the following generic optimization problem with a decision variable $X \in \symMat{n}$, generalizing the ALM subproblem \eqref{eq: ALmin},
\begin{equation}\label{eq: f_g_A}\tag{Gen-ALM}
\begin{aligned}
    \text{minimize} &\quad f(X) \quad  
 \text{subject to}&\quad X\succeq 0\ ,
\end{aligned}
\end{equation}
where the function $f:\symMat{n}\rightarrow \real$ is convex and $L$-smooth.\footnote{The $L_f$-smooth assumption can be further relaxed to local $L$-smoothness near $X_0$. We assume global $L$-smoothness for simplicity.} Note that these standard assumptions are satisfied for the augmented Lagrangian subproblem \eqref{eq: ALmin}.  We shall also assume \eqref{eq: f_g_A} has a \emph{unique} solution $X_0$ throughout this subsection, which is also satisfied by \eqref{eq: ALmin} under conditions for $y$ specified in Theorem \ref{thm: ALmin_primal_simple_qg}. We then consider the corresponding Burer-Monteiro formulation with a variable $F\in \mathbb{R}^{n\times r}$:
 \begin{equation}\label{eq: f_g_A.BM}\tag{Gen-BM}
\begin{aligned}
    \text{minimize} &\quad h(F) :=f(FF^\top), 
\end{aligned}
\end{equation} 
and the gradient descent method: select a stepsize $\eta>0$,
\begin{equation}\label{eq: GD}\tag{\texttt{GD}}
    \text{initialize at}\;F_1\in \real^{n\times r},\quad \text{and iterate} \quad F_{t+1} = F_t - \eta \nabla h(F_t)\;\text{for}\;t=1,2,\dots
\end{equation}


We shall establish a proposition, Proposition \ref{thm: f_g_A_GD_linear_convergence} below, of linear convergence of \eqref{eq: GD} with explicit control of convergence rate and initial radius. The proof of Theorem \ref{cor: LCBM} then quickly follows. A key assumption on \eqref{eq: f_g_A} that drives the proof for the $r=\rstar$ case is the following quadratic growth (QG) condition. Note that we also include the condition that the solution is unique and nontrivial in the definition. 
\footnote{The nontrivial solution assumption eliminates trivial situations and ensures $\gamma \leq L$ always.}
\begin{mydef}[$(\gamma,\,r,\,B)$-QG]
Problem \eqref{eq: f_g_A} satisfies the quadratic growth condition with strength $\gamma$, rank $r$, and radius $B$ if (i) \eqref{eq: f_g_A} has a unique nontrivial optimal solution $X_0\not=0$ and $\rank(X_0)\leq r$, and (ii) for all $X\succeq 0$ with $\nucnorm{X}\leq B$ and $\rank(X)\leq r$, there holds the inequality
  \begin{equation}
   f(X) - f(X_0) \geq \gamma  \fronorm{X-X_0}^2. 
  \end{equation}
  If $r = n$ and $B = +\infty$, then we say Problem \eqref{eq: f_g_A} satisfies $\gamma$-QG. 
\end{mydef}

The quadratic growth condition is actually satisfied in many scenarios, including our focused problem: Augmented Lagrangian \eqref{eq: ALmin} under primal simplicity and the condition that the dual vector $y$ is close to a strict complementary dual solution $\ysol$ (as demonstrated by Theorem \ref{thm: ALmin_primal_simple_qg}). See Appendix \ref{eq: f_g_A_QG_example} for more examples. 

Lastly, we overload notations concerning the factor space as we denote the optimal solution of \eqref{eq: f_g_A} as $X_0$:
For any matrix $F\in \mathbb{R}^{n\times r}$, we denote again $F^*_\pi := \arg\min_{GG^\top =X_0} \fronorm{F-F'}$. The error term for $F$ is denoted again by $\Delta = F-F^*_\pi$. 
When analyzing the iterates $F_t$ of \eqref{eq: GD}, we will use the specific notation $X_t=F_tF_t^\top$ and $\Delta_t = F_t - F_{t,\pi}$ where $F_{t,\pi}:=(F_t)_\pi^*$. 
%
%
We are now ready to present the linear convergence result of \eqref{eq: GD}, whose proof can be found in Section \ref{sec: pf_f_g_A_GD_linear_convergence}

\begin{prop}[Linear convergence of \eqref{eq: GD} for exact parametrization] \label{thm: f_g_A_GD_linear_convergence}
Suppose \eqref{eq: f_g_A} has $(\gamma, r, B)$ quadratic growth for some $B\geq 2\nucnorm{X_0}$ and \eqref{eq: f_g_A.BM} has $r = \rank(X_0)$. Fix any two positive constants $c_0 \leq  \frac{\gamma \sigma_{\min>0}(X_0)}{2}$ and  $c_1 \leq \frac{\gamma \sigma_{\min>0}(X_0)}{4L}$. Then, for any positive
$c_2 \leq  \min \left\{
\frac{c_1\sqrt{\sigma_{\min>0}(X_0)}}{2\sqrt{3}\sqrt{\sigma_{\max}(X_0)}}, \frac{\sigma_{\min>0}(X_0)}{2\sqrt{3}}
\right\}$, 
and 
$c_3 \leq \frac{c_0}{2L_h^2}$ where $L_h = (10\sigma_{\max}(X_0)L+ 2\opnorm{\nabla f(X_0)})$, we have the following inequality for 
gradient descent \eqref{eq: GD} with any $F_1$ satisfying $\fronorm{F_1F_1^\top -X_0}\leq c_2$ and any stepsize $\eta\leq c_3$:
\begin{equation}\label{eq: D_shrink}
\fronorm{\Delta_{t+1}}^2\leq (1-\eta c_0)^t \fronorm{\Delta_1}^2\ . 
\end{equation}
\end{prop}

Before we proceed to the proof of Theorem \ref{cor: LCBM} for the $r=\rstar$ case, we remark the following.
\begin{rem}[Rank identification]
In Proposition \ref{thm: f_g_A_GD_linear_convergence}, we require rank identification, i.e., $r=\rank(X_0)$ to achieve linear convergence of \eqref{eq: GD} with a \emph{constant} stepsize. This requirement is not an artifact of the proof but is intrinsic to the problem structure \cite{zhuo2024computational,zhang2021preconditioned,xiong2023over} when only quadratic growth of \eqref{eq: f_g_A} is assumed. In Section \ref{sec: sublinear_convergence} in the appendix, we show by an example of \eqref{eq: ALmin} that if the rank is overspecified, $r>\rank(X_0)$, the convergence rate is indeed sublinear, no matter how close the initialization is. However, as we demonstrated in the next section, for the specific \eqref{eq: ALmin}, with the primal simplicity assumption on \eqref{eq: sdp.p}, linear convergence of gradient descent is retained even if $r\geq \rstar$.
\end{rem}

\subsubsection{Proof of Theorem \ref{cor: LCBM} for the exact parametrized case $r=\rstar$}
\label{sec: pf_coro_BM}
Here we prove the exact-rank case of Theorem \ref{cor: LCBM} by invoking Proposition \ref{thm: f_g_A_GD_linear_convergence}. 
\begin{proof}[Proof of Theorem \ref{cor: LCBM}]
Let us determine $c_i$, $i=0,\,1,\,2,\,3$ mentioned in Proposition \ref{thm: f_g_A_GD_linear_convergence}. 
With reasons explained in the parenthesis, we know that
there is a $c_d>0$ (depending only on $\Amap$, $C$, $b$, $\rho$, and the selected pair $(\Xsol,\ysol)$), such that for any $\twonorm{y-\ysol}\leq c_d$, the following items hold.
\begin{enumerate}
    \item The inequalities $
\frac{9}{10} \sigma_{\min>0}(\Xsol) 
\leq \sigma_{\min>0}(\Xyp)\leq 
    \sigma_{\max}(X_{y,\rho}) \leq 
\frac{11}{10}\sigma_{\max}(\Xsol)$ and $\nucnorm{\Xyp}\leq \frac{3}{2} \nucnorm{\Xsol}$ (Lemma \ref{lem: X_yrho_X_star_diff_dual_diff} and Weyl's inequality);
\item The Lipschitz constant of the gradient of $\mathcal{L}_\rho$ is $\rho\opnorm{\Amap}^2$ and 
$\opnorm{\nabla_X \mathcal{L}_\rho(\Xyp,y)}
\overset{\eqref{eq: ALmin.KKT.fo}}{=}\opnorm{Z(\zyp)}\leq \frac{11}{10}\sigma_{\max}(Z(\ysol))$ 
(Weyl's inequality and the nonexpansive lemma, Lemma \ref{lem: zypyysol});
\item  The function $\mathcal{L}_\rho(\cdot,y)$ has $(\gamma, n,3\nucnorm{X_\star})$-QG with
$\gamma$ independent of $y$ from \eqref{eq: ALmin.gamma_formula} (Theorem \ref{thm: ALmin_primal_simple_qg}). 
\end{enumerate}
Hence, we see all the constants in Proposition \ref{thm: f_g_A_GD_linear_convergence} can be chosen depending on the problem data $\Amap$, $C$, $b$, and $\rho$. Specifically, we have the following choices:
\begin{equation}
\begin{aligned}
c_0 & = \frac{\gamma \sigma_{\min>0}(\Xsol)}{4}\ , 
& c_1 & = \frac{\gamma \sigma_{\min>0}(\Xsol)}{8\rho\opnorm{\Amap}^2}\ ,\\ 
c_2 & = \min
\left\{\frac{c_1 \sqrt{\sigma_{\min>0}(\Xsol)}}{4\sqrt{3}\sqrt{\sigma_{\max}(\Xsol)}},\frac{\sigma_{\min>0}(\Xsol)}{4\sqrt{3}}\right\}\ , 
& c_3 &= \frac{c_0}{2(15\rho\sigma_{\max}(\Xsol)\opnorm{\Amap}^2+3\fronorm{Z(\ysol)})^2} \ .
\end{aligned}
\end{equation}

Moreover, from Lemma \ref{lem: X_yrho_X_star_diff_dual_diff}, we also have 
$\fronorm{X_{y,\rho}-\Xsol} \leq \frac{c_2}{2}$ if $c_d$ is small enough.

With the above choices, if $\fronorm{FF^\top -\Xsol}\leq \frac{c_2}{2} = c_p$, then $\fronorm{FF^\top -X_{y,\rho}}\leq c_2$. The condition of Proposition \ref{thm: f_g_A_GD_linear_convergence} is satisfied, and we have that \eqref{eq: ALM_GD_local_convergencen} holds with $c_{\text{AL}}=c_0$ and $c_{\text{step}} = c_3$. Our proof is complete.
\end{proof}



\subsubsection{Proof of Proposition \ref{thm: f_g_A_GD_linear_convergence}}\label{sec: pf_f_g_A_GD_linear_convergence}
To establish Proposition \ref{thm: f_g_A_GD_linear_convergence}, we first prove that a local descent condition, inspired by the works  \cite{chen2015fast,yi2016fast}, holds for our problem. We then analyze the convergence of the algorithm based on this condition. Notably, unlike the approach in \cite{chen2015fast,yi2016fast}, which focuses on machine learning applications and relies on distributional assumptions on the input data, our framework begins directly from the quadratic growth condition of the original problem \eqref{eq: f_g_A}, as motivated by our main focus \eqref{eq: ALmin}. 
\textcolor{black}{Apart from the local descent, another key condition for linear convergence of gradient descent is the local smoothness of the gradient, which we also establish in the following lemma.}

\begin{lem}[Local descent and smoothness]\label{thm: f_g_A}
  Suppose the assumptions of Proposition \ref{thm: f_g_A_GD_linear_convergence} hold.
    Then the objective function in  
    \eqref{eq: f_g_A.BM} satisfies the following two conditions:
\begin{enumerate}
    \item Local descent condition: there holds the inequality 
    \begin{equation}\label{eq: local_descent}
        \inprd{\nabla h (F)}{\Delta } \geq c_0 \fronorm{\Delta }^2 \text{ for all $F$ with $\fronorm{FF^\top -X_0}\leq c_1$}\ ,
    \end{equation}
  for any $c_0 \leq  \frac{\gamma \sigma_{\min>0}(X_0)}{2}$ and $c_1 \leq \frac{\gamma \sigma_{\min>0}(X_0)}{4L}$. 
  \item Local smoothness:  For any $F$ with $\fronorm{\Delta}
   \leq \frac{1}{4}\sqrt{\sigma_{\min>0}(X_0)}$, we have that the gradient $\nabla h(F)$ satisfies 
   \begin{equation}\label{eq: local_smoothness}
       \fronorm{\nabla h(F)}\leq L_h \fronorm{\Delta},
   \end{equation}
where $L_h = 10 \sigma_{\max}(X_0)L + 2\opnorm{\nabla f(X_0)}$. Recall $\Delta = F-F^*_\pi$ where $F_\pi^* := \arg\min _{F'(F')^\top = X_0} \fronorm{F- F'}.$
\end{enumerate}
Moreover, the conditions $\fronorm{FF^\top -X_0}\leq c_1$ and $\fronorm{\Delta}\leq \frac{1}{4}\sqrt{\sigma_{\min>0}(X_0)}$ are satisfied if  $\fronorm{\Delta}\leq c_\Delta$ where $c_\Delta = \min \left\{
\frac{c_1}{4\sqrt{\sigma_{\max}(X_0)}}, \frac{\sqrt{\sigma_{\min>0}(X_0)}}{4}
\right\}.$
\end{lem}

\begin{proof} 
We shall denote $X= FF^\top$ in the following. Let us first prove the \emph{local descent condition}. Note
$
\inprd{\nabla h(F)}{\Delta} \overset{(a)}{=} \inprd{ 2 \nabla f(X) F}{\Delta} \overset{(b)}{=} \inprd{\nabla f(X)}{\Delta F^\top + F \Delta^\top} \overset{(c)}{=} 
    \inprd{\nabla f(X)}{X-X_0+\Delta\Delta^\top}\ .
$
Here the step $(a)$ is due to $\nabla h(F) = 2 \nabla f(X)F$. The step $(b)$ is because $\nabla f(X)$ is symmetric. The last step $(c)$ is because $\Delta F^\top + F \Delta^\top = X-X_0+\Delta\Delta^\top$. Using the above identity, we have 
\begin{equation} 
 \begin{aligned}
\inprd{\nabla h (F)}{\Delta} 
& = \inprd{\nabla f(X)}{X - X_0} 
+ \inprd{\nabla f(X) - \nabla f(X_0)}{\Delta\Delta^\top} + \inprd{\nabla f(X_0)}{\Delta \Delta^\top} \\ 
&\overset{(a)}{\geq} \gamma \fronorm{X-X_0}^2 - \frac{\gamma\sigma_{\min>0}(X_0)}{4}\fronorm{\Delta \Delta^\top} \\ 
&\overset{(b)}{\geq}
\frac{3\gamma \sigma_{\min>0}(X_0)}{4} \fronorm{\Delta}^2 - \frac{\gamma}{4}\sigma_{\min>0}(X_0)\fronorm{\Delta}^2\\
&\geq \frac{\gamma \sigma_{\min>0}(X_0)}{2}\fronorm{\Delta}^2\ . 
    \end{aligned}
    \end{equation}
In step $(a)$, we lower bound 
$\inprd{\nabla f(X)}{X-X_0}\geq f(X) - f(X_0)$, thanks to the convexity of $f$, we then use quadratic growth to bound $f(X)-f(X_0)$.\footnote{Note we do have $\nucnorm{X}\leq 2\nucnorm{X_0}$ here to apply the quadratic growth property, thanks to our choice of $c_1$ and $\rank(X)\leq \rank(X_0)$. Indeed, we have 
$    
\nucnorm{X} \leq \nucnorm{X-X_0} + \nucnorm{X_0}\overset{(i)}{\leq}\sqrt{2\rank(X_0)} \fronorm{X-X_0} +\nucnorm{X_0}\overset{(ii)}{\leq} 2\nucnorm{X_0}.
$
In the step $(i)$, we use the fact that $X$ and $X_0$ have rank no more than $\rank(X_0)$. In the step $(ii)$, we use $\fronorm{X-X_0}\leq c_1$. 
}
The middle term $\inprd{\nabla f(X) - \nabla f(X_0)}{\Delta\Delta^\top}$ is bounded by Cauchy-Schwarz, the Lipschitz continuity of $\nabla f(X)$, and our requirement on $\fronorm{FF^\top -X_0}$. We drop the last term $\inprd{\nabla f(X_0)}{\Delta \Delta^\top}$ as both matrices are PSD. In step $(b)$, we use \cite[Lemma 6]{ge2017no} for lower bounding $\fronorm{X-X_0}$.

Next, we prove  the \emph{local smoothness}.
Note that condition $\fronorm{\Delta}\leq \frac{1}{4}\sqrt{\sigma_{\min>0}(X_0)}$ implies that 
\begin{align}\label{eq: F_bound_BM_X_bound}
    \opnorm{F}
     \leq \frac{5}{4}\sqrt{\sigma_{\max}(X_0)}\quad \text{and}\quad 
    \fronorm{X-X_0} 
     \overset{(a)}{\leq} 4\sqrt{\sigma_{\max}(X_0)} \fronorm{\Delta},
\end{align}
where the step $(a)$ is due to 
$X-X_0= \Delta F^\top +F\Delta ^\top -\Delta \Delta ^\top$.
We are now ready to estimate the Frobenius norm of  $\nabla h(F) = 2\nabla f(X)F$:
\begin{equation}
\begin{aligned}\label{eq: bm_gradient_lip}
\fronorm{\nabla h(F)} =  2\fronorm{\nabla f(X)F} 
    & \overset{(a)}{=} 
    2\fronorm{\nabla f(X)F -\nabla f(X_0)F_{\pi}^*} \\ 
    & \overset{(b)}{\leq} 
    2\fronorm{\nabla f(X)F -\nabla f(X_0)F} + 
    2\fronorm{\nabla f(X_0)F -\nabla f(X_0)F_{\pi}^*}\\ 
    & \overset{(c)}{\leq} 
    2\fronorm{\nabla f(X) -\nabla f(X_0)} \opnorm{F} + 
    2\opnorm{\nabla f(X_0)}\fronorm{\Delta}\\
    & \overset{(d)}{\leq} (10\sigma_{\max}(X_0)L+ 2\opnorm{\nabla f(X_0)})\fronorm{\Delta}\ .
\end{aligned}
\end{equation}
Here, in the step $(a)$, we use $\nabla f(X_0) F_{\pi}^* =0$ due to optimality of $F_\pi^*$ to \eqref{eq: f_g_A.BM}. In the step (b), we use the triangle inequality and add and subtract the term $\nabla f(X_0)F$. In the step $(c)$, we use the properties of the Frobenius norm. In the last step $(d)$, we use that $\nabla f$ is $L$-Lipschitz and \eqref{eq: F_bound_BM_X_bound}.

Finally, we note that $\fronorm{\Delta}\leq \frac{1}{4}\sqrt{\sigma_{\min>0}(X_0)}$ is implied by $\fronorm{\Delta}\leq c_\Delta$ due to our choice of $c_\Delta$. 
And $\fronorm{FF^\top -X_0}\leq c_1$ is implied by $\fronorm{\Delta}\leq c_\Delta$ and \eqref{eq: F_bound_BM_X_bound}.
\end{proof}

With the lemma in hand, let us prove Proposition \ref{thm: f_g_A_GD_linear_convergence}. 

\begin{proof}[Proof of Proposition \ref{thm: f_g_A_GD_linear_convergence}]
We will show that for all $t>0$, we have
\begin{align}
\fronorm{\Delta_{t+1}}^2
&
\leq (1-\eta c_0) \fronorm{\Delta_t}^2 \label{eq: D_shrink_onestep}\\
\fronorm{\Delta_t}
&\leq c_\Delta =
\min \left\{
\frac{c_1}{4\sqrt{\sigma_{\max}(X_0)}}, \frac{\sqrt{\sigma_{\min>0}(X_0)}}{4}
\right\}
\ . \label{eq: Delta_bound_inducition}
\end{align}
The convergence inequality \eqref{eq: D_shrink} is then immediate after establishing \eqref{eq: D_shrink_onestep}. First, 
note that the bound \eqref{eq: Delta_bound_inducition} for $t=1$ holds. Indeed, from
 \cite[Lemma 6]{ge2017no}, we have that 
 \begin{equation}\label{eq: BM_Delta_bound}
   \fronorm{\Delta }\leq \frac{\sqrt{3}}{2\sqrt{\sigma_{\min>0}(X_0)}}\fronorm{X - X_0} 
   \overset{(a)}{\leq} 
   c_\Delta \ ,
 \end{equation}
 where the step $(a)$ is due to our choice of $c_2$. 
Next, note that if \eqref{eq: Delta_bound_inducition} holds for $t$, then \eqref{eq: D_shrink_onestep} holds:
%
%
\begin{equation}
    \begin{aligned}
    \label{eq: shrinkage_derivation_BM_steps}
\fronorm{\Delta_{t+1}}^2 &
\overset{(a)}{\leq}  \fronorm{\Delta_{t}}^2 
-2\eta \inprd{\nabla h(F_t)}{F_t-F^*_{t,\pi}} + \eta ^2 \fronorm{\nabla h(F_t)}^2 \\ 
&\overset{(b)}{\leq} \fronorm{\Delta_t}^2 -2\eta c_0 \fronorm{\Delta_t}^2 +\eta^2 L_h^2 \fronorm{\Delta_t}^2
 \overset{(c)}{\leq} (1-\eta c_0)\fronorm{\Delta_t}^2.
    \end{aligned}
\end{equation}
Here, in the step $(a)$, we first use 
$\fronorm{\Delta_{t+1}}^2
\leq 
\fronorm{\Delta_t- \eta \nabla h(F_t)}^2 $ 
due to the optimality of $F_{t+1,\pi}^*$, and then we expand the square  $\fronorm{\Delta_t- \eta \nabla h(F_t)}^2$. The step $(b)$ is because  the local descent condition \eqref{eq: local_descent} and the local smoothness \eqref{eq: local_smoothness} are satisfied due to the assumed \eqref{eq: Delta_bound_inducition} at $t$ (see the ``Moreover" part of Lemma \ref{thm: f_g_A}). 
The step $(c)$ is because $\eta \leq \frac{c_0}{2L_h^2}$. Lastly, note by our choice of $\eta$, we also have $\fronorm{\Delta_{t+1}}\leq \fronorm{\Delta_t}\leq c_\Delta$ from \eqref{eq: D_shrink_onestep}. 
Thus, combining the three notes above, a simple induction on \eqref{eq: D_shrink_onestep}
 \eqref{eq: Delta_bound_inducition} yields both of them for all $t>0$.
%
%

%
\end{proof}

\begin{rem}[Related results and techniques]
The linear convergence result here may not be that surprising given the literature. Using  \cite[Lemma 6]{ge2017no}, one can quickly conclude that \eqref{eq: ALmin-BM} admits a quadratic growth condition from the QG of \eqref{eq: f_g_A}. Using the results in \cite{rebjock2024fast}, one can conclude that a local Polyak-\L ojasiewicz inequality holds. Consequently, the linear convergence result follows assuming the convergence of the \ref{eq: GD} iterates. In \cite{tang2024feasible}, the authors derive the linear convergence in a Riemannian setting, following the above line, with a direct derivation of the local Polyak-\L ojasiewicz inequality, assuming the convergence of the iterate.
Note that, in this case, the contraction factor in \ref{eq: GD} is implicit and its dependence is unclear, so it may approach $1$ for the subproblems of \eqref{eq: ALmin-BM} as $y$ approaches an optimal solution. Moreover, convergence is measured by the objective value gap.

Here, we take a different approach by verifying the local descent condition \cite{chen2015fast,yi2016fast} and a careful tracking of the constants. As a consequence, we establish very strong convergence results: (i) our convergence is \emph{derived rather than assumed}, (ii)  we have a precise control over the contraction factor of linear convergence and the initial convergence radius, and (iii) we establish convergence in the distance to the solution set, a strong measure of convergence. Such results are critical to the overall convergence analysis of ALM, as we explain in Remark \ref{rem: constant dependence}.
\end{rem}

\subsection{Analysis of the general parametrized case: $r\ge \rstar$}\label{sec: overparametrized ALG-BM}

The result in the prior section requires that the parametrized rank $r = \rstar$, which can be very restrictive in practice. In this section, by employing different techniques, we show that even if \eqref{eq: ALmin-BM} is over-parametrized, i.e., $r\geq \rstar$, under the same assumption in Theorem \ref{thm: ALmin_primal_simple_qg}, gradient descent with constant stepsize still converges linearly to the set of optimal solutions of \eqref{eq: ALmin-BM} with nondegraded rate and initial radius. The reason that we still achieve linear convergence without requiring $r=\rstar$ is that our primal simplicity assumption is a stronger condition than merely quadratic growth of \eqref{eq: ALmin} and we utilize the special form of \eqref{eq: ALmin} rather than the general form \eqref{eq: f_g_A}.
%
%
%
%
%
However, the contraction factor, the closeness of the dual vector to $\ysol$, and initial $F_1F_1^\top$ to $\Xsol$  in the proof of Theorem \ref{cor: LCBM} for the exact parametrized case might be better than those in the proof of Theorem \ref{cor: LCBM} for the general parametrized case.

We proceed with the proof in two steps: First, we establish a key lemma on the Hessian of the objective of \eqref{eq: ALmin-BM}. Second, we then conclude the 
proof based on the local Lipschitzness of the gradient and the Hessian of \eqref{eq: ALmin-BM}.

\subsubsection{A key lemma on restricted positive definiteness of Hessian}
The key to show the linear convergence of gradient descent is to show  that the Hessian of \eqref{eq: ALmin-BM} at any optimal solution $F$ to \eqref{eq: ALmin-BM} is positive definite along the normal direction to the solution set at the optimal solution $F$. Note that 
for \eqref{eq: ALmin-BM}, if $F_y$ is an optimal solution, then $F_y O$ is also an optimal solution for $OO^\top = I_r$ and $O\in \mathbb{R}^{r\times r}$. Thus, due to non-isolated solutions, the Hessian is always degenerate for nontrivial $\Xyp$, and it is necessary to consider only the directions normal to the solution set. Moreover, since we are actually considering all $y$ local to the $\ysol$, we need to make sure the positive definiteness along the normal directions is bounded away from zero. Both of them are technical challenges we need to overcome. We establish the above precisely below.

\begin{lem}\label{lem: pd_hessian_AL_along_critical_directions}
Suppose \eqref{eq: sdp.p} is primal simple. Let $(\Xsol,\ysol)$ be a strict complementary pair with optimal dual slack $Z(\ysol) = C- \Amap^* \ysol.$ Define the following set 
$\mathcal{H}\subset  \mathbb{R}^{n \times r}\times \mathbb{R}^{n\times r}\times \mathbb{S}^n$:
\begin{equation}\label{eq: Delta_F_Z_set_pd_hessian_AL}
\begin{aligned}
\mathcal{H} =\bigg\{
(\Delta, F, Z) \mid 
& \Delta^\top F = F^\top \Delta,  \opnorm{FF^\top -\Xsol} \leq \frac{\sigma_{\min>0}(\Xsol) }{10}, \\
& ZF =0, Z\succeq 0, \fronorm{Z-Z(\ysol)}\leq \frac{\sigma_{n-\rstar}(Z(\ysol))\sigma_{\min}(\Amap_{\Vstar})}{40\sqrt{2}\opnorm{\Amap}} \bigg\} \ .
\end{aligned}
\end{equation}
Then we have 
\begin{equation}\label{eq: qf_hessian_positive_AL}
\mathcal{Q}_{Z,F}(\Delta) := \mathcal{Q}_Z(\Delta) + \mathcal{Q}_F(\Delta)  > 0 \quad \text{for any}\quad 
(\Delta, F,Z) \in \mathcal{H} 
\;\text{with}\;\Delta\not=0,
\end{equation}
where $\mathcal{Q}_Z(\Delta) :=\inprd{Z}{\Delta \Delta^\top}$ and $\mathcal{Q}_F(\Delta):= \frac{\rho}{2} \twonorm{\Amap(\Delta F^\top +F\Delta^\top)}^2 $,
and
\begin{equation}\label{eq: pd_hessian_AL}
\sigma_{AL}>0 \quad \text{where}\quad  \sigma_{AL}:= \min_{(\Delta,F,Z)\in \mathcal{H},\fronorm{\Delta}=1}
\mathcal{Q}_{Z,F}(\Delta). 
\end{equation}
\end{lem}
 The purpose of defining $\mathcal{H}$ above is that we care about not a single dual vector $y$ but all the vectors around $\ysol$, which induces many slack matrices $Z(y)$ around $Z(\ysol)$. The function $\mathcal{Q}_{Z,F}(\Delta)$ should be understood as the quadratic form of Hessian of the objective function \eqref{eq: ALmin-BM} at an optimal solution $F$. The restriction on $\Delta$, $F^\top \Delta = \Delta^\top F$, is a necessary condition to  ensure $\Delta$ is normal to the solution set at $F$.

\begin{proof}
We first note that the inequality $\sigma_{AL}>0$ follows from \eqref{eq: qf_hessian_positive_AL} as the set $\mathcal{H}\cap \{(\Delta,F,Z)\mid \fronorm{\Delta}=1\}$ is compact and the function $Q_{Z,F}$ is continuous in $(\Delta,F,Z)$. 
Thus, we focus on establishing \eqref{eq: qf_hessian_positive_AL}.

\underline{Preparation.} The inequalities of $\opnorm{FF^\top - \Xsol}$ and $\fronorm{Z-Z(\ysol)}$ in \eqref{eq: Delta_F_Z_set_pd_hessian_AL} implies that $\rank(F)=\rank(FF^\top) \geq \rank(\Xsol)$ and $\rank(Z)\geq \rank(Z(\ysol))$ as consequences of Weyl's inequality. Combining the rank bounds with the condition $ZF =0$ and $Z\succeq 0$ in \eqref{eq: Delta_F_Z_set_pd_hessian_AL}, we obtain
\begin{align}
& \rank(Z) = \rank(Z(\ysol)) =n - \rank(\Xsol) = n-\rank(F) 
& \nullspace(Z)  = \range(F) \label{eq: p_d_space_al}. 
\end{align}
Next, let the economic SVD of $F$ be 
$F = U\Sigma V^\top $ where $U\in\mathbb{R}^{n\times \rstar}$, $\Sigma \in \mathbb{S}^{\rstar \times \rstar}$ and has positive diagonals,  and $V\in \mathbb{R}^{r \times \rstar }$. 
Note that $\nullspace(Z) = \range(F) = \range(U)$ as we just established. Since $Z\succeq 0$, we know the eigenspace corresponding to the $\rstar$ smallest eigenvalues of $Z$ is the same as $\nullspace(Z)$. Using Lemma \ref{lem: Z_Z_star_A_U_inj} (stated and proved in Appendix \ref{sec: pf_lem_auxiliary}) and the bound on $\fronorm{Z-Z(\ysol)}$, we have 
\begin{equation}\label{eq: injective_A_U_S}
\text{the map} \;\Amap_{U}: \; Y \in \mathbb{S}^{\rstar} \mapsto \Amap(U Y  U^\top ) \in \mathbb{R}^m\;\text{is injective}.
\end{equation} 

\underline{Establishing \eqref{eq: qf_hessian_positive_AL}.} Note that the quadratic forms $Q_Z$ and $Q_F$ are positive semidefinite, as $Z$ and $\Delta\Delta^\top$ are positive semidefinite. Hence, the quadratic form $Q_{Z,F}(\Delta)\geq 0$ always. Suppose there is $(\Delta,F,Z) \in \mathcal{H}$ such that $\mathcal{Q}_{Z,F} (\Delta)=0$. We aim to show that $\Delta =0$ in this case. The assumption $Q_{Z,F}(\Delta)=0$ implies that 
\begin{align}
    & \inprd{Z}{\Delta \Delta^\top} =0 \implies \range(\Delta) \subset \nullspace(Z) = \range(F) \label{eq: Delta_range}\\  
    & \twonorm{\Amap(\Delta F^\top + F\Delta^\top)} =0. \label{eq: A_F_0}
\end{align}
Using \eqref{eq: Delta_range}, and the SVD of $F$, $F = U\Sigma V^\top$, we know $\Delta = U\Sigma S$ for some $S\in \mathbb{R}^{\rstar \times r}$. Using this expression of $\Delta$ in \eqref{eq: A_F_0}, we obtain that 
\begin{equation}\label{eq: skew-symmetry}
0 = \Amap(U(\Sigma SV\Sigma +\Sigma V ^\top S^\top \Sigma ) U^\top) \overset{\eqref{eq: injective_A_U_S}}{ \implies} SV + V^\top S^\top =0
\end{equation}
Since $(F,\Delta, Z) \in \mathcal{H}$, we know that 
\begin{equation}
F^\top \Delta = \Delta^\top F \implies 
V \Sigma^2 S =  S^\top \Sigma^2 V^\top. 
\end{equation}
Using $V^\top V = I_{\rstar}$, continuing the above, we have 
\begin{align}\label{eq: skew-symmetry-Sigma}
& V\Sigma^2 S = S^\top \Sigma^2 V^\top 
\implies  V^\top V \Sigma^2 S V= V^\top 
S^\top \Sigma^2 V^\top V
\iff  \Sigma^2 S V =  V^\top 
S^\top \Sigma^2.
\end{align}
Combining the above with \eqref{eq: skew-symmetry}, we see 
\[
-\Sigma^2 V^\top S^\top \overset{\eqref{eq: skew-symmetry}}{=} \Sigma^2 SV \overset{\eqref{eq: skew-symmetry-Sigma}}{=}
V^\top S^\top \Sigma^2  
\implies  SV=0 \quad \text{as $\Sigma$ has positive diagonals.}
\]
Next, using $V\Sigma^2 S = S^\top \Sigma^2 V^\top$ again, we have 
\begin{align*}
 V\Sigma^2 S = S^\top \Sigma^2 V^\top 
\implies V \Sigma^2 S V= 
S^\top \Sigma^2 V^\top V
\overset{SV=0}{\implies}  
0 = S^\top \Sigma^2 
\implies S = 0.
\end{align*}
Thus, we know $\Delta = U\Sigma S = 0$ and we have shown \eqref{eq: qf_hessian_positive_AL}.
\end{proof}

\subsubsection{Proof of Theorem \ref{cor: LCBM} for the general $r\geq \rstar$ case}

We first present two auxiliary lemmas, proved in Appendix \ref{sec: appendix_over_param}, that facilitate the proof of Theorem \ref{cor: LCBM}. The first one demonstrates the Lipschitzness of the gradient and Hessian of \eqref{eq: ALmin-BM}.

\begin{lem}[Local Lipschitz property of gradient and Hessian] 
\label{lem: local_lips_grad_hessian} Suppose the assumptions of Theorem \ref{thm: ALmin_primal_simple_qg} hold and suppose \eqref{eq: ALmin} has a unique solution $\Xyp$. Also suppose $F\in \mathbb{R}^{n\times r}$ with $\fronorm{F-\bar{F}}\leq \frac{1}{4}\sqrt{\sigma_{\min>0}(\Xyp)}$ for some $\bar{F}$ satisfying $\bar{F}\bar{F}^\top = \Xyp$. Then, we have the following Lipschitz property of the derivatives of  the objective of \eqref{eq: ALmin-BM}:
\begin{align}
\fronorm{\nabla_F \Lf(F,y) - \nabla_F \Lf(\bar{F},y)}
& \leq L_{\rho,1} \fronorm{F-\bar{F}},\quad \text{and}\\
\opnorm{\nabla^2_F \Lf (F,y) - \nabla ^2_F \Lf(\bar{F},y)}
&\leq 
L_{\rho,2} \fronorm{F-\bar{F}},
\end{align}
$\text{where}\;L_{\rho,1} \leq  
10\rho\sigma_{\max}(\Xyp)\opnorm{\Amap}^2+ 2\opnorm{Z(\zyp)}$ and $L_{\rho,2} = 15 \rho \opnorm{\Amap}^2 \sqrt{\sigma_{\max}(\Xyp)}$.
\end{lem}


Our next lemma shows that the Hessian is positive definite along the normal direction to the solution set at points near the optimal solution set. It is a simple consequence of Lemma \ref{lem: pd_hessian_AL_along_critical_directions}  and Lemma \ref{lem: local_lips_grad_hessian}. The formal proof is in Appendix \ref{sec: appendix_over_param}.

\begin{lem}[Local Hessian positive definiteness along some critical directions] \label{lem: Hessian_pd_critical_direction}
Suppose the assumptions of Theorem \ref{thm: ALmin_primal_simple_qg} hold and suppose \eqref{eq: ALmin} has a unique solution $\Xyp$. Denote
$c_H =
\min\{
\frac{\sigma_{\min>0}(\Xsol)}{10c_{pd}},c_{\rho},c_{\text{aux}}\}$ where $c_{\rho}$, $c_{\text{aux}}$, and $c_{pd}$ are defined in Lemma \ref{lem: X_yrho_X_star_diff_dual_diff}.
Then for any dual vector $y$ with $\twonorm{y-\ysol}\leq c_{H}$, the Hessian $\nabla ^2\Lf_\rho(F,y)$ for any $F$ with $\fronorm{F-\bar{F}}\leq \frac{\sigma_{AL}}{L_{\rho,2}}$ for some $\bar{F}\bar{F}^\top=\Xyp$ 
is positive definite along any direction $\Delta$ satisfying 
$\Delta ^\top \bar{F} = \bar{F}^\top \Delta$ with parameter $\sigma_{AL}$, i.e.,
$
\nabla ^2 \Lf_\rho (F,y)[\Delta,\Delta]\geq \sigma_{AL}\fronorm{\Delta}^2,
$ where $\sigma_{AL}$ is defined in Lemma \ref{lem: pd_hessian_AL_along_critical_directions}.
\end{lem}

We are now ready to prove Theorem \ref{cor: LCBM} for the general $r\geq \rstar$ case.
\begin{proof}[Proof of Theorem \ref{cor: LCBM} for the general $r\geq \rstar$ case]
We first prepare the proof by defining some constants. Recall $\sigma_{AL}$ defined in Lemma \ref{lem: pd_hessian_AL_along_critical_directions}.
We define $L_1$, $L_2$, and $c_R$, and set the constants $c_p$, $c_d$, $c_{\text{step}}$, and $c_{AL}$ as follows:
\begin{equation}
\begin{aligned}
\label{eq: constants_over_param}
L_1 &= 22\rho \sigma_{\max}(\Xsol)\opnorm{\Amap}^2 + 3\opnorm{Z(\ysol)},\;
L_2 = 22\rho \opnorm{\Amap}^2\sqrt{\sigma_{\max}(\Xsol)},\; 
c_{\text{step}}  = \frac{\sigma_{AL}}{L_1^2} , c_{AL} = \sigma_{AL}\; \\
c_R & =\min \{\frac{1}{5} \sqrt{\sigma_{\min>0}(\Xsol)}, \frac{\sigma_{AL}}{L_{2}}\}, 
c_p =  \frac{1}{80}\min\left \{\frac{c_R
^2}{\sqrt{r}},c_R\sqrt{\sigma_{\min>0}(\Xsol)}\right\}, \;\text{and}\; c_d = \min\{\frac{c_\star}{10},\;\frac{c_{p}}{c_{pd}},c_H\},
\end{aligned}
\end{equation}
where $c_\star$  and $c_H$ are defined in Theorem \ref{thm: ALmin_primal_simple_qg} and Lemma \ref{lem: Hessian_pd_critical_direction}, respectively. 
From the bound $\twonorm{y-\ysol}\leq c_d\leq  c_\star$, it is immediate from Theorem \ref{thm: ALmin_primal_simple_qg} that primal simplicity holds for \eqref{eq: ALmin} and $\rank(\Xyp)=n-\rank(Z(\zyp)) = \rstar$. 

\underline{Activation of Lemma \ref{lem: local_lips_grad_hessian} and \ref{lem: Hessian_pd_critical_direction} for $t=1$.}
Let us now  show that $\fronorm{\Delta_{1}}\leq c_{D} :\,= \min \{\frac{1}{4} \sqrt{\sigma_{\min>0}(\Xyp)}, \frac{\sigma_{AL}}{L_{\rho,2}}\}$, so that we can apply both Lemma \ref{lem: local_lips_grad_hessian} and \ref{lem: Hessian_pd_critical_direction} with $\bar{F} = F_{t,\pi}^*$ and $F$ being any point on the line segment of $F_t$ and $F_{t,\pi}^*$ for $t=1$. Indeed, because of the assumption on $\twonorm{y-\ysol}\leq \frac{c_\star}{10}$, we know from Lemma \ref{lem: X_yrho_X_star_diff_dual_diff} and Weyl's inequality that 
\begin{equation}\label{eq: Sigma_Xsol_Xyp}
\frac{9}{10}\sigma_{\min>0}(\Xsol)
\leq \sigma_{\min>0}(\Xyp) \leq 
\sigma_{\max}(\Xyp)\leq \frac{11}{10}\sigma_{\max}(\Xsol),\; \text{and}\;\opnorm{Z(\zyp)}\leq \frac{11}{10}\opnorm{Z(\ysol)}.
\end{equation}
In particular, the above implies $L_1\geq L_{\rho,1}$ and $L_2\geq L_{\rho,2}$, where $L_{\rho,1}$ and $L_{\rho,2}$ are defined in Lemma \ref{lem: local_lips_grad_hessian}, and $c_R\leq  c_{D}$. 
Next, using the triangle inequality, we see 
\begin{equation}\label{eq: FFXyp_over_param_y_ysol_FF_Xsol}
\fronorm{F_1F_1^\top - \Xyp}\leq \fronorm{\Xyp - \Xsol} + 
\fronorm{F_1F_1^\top -\Xsol} \overset{(a)}{\leq}
c_{pd} \twonorm{y-\ysol} + \fronorm{F_1F_1^\top -\Xsol} 
\end{equation}
where the step $(a)$ is because of Lemma \ref{lem: X_yrho_X_star_diff_dual_diff}. Finally, using Lemma \ref{lem: over_para_F_bar_F_distance_FF_Fbar_Fbar_relation}, which upper bounds $\fronorm{\Delta_1}$ in terms of $\fronorm{F_1F_1^\top - \Xyp}$, we establish the bound $\fronorm{\Delta}\leq c_R$ by combining previous inequalities 
\eqref{eq: Sigma_Xsol_Xyp}, \eqref{eq: FFXyp_over_param_y_ysol_FF_Xsol}, 
and the bounds $\twonorm{y-\ysol}\leq c_d$ and $\fronorm{F_1F_1^\top -\Xsol} \leq c_p$.  As $c_R\leq c_D$, we see the bound $\fronorm{\Delta_1}\leq c_D$ holds.

\underline{Analysis of $\Delta_t$.} Let us derive the bound on $\fronorm{\Delta_{t+1}}$ for $t=1$. As $y$ is fixed, we omit the dependence of $y$ for $\Lf$ in the following. We have that  
\begin{equation}\label{eq: gd_decomposition}
\begin{aligned}
\fronorm{\Delta_{t+1}}^2
\overset{(a)}{\leq}&
\fronorm{F_{t} - \eta \nabla \Lf(F_{t}) - F_{t,\pi}^*}^2 = 
\fronorm{\Delta_t}^2 - 2\eta \inprd{\nabla \Lf(F_t)}{\Delta_t} +\eta^2 \fronorm{\nabla \Lf(F_t)}^2,\\
\overset{(b)}{\leq}&
\fronorm{\Delta_t}^2 - 2\eta \inprd{\nabla \Lf(F_t)}{\Delta_t} +\eta^2 L_{\rho,1}^2 \fronorm{\Delta_t}^2,
\end{aligned}
\end{equation}
where we use the optimality of $F_{t+1,\pi}^*$ in the step $(a)$ and the Lipschitzness of gradient in the step $(b)$. For the middle term, we have 
\begin{equation}
\inprd{\nabla \Lf(F_t)}{\Delta_t} = 
\inprd{\nabla \Lf(F_t)-\nabla \Lf(F_{t,\pi}^*)}{\Delta_t}
= \int_{0}^1 \nabla^2\Lf(sF_t+(1-s)F_{t,\pi}^*)[\Delta_t,\Delta_t]ds\overset{(a)}{\geq} 
\sigma_{AL}\fronorm{\Delta_t}^2.
\end{equation}
where in the step $(a)$, we use Lemma \ref{lem: Hessian_pd_critical_direction} (with $\bar{F}=F^*_{t,\pi}$ and $F = 
sF_t + (1-s)F_{t,\pi}^*$) and $\Delta_t^\top F^*_{t,\pi} = (F_{t,\pi}^*)^\top \Delta_t$ thanks to Lemma \ref{lem: Delta_barF} (proved in Appendix \ref{sec: appendix_over_param}). Thus, we may continue the chain of inequalities in \eqref{eq: gd_decomposition}:
\begin{equation}\label{eq: gd_cnvgnc}
\begin{aligned}
\fronorm{\Delta_{t+1}}^2\leq (1-2\sigma_{AL} \eta + \eta^2 L_{\rho,1}^2)\fronorm{\Delta_t}^2\leq 
(1-\eta \sigma_{AL})\fronorm{\Delta_t}^2.
\end{aligned}
\end{equation}
where the last step is due to our choice of $\eta$ as $L_1\geq L_{\rho,1}$ proved earlier. Note that $\fronorm{\Delta_{t+1}}\leq \fronorm{\Delta_t}$ by the above inequality. Hence, Lemma \ref{lem: local_lips_grad_hessian} and \ref{lem: Hessian_pd_critical_direction} are again active for $t=2$ and the above argument applies again on $\fronorm{\Delta_{t+1}}$ for $t=2$. Thus, an easy induction on $\fronorm{\Delta_{1}}\leq c_D = \min \{\frac{1}{4} \sqrt{\sigma_{\min>0}(\Xyp)}, \frac{\sigma_{AL}}{L_{\rho,2}}\}$ and \eqref{eq: gd_cnvgnc} proves both of them hold for all $t$. Our desired results then follow, and the proof is complete.

\end{proof}

\subsection{Sensitivity of Assumptions in Theorem \ref{cor: LCBM}}\label{sec: app_ALM_sbprblm_non_exmpl}

In this section, we present examples illustrating that the localization conditions on the dual vector $y$ and the initialization $F_1$ in Theorem \ref{cor: LCBM} are necessary. 

%
%

Our first example demonstrates that for some choices of $C,\;\Amap,\;b$, and $\rho$, the original \eqref{eq: sdp.p} is primal simple with a rank one solution, yet the ALM subproblem  \eqref{eq: ALmin-BM} with some dual infeasible $y$ admits a high-rank (as high as $n-1$) spurious \emph{local minimizer} $\hat{F}$ of  \eqref{eq: ALmin-BM}. Moreover, the product $\hat{F}\hat{F}^\top$ is feasible for \eqref{eq: sdp.p}. Combining the update of \eqref{eq: ALM}, we see that the ALM in this case will halt at the current $(\hat{F},y)$ if one tries to solve the subproblem \eqref{eq: ALmin-BM} by
gradient descent starting at $\hat{F}$. 
Recall that typical first- and second-order methods can at best guarantee to find a second-order stationary point if they start near $\hat{F}$. Thus, this example also puts more complicated algorithms for solving \eqref{eq: ALmin} at risk. The example is inspired by \cite[Theorem 1]{bhojanapalli2018smoothed}.

\begin{prop}[Necessity of dual localization]
For any $n$, $\rho>0$, and $r\leq n-1$, there is some $C$, $\Amap$, $b$, such that \eqref{eq: sdp.p} is primal simple with a rank one solution, and there is some $y$ for \eqref{eq: ALmin-BM} so that \eqref{eq: ALmin-BM}  admits a spurious local minimizer $\hat{F}$ with rank $r$ and $\Amap(\hat{F}\hat{F}^\top) = b$.
\end{prop}
\begin{proof}
Let us now specify $\Amap$, $C$, and $b$. For $\Amap$ and $C$, we set them in the following way:
\begin{equation}\label{eq: ALM-BM-non-example-A-C-1}
    A_i = \frac{1}{2}\left(e_ie_n^\top + e_n e_i^\top \right), \; i=1,\dots, n-1,\; A_n = \epsilon \begin{bmatrix}
        I_{n-1} &  0 \\
        0 & 1
    \end{bmatrix}, \;\text{and}\; 
    C = \epsilon \begin{bmatrix}
        2I_{n-1} & 0 \\
        0 & 1
    \end{bmatrix}\ ,
\end{equation}
where $\epsilon>0$ is a number to be determined later. We set $b = \epsilon a \begin{bmatrix}
    0_{n-1} \\ 1 
\end{bmatrix}$ for some $a>0$. 
With the above setup, one can easily verify that $\Xsol = \begin{bmatrix}
    0_{(n-1)\times (n-1)} & 0_{(n-1)\times 1} \\
    0_{1 \times (n-1)} & a
\end{bmatrix}$ is the unique rank one solution, and $\ysol = \begin{bmatrix}
    0_{(n-1)\times 1} \\
    1
\end{bmatrix}$ is a dual optimal solution satisfying strict complementarity. Now consider \eqref{eq: ALmin-BM} with $y = [0_{n-1},\;2]^\top$ and $\hat{F} = \eta \begin{bmatrix}
    I_r \\ 0 
\end{bmatrix}$ for some $\eta >0$. Let $\hat{X} = \hat{F}\hat{F}^\top$. We require the following relationship between $a$, $\epsilon$, and $\eta$ 
\begin{equation}
\epsilon <\frac{\eta^2 \rho}{4},\quad \text{and}\quad \eta^2 r = a\ . 
\end{equation}
With the above setup, we have 
\begin{equation}
  \Amap(\hat{X}) -b
     =0 \quad \text{and}\quad 
    C - \Amap^* y = \epsilon \begin{bmatrix}
        0 & 0 \\
        0 & -1
    \end{bmatrix} \ .
\end{equation}

Note that the first optimality condition of \eqref{eq: ALmin-BM}, $\nabla _F \bar{\mathcal{L}}_\rho (\hat{F},y)=0$, is satisfied:
\begin{equation}\label{eq: BM-nonexample-non-local-foc}
   \nabla _F \bar{\mathcal{L}}_\rho (\hat{F},y)= 2[C-\Amap^*(y +\rho(b-\Amap(\hat{X})))]\hat{F} = 2[C-\Amap^*y]\hat{F}  =0\ .
\end{equation}

Let us show that $\hat{F}$ is in fact a local minimizer of \eqref{eq: ALmin-BM}. 
Note this implies that $\hat{F}$ also satisfies the second order necessary condition, $\nabla ^2_F \bar{\mathcal{L}}_\rho (\hat{F},y)\succeq 0$. 
The objective difference between $\bar{\mathcal{L}}_\rho (F,y)$ for an arbitrary $F=\hat{F} +\Delta$ and $\bar{\mathcal{L}}_\rho (\hat{F},y)$ is 
\begin{equation}
    \begin{aligned}
  \bar{\mathcal{L}}_\rho (F,y) - \bar{\mathcal{L}}_\rho (\hat{F},y) 
\overset{(a)}{=}   \inprd{C-\Amap^* y}{\Delta \Delta ^\top} + \frac{\rho}{2} \twonorm{\Amap(\hat{F}\Delta ^\top  + \Delta \hat{F}^\top+\Delta \Delta^\top)}^2\ .
    \end{aligned} 
\end{equation}
Here, the step $(a)$ is due to the step $(a)$ in \eqref{eq: Al_diff} and the first-order condition \eqref{eq: BM-nonexample-non-local-foc}.
Since 
\begin{equation}
\twonorm{\Amap(\hat{F}\Delta ^\top  + \Delta \hat{F}^\top+\Delta \Delta^\top)}^2 = 
p_2^2 + \twonorm{(\eta I_r + \Delta_r)u_n}^2 + \sum_{i=r+1}^{n-1} (u_{i}^\top u_n)^2\ , 
\end{equation}
where $\Delta_r = \begin{bmatrix}
    u_1,\;\cdots,\; u_r 
\end{bmatrix}^\top $
and $p_2^2$ is a nonnegative term. Thus, for any $\opnorm{\Delta}\leq \eta (1-\frac{1}{\sqrt{2}})$ and $\epsilon<\frac{\rho\eta^2}{4}$, we see 
\begin{equation}
    \begin{aligned}
 \bar{\mathcal{L}}_\rho (F,y) - \bar{\mathcal{L}}_\rho (\hat{F},y) 
 \geq  &  -\epsilon \twonorm{u_n}^2 + 
 \frac{\rho\eta^2}{4} \twonorm{u_n}^2 \geq 0\ .
    \end{aligned} 
\end{equation}
Hence, we see that $\hat{F}$ is a local minimizer within the radius $\opnorm{\Delta}\leq \eta\left(1-\frac{1}{\sqrt{2}}\right)$. It is not globally optimal since the KKT condition for $\hat{X}$ is not satisfied as $C-\Amap^*(y +\rho(b-\Amap(\hat{X})))\not\succeq 0$.
\end{proof}

Our second example shows that in general, the problem \eqref{eq: ALmin-BM} with an optimal $\ysol$ satisfying strict complementarity may still admit a spurious local minimizer for any $r\leq n-1$, even if the optimal solution to \eqref{eq: sdp.p} is rank $1$. Thus, to solve \eqref{eq: ALmin-BM}, methods should start from a point close to a global optimal point of \eqref{eq: ALmin-BM} to avoid being trapped by local minimizers.

\begin{prop}[Necessity of primal localization]
For any $n$ and $r\leq n-1$, there is some $C$, $\Amap$, $b$, such that \eqref{eq: sdp.p} is primal simple with a rank one solution. And for this SDP, for any $\rho>0$, there is an optimal strict complementarity dual solution $\ysol$  such that \eqref{eq: ALmin-BM} with $\ysol$ admits a spurious local minimizer with rank $r$. Note that  \eqref{eq: ALmin} with $\ysol$, in this case, admits $\Xsol$ as its unique optimal solution as well. 
\end{prop}
\begin{proof}
Let us now specify $\Amap$, $C$, and $b$. We set $C=0$. We set $\Amap$ in the following way:
\begin{equation}\label{eq: ALM-BM-non-example-A-C-2}
    A_i = \frac{1}{2}\left(e_ie_\dm^\top + e_\dm e_i^\top \right), \; i=1,\dots, \dm-1,\; A_\dm = \epsilon \begin{bmatrix}
        I_{\dm-1} &  0 \\
        0 & 1
    \end{bmatrix}, \;\text{and}\; 
    A_{\dm+1} = \epsilon \begin{bmatrix}
        2I_{\dm-1} & 0 \\
        0 & 1
    \end{bmatrix},
\end{equation}
where $\epsilon>0$ is a number to be determined later. We set $b = \epsilon a \begin{bmatrix}
    0_{\dm-1} \\ 1 \\ 1
\end{bmatrix} \in \RR^{\dm+1}$ for some $a>0$. 
One can verify that $\Xsol = \begin{bmatrix}
    0_{(\dm-1)\times (\dm-1)} & 0_{(\dm-1)\times 1} \\
    0_{1 \times (\dm-1)} & a
\end{bmatrix}$ is the unique rank one solution, and $\ysol = \begin{bmatrix}
    0_{(\dm-1)\times 1} &
    \rho  & -\rho
\end{bmatrix}$ is a dual optimal solution satisfying strict complementarity. Now consider \eqref{eq: ALmin-BM} with $r\geq 1$, $y = \ysol$ and $\hat{F} = \eta \begin{bmatrix}
    I_r \\ 0 
\end{bmatrix}$ for some $\eta >0$ to be determined later. Let $\hat{X} = \hat{F}\hat{F}^\top$. 
 We require the following relationship between $a$, $\epsilon$, and $\eta$ 
\begin{equation}\label{eq: eps_eta_r_a_choices}
\epsilon< \frac{\eta^2}{8},\quad  \epsilon(5\eta^2 r - 3a) = - 1, \quad \text{and}\quad  
\epsilon(3\eta^2 r - 2a) = - 2,
\end{equation}
which can be satisfied for any $\epsilon <\frac{1}{\sqrt{2r}}$, $\eta ^2= \frac{4}{\epsilon r}$, and $a = \frac{7}{\epsilon}$.
With this setup, we have  
\begin{equation}
\begin{aligned}
    C - \Amap^*\ysol  + \rho \Amap^*(\Amap(\hat{X})-b) 
    & = \epsilon \begin{bmatrix}
        \rho I_{\dm-1} & 0 \\
        0 & 0
    \end{bmatrix} + \rho \epsilon^2 (\eta ^2 r -a)I + \rho 
    \epsilon^2 (2\eta ^2 r -a)
    \begin{bmatrix}
        2I_{r-1} & 0 \\
        0 & 1 
    \end{bmatrix}
    \\
    & = 
    \epsilon \rho 
\begin{bmatrix}
    0_{(\dm-1)\times (\dm-1)} & 0_{(\dm-1)\times 1} \\
    0_{1 \times (\dm-1)} & -2
\end{bmatrix}\ .
\end{aligned}
\end{equation}
Thus, the first optimality condition of \eqref{eq: ALmin-BM}, 
    $[C-\Amap^*(\ysol +\rho(b-\Amap(\hat{X})))]\hat{F} =0$, is satisfied,

Let us show that $\hat{F}$ is in fact a local minimizer. The objective difference between $\bar{\mathcal{L}}_\rho (F,y)$ for an arbitrary $F=\hat{F} +\Delta$ and $\bar{\mathcal{L}}_\rho (\hat{F},y)$ is 
\begin{equation}
    \begin{aligned}
  \bar{\mathcal{L}}_\rho (F,\ysol) - \bar{\mathcal{L}}_\rho (\hat{F},\ysol) 
 =  \inprd{C-\Amap^*(\ysol+\rho(b-\Amap(\hat{X})))}{\Delta \Delta ^\top} + \frac{\rho}{2} \twonorm{\Amap(\hat{F}\Delta ^\top  + \Delta \hat{F}^\top+\Delta \Delta^\top)}^2\ .
    \end{aligned} 
\end{equation}
Since 
$\twonorm{\Amap(\hat{F}\Delta ^\top  + \Delta \hat{F}^\top+\Delta \Delta^\top)}^2 = 
p_2^2 + \twonorm{(\eta I_r + \Delta_r)u_\dm}^2 + \sum_{i=r+1}^{\dm-1} (u_{i}^\top u_\dm)^2$, where $\Delta_r = \begin{bmatrix}
    u_1,\cdots, \;u_r
\end{bmatrix}^\top $ and $p_2^2$ is a nonnegative term. Thus, for any $\opnorm{\Delta}\leq \eta\left(1-\frac{1}{\sqrt{2}}\right)$, we have 
\begin{equation}
    \begin{aligned}
  \bar{\mathcal{L}}_\rho (F,\ysol) - \bar{\mathcal{L}}_\rho (\hat{F},\ysol) 
 \geq    -2\rho\epsilon\twonorm{u_\dm}^2 + 
 \frac{\rho\eta^2}{4} \twonorm{u_\dm}^2 \overset{(a)}{\geq} 0,
    \end{aligned} 
\end{equation}
where the step $(a)$ is due to our choices of $\eta$ and $\epsilon$ in \eqref{eq: eps_eta_r_a_choices}. Hence, the point $\hat{F}$ is a local minimizer within the radius $\opnorm{\Delta}\leq \eta( 1-\frac{1}{\sqrt{2}})$.
\end{proof}


\section{ALORA: a rank-adaptive augmented Lagrangian method}\label{sec:alora}

The preceding sections establish favorable theoretical properties of the augmented Lagrangian method and the Burer–Monteiro formulation. In particular, we demonstrate the regularity of augmented Lagrangian subproblems in the local regime and highlight the effectiveness of first-order methods, such as gradient descent, in solving these subproblems.

Building on our theoretical insights and recent advances in augmented Lagrangian methods~\cite{burer2003nonlinear,wang2023solving,monteiro2024low}, we propose ALORA (Augmented Lagrangian Optimizer with Rank Adaptation) for solving the semidefinite program~\eqref{eq: sdp.p}. The development of ALORA is motivated by two key considerations. First, our analysis reveals that the augmented Lagrangian subproblems exhibit a favorable local geometry that first-order methods, such as gradient descent, can effectively exploit. Second, the algorithm is designed to utilize the computational advantages of modern GPUs. These perspectives, grounded in both theoretical understanding and hardware-aware implementation, guide the formulation of the method. A key difficulty, however, arises from the nonconvexity introduced by the Burer–Monteiro formulation. To address this, ALORA combines the structure of the augmented Lagrangian method with a rank adaptation strategy, allowing it to navigate the nonconvex landscape while maintaining compatibility with GPU-based computation. The full procedure is described in Algorithm~\ref{alg:alora}.

\begin{algorithm}[ht!]
        \SetAlgoLined
        {\bf Input:} tolerance $\epsilon$, penalty parameter $\rho$ \;
        {\bf initialize.} $y\gets 0$\;
        \Repeat{$\mathrm{KKT}(FF^\top,y)\leq \epsilon$}{
            $F\gets \argmin_F L_\rho(FF^\top,y)$ {by running a first-order gradient descent based procedure}\;
            $V\gets$ the $k$ eigenvectors of $\nabla_X L_\rho(X,y)\mid_{X=FF^\top}$
            corresponding to the smallest eigenvalues
            by a Lanczos algorithm\;
            \If{$\lambda_{\operatorname{min}}(\nabla_X L_\rho(X,y)\mid_{X=FF^\top})$ is sufficiently negative}{
                $S\gets \argmin_{S\succeq 0} L_\rho(FF^\top+VSV^\top,y)$\;
                $F\gets \begin{bmatrix}
                    F & VS^{1/2}
                \end{bmatrix}$\;         
            }
        $y\gets y+\rho(b-\mathcal AFF^\top)$\;
        }
        {\bf Output:} $FF^\top$.
        \caption{Augmented Lagrangian Optimizer with Rank Adaptation (ALORA)}
        \label{alg:alora}
\end{algorithm}

ALORA builds upon the augmented Lagrangian method combined with the Burer–Monteiro factorization. In each iteration, the algorithm begins by approximately solving the nonconvex augmented Lagrangian subproblem. {As discussed in Section~\ref{sec:gd}, first-order methods such as gradient descent are well-suited for this task, particularly in the local regime where the subproblem exhibits favorable geometry.}

Given a solution $F$ to the current subproblem, ALORA analyzes the top spectrum of the augmented Lagrangian gradient $\nabla_X L_\rho(X,y)\mid_{X=FF^\top}$. Specifically, a small number of eigenvectors corresponding to the most negative eigenvalues are extracted. Rank adaptation is then performed based on the spectral information: if the minimum eigenvalue is sufficiently negative, ALORA solves a small-scale semidefinite subproblem to determine the appropriate ``step size'' in the directions of negative curvature. The factor matrix $F$ is then expanded accordingly. This is followed by an update of the dual variable $y$, completing one iteration of the algorithm.

ALORA incorporates several enhancements over the standard augmented Lagrangian method with the Burer–Monteiro factorization. From an algorithmic perspective, these modifications are largely designed to overcome the difficulty of nonconvexity arising from Burer–Monteiro factorization. On the other hand, using first-order methods to solve ALM-BM subproblems has motivations from both our theoretical understanding and the computational structure of GPUs. We discuss details below:
\begin{itemize} 
    \item  {\textbf{Solving nonconvex ALM-BM subproblems by gradient descent} (lines 4 in Algorithm~\ref{alg:alora}): Motivated by the favorable local regularity properties established in our theoretical analysis, we employ a gradient descent variant to solve the ALM–BM subproblems. While the guarantees are primarily local, we find that gradient descent performs well in practice. Its simplicity and compatibility with GPU acceleration make it a natural choice for large-scale settings.}

    \item \textbf{Adaptive rank updates via spectral information} (lines 5–9 in Algorithm~\ref{alg:alora}): ALORA adaptively adjusts the rank by computing the eigenvectors corresponding to the \emph{$k$ most negative eigenvalues} of the gradient $\nabla_X L_\rho(X,y)\mid_{X=FF^\top}$. This strategy helps mitigate the nonconvexity challenges introduced by the Burer–Monteiro formulation and promotes improved global convergence. Similar spectral-based rank adjustment techniques, though often restricted to \emph{one eigenvector}, have been explored in prior work~\cite{yurtsever2021scalable}, and further explored in~\cite{wang2023solving} for manifold optimization and in~\cite{monteiro2024low} for trace-constrained SDPs.

    \item \textbf{Exploring negative curvature directions} (line 7 in Algorithm~\ref{alg:alora}): When negative curvature is detected, ALORA solves an additional low-dimensional semidefinite subproblem. This procedure is computationally efficient due to the small size of the subproblem and is motivated by previous approaches~\cite{ding2023revisiting}, with adaptations to suit our setting. The goal is to encode the spectral information of $\nabla_X L_\rho(X,y)\mid_{X=FF^\top}$ into the factor matrix $F$, using a stepsize matrix $S$ that ensures a decrease in the augmented Lagrangian value.

    \item {\textbf{GPU implementation}: ALORA is particularly amenable to GPU acceleration due to its reliance on highly parallelizable linear algebra operations. The pivotal computational kernels in ALORA, such as sparse matrix-matrix multiplication (SpMM) for gradient evaluation and dense matrix-matrix multiplication (GEMM) for adaptive rank adjustment, among others, are all operations that map efficiently to modern GPU architectures. These routines dominate the cost of computing gradients, solving the augmented Lagrangian subproblem and the rank adaptation step. Therefore, in this section, a practical variant of ALORA (referred to simply as ALORA) is implemented on GPUs to demonstrate its scalability and efficiency. }

     \item {\textbf{Implementation of Lanczos algorithm} (line 5 in Algorithm~\ref{alg:alora}): As a byproduct of the GPU-implementation of ALORA, an efficient implementation of the thick-restart Lanczos algorithm~\cite{wu2000thick} on GPUs is also provided to compute the $k$-smallest eigenvectors of a sparse symmetric matrix.}
\end{itemize}
These enhancements enable ALORA to dynamically adapt the rank of the factorized solution $F$, thereby improving robustness and promoting better global convergence behavior. We emphasize that the theoretical results in earlier sections provide local guarantees for
fixed-rank ALM-BM subproblems and serve as motivation for the design choices in ALORA. A full convergence analysis of ALORA, including
the interaction between rank updates, negative-curvature exploration, and
inexact subproblem solves, is an important direction for future work.

\subsection{Numerical experiments}\label{sec:gpu}
In this section, we evaluate the numerical performance of ALORA on two classes of large-scale semidefinite programs arising from MaxCut and matrix completion problems.

\textbf{Computing environments.} ALORA is implemented on GPUs in the Julia programming language~\cite{bezanson2017julia}, with GPU acceleration handled via CUDA.jl~\cite{besard2018effective}. Experiments are conducted using an NVIDIA-A100-80GB-PCIe GPU with CUDA version 12.4, deployed on a computing cluster equipped with an Intel Xeon Silver 4316 CPU.

\textbf{Termination criteria.}
ALORA terminates when relative KKT error of \eqref{eq: sdp.p} is no greater than the termination tolerance $\epsilon\in(0,\infty)$, defined as:
\begin{equation}
    \mathrm{KKT}(X,y):=\max\left\{\frac{\|\mathcal AX-b\|_2}{1+\|b\|_2}, \frac{|\min\{0,\lambda_{\text{min}}(C-\mathcal A^*y)\}|}{1+\|C\|_F}, \frac{|\langle C,X\rangle-b^\top y|}{1+|\langle C,X\rangle|+|b^\top y|}\right\} \leq \epsilon \ ,
\end{equation}
where $\lambda_{\text{min}}(\cdot)$ is the minimum eigenvalue. We set $\epsilon=10^{-5}$ for all experiments in this section. Notably, we use the $\ell_2$ norm in the relative terms in primal and dual infeasibility measures, resulting in a significantly stricter termination criterion, often by several orders of magnitude, compared to the $\ell_1$ norm-based criteria adopted in prior work~\cite{han2024accelerating}.

\subsubsection{MaxCut}
The MaxCut problem is a classical combinatorial optimization task with a well-known SDP relaxation. Given an undirected graph $G=(V,\;E)$ with weights $\omega_{ij}\geq 0$ on the edges, the MaxCut problem seeks a partition of the vertex set $V$ into two disjoint subsets that maximizes the total weight of edges crossing the partition. The standard SDP relaxation of MaxCut is formulated as
\begin{equation*}
    \begin{aligned}
\max_{X \in \mathbb{R}^{n \times n}} \quad & \frac{1}{4} \langle L, X \rangle \\
\text{s.t.} \quad & X_{ii} = 1 \quad \text{for all } i \in [n], \\
& X \succeq 0 \ ,
\end{aligned}
\end{equation*}
where the Laplacian matrix \( L \in \mathbb{R}^{n \times n} \) is defined by $L_{ij} = 
\begin{cases}
-\omega_{ij}, & \text{if } (i,j)\in E, \\
\sum_{k} \omega_{ik}, & \text{if } i = j\\
0, & \text{otherwise}
\end{cases}$.

For our experiments, we use benchmark graphs from the DIMACS10 collection~\cite{davis2011university}, which consists of various types of graphs from the 10th DIMACS Implementation Challenge on graph partitioning and graph clustering. We filter and select the large-scale instances by retaining only graphs with more than 1,000,000 vertices. For each selected graph, we construct the corresponding Laplacian matrix $L$, which is then used to form the SDP relaxation.

\begin{table}[ht!]
\centering
\begin{minipage}{0.45\textwidth}
\centering
\begin{tabular}{lrr}
\toprule
instance & dimension & time \\
\midrule
\texttt{NACA0015} & 1039183 & 11.485 \\
\texttt{delaunay\_n20} & 1048576 & 7.333 \\
\texttt{kron\_g500-logn20} & 1048576 & 131.911 \\
\texttt{rgg\_n\_2\_20\_s0} & 1048576 & 9.888 \\
\texttt{belgium\_osm} & 1441295 & 8.283 \\
\texttt{delaunay\_n21} & 2097152 & 15.262 \\
\texttt{kron\_g500-logn21} & 2097152 & 335.199 \\
\texttt{rgg\_n\_2\_21\_s0} & 2097152 & 21.679 \\
\texttt{packing-500x100x100-b050} & 2145852 & 26.731 \\
\texttt{netherlands\_osm} & 2216688 & 12.081 \\
\texttt{M6} & 3501776 & 46.567 \\
\texttt{333SP} & 3712815 & 40.248 \\
\texttt{AS365} & 3799275 & 49.985 \\
\texttt{venturiLevel3} & 4026819 & 41.529 \\
\texttt{NLR} & 4163763 & 51.739 \\
\texttt{delaunay\_n22} & 4194304 & 33.287 \\
\texttt{rgg\_n\_2\_22\_s0} & 4194304 & 47.507 \\
\texttt{hugetrace-00000} & 4588484 & 50.669 \\
\texttt{channel-500x100x100-b050} & 4802000 & 80.858 \\
\bottomrule
\end{tabular}
\end{minipage}
\hfill
\begin{minipage}{0.45\textwidth}
\centering
\begin{tabular}{lrr}
\toprule
instance & dimension & time \\
\midrule
\texttt{hugetric-00000} & 5824554 & 77.006 \\
\texttt{hugetric-00010} & 6592765 & 79.448 \\
\texttt{italy\_osm} & 6686493 & 40.709 \\
\texttt{adaptive} & 6815744 & 96.847 \\
\texttt{hugetric-00020} & 7122792 & 83.734 \\
\texttt{great-britain\_osm} & 7733822 & 61.270 \\
\texttt{delaunay\_n23} & 8388608 & 67.985 \\
\texttt{rgg\_n\_2\_23\_s0} & 8388608 & 164.510 \\
\texttt{germany\_osm} & 11548845 & 94.783 \\
\texttt{asia\_osm} & 11950757 & 108.128 \\
\texttt{hugetrace-00010} & 12057441 & 186.480 \\
\texttt{road\_central} & 14081816 & 174.989 \\
\texttt{hugetrace-00020} & 16002413 & 314.768 \\
\texttt{delaunay\_n24} & 16777216 & 189.024 \\
\texttt{rgg\_n\_2\_24\_s0} & 16777216 & 412.733 \\
\texttt{hugebubbles-00000} & 18318143 & 387.761 \\
\texttt{hugebubbles-00010} & 19458087 & 280.580 \\
\texttt{hugebubbles-00020} & 21198119 & 308.113 \\
\texttt{road\_usa} & 23947347 & 308.473 \\
\bottomrule
\end{tabular}
\end{minipage}
\caption{Performance of ALORA on MaxCut instances. Solve time in seconds.}
\label{tab:maxcut}
\end{table}

Table \ref{tab:maxcut} reports the performance of ALORA on large-scale MaxCut instances. We observe that the solver consistently handles graphs with millions of vertices, with solving times slightly increasing with problem size. For graphs with around 1–5 million vertices, ALORA typically completes within a minute. For larger instances exceeding 10 million vertices, solving times are still within several minutes, even for the largest graph \texttt{road\_usa} with approximately 24 million vertices. Overall, the results demonstrate the scalability and efficiency of ALORA across a diverse set of challenging graphs.

\subsubsection{Matrix completion}
In the matrix completion experiments, we formulate the recovery problem as a semidefinite program based on nuclear norm minimization as follows
\begin{equation*}
\begin{aligned}
    \min_{Y} \quad & \|Y\|_*\\ 
    \text{s.t.} \quad & Y_{ij} = M_{ij}, \quad \forall (i,j) \in \Omega \ ,
\end{aligned}  
\end{equation*}
which can be equivalently written as the following SDP:
\begin{equation*}
    \begin{aligned}
        \min_{X\in\mathbb R^{n\times n}} \quad & \frac{1}{2} \operatorname{tr}(X)\\
        \text{s.t.} \quad &
X = \begin{pmatrix} W_1 & Y \\ Y^\top & W_2 \end{pmatrix} \succeq 0, \quad Y_{ij} = M_{ij}, \quad \forall (i,j) \in \Omega \ ,
    \end{aligned}
\end{equation*}
Instances are generated using the same procedure as \cite{monteiro2024low}. More specifically, given rank $\rstar\leq \min(n_1,n_2)$, the ground truth matrix $M$ is generated as $M=UV^\top$, where $U\in\mathbb R^{n_1\times \rstar}$ and $V\in\mathbb R^{n_2\times \rstar}$ have independent standard Gaussian entries, and denote $n=n_1+n_2$ the matrix dimension. Observations are sampled uniformly at random, with the number of observed entries set to $m = \left\lceil \gamma \rstar (n - \rstar) \right\rceil$ where the oversampling ratio is given by $\gamma = \rstar \log(n)$.

\begin{table}[ht!]
\centering
\begin{tabular}{ccc||ccc}
\toprule
 \multicolumn{3}{c||}{$r = 3$} & \multicolumn{3}{c}{$r = 5$} \\
 $n$ & $m$ & time & $n$ & $m$ & time \\
\midrule
10000  & 828659    & 0.379  & 10000 & 2300917   & 0.504  \\
20000  & 1782453   & 0.726  & 20000 & 4952616   & 0.969  \\
50000  & 4868248   & 1.409  & 50000 & 13522024  & 3.025  \\
100000 & 10361604  & 2.944  & 100000 & 28777073  & 6.762  \\
200000 & 21961921  & 5.511  & 200000 & 61013229  & 14.786 \\
350000 & 40223331  & 12.759 & 350000 & 111700922 & 28.282 \\
\bottomrule
\end{tabular}
\caption{Performance of ALORA on matrix completion with varying $r$, $n$, and $m$. Solve time in seconds.}
\label{tab:mc}
\end{table}

Table \ref{tab:mc} demonstrates the efficiency of ALORA in solving matrix completion problems. The results show that instances with millions of observations can be solved within just a few seconds. For example, problems with approximately $5\times10^6$ constraints are completed in around one second. Even for large-scale instances with over $10^7$ observed entries, ALORA solves them in less than a minute. These results highlight that ALORA can reliably handle large-scale matrix completion problems involving millions to tens of millions of constraints at high speed.

\section{Conclusions}\label{sec: conlusion}

Despite the notable empirical success of the ALM-BM framework for solving large-scale low-rank SDPs, the ALM subproblems have largely lacked rigorous theoretical understanding concerning their structural inheritance and efficient solvability by first-order methods. Our theoretical results aim to bridge this gap between theory and computation. We rigorously demonstrate that, contingent upon the regularity of the original SDP, termed as primal simplicity, ALM subproblems \eqref{eq: ALmin} inherit crucial properties such as low-rankness and strict complementarity when the dual variable is localized. Furthermore, we establish the quadratic growth condition of ALM subproblems, and demonstrate that nonconvex ALM-BM subproblems \eqref{eq: ALmin-BM} are amenable to global optimization via gradient descent, exhibiting linear convergence under conditions of local initialization and dual variable proximity. Crucially, through illustrative examples, we demonstrate that the local nature of these theoretical guarantees is not a restrictive limitation but rather an inherent and essential characteristic of the underlying problem structure, thereby underscoring the necessity of our assumptions.

Motivated by these theoretical understandings, we designed and implemented a new GPU-based SDP solver, ALORA. By leveraging the favorable structural properties, namely, guaranteed low-rank solutions and efficient local convergence, ALORA enhances the ALM–BM framework with adaptive rank updates and exploration of negative curvature directions. Its GPU-based implementation enables efficient handling of large-scale instances, demonstrating strong scalability and computational performance in solving large-scale low-rank SDPs.

\section*{Acknowledgement}
Lijun Ding thanks Mateo D\'{i}az for an intriguing discussion in the early stages of this project. Haihao Lu is supported by AFOSR Grant No. FA9550-24-1-0051 and ONR Grant No. N000142412735. Jinwen Yang is supported by AFOSR Grant No. FA9550-24-1-0051.

\bibliographystyle{amsplain}
\bibliography{ref-papers}

@article{monteiro2024low,
  title={A low-rank augmented Lagrangian method for large-scale semidefinite programming based on a hybrid convex-nonconvex approach},
  author={Monteiro, Renato DC and Sujanani, Arnesh and Cifuentes, Diego},
  journal={arXiv preprint arXiv:2401.12490},
  year={2024}
}

@article{chan2008constraint,
  title={Constraint nondegeneracy, strong regularity, and nonsingularity in semidefinite programming},
  author={Chan, Zi Xian and Sun, Defeng},
  journal={SIAM Journal on optimization},
  volume={19},
  number={1},
  pages={370--396},
  year={2008},
  publisher={SIAM}
}

@article{kang2024fast,
  title={Fast and certifiable trajectory optimization},
  author={Kang, Shucheng and Xu, Xiaoyang and Sarva, Jay and Liang, Ling and Yang, Heng},
  journal={arXiv preprint arXiv:2406.05846},
  year={2024}
}

@article{aguirre2025cuhallar,
  title={cuhallar: A gpu accelerated low-rank augmented lagrangian method for large-scale semidefinite programming},
  author={Aguirre, Jacob M and Cifuentes, Diego and Guigues, Vincent and Monteiro, Renato DC and Nascimento, Victor Hugo and Sujanani, Arnesh},
  journal={arXiv preprint arXiv:2505.13719},
  year={2025}
}

@inproceedings{ge2017no,
  title={No spurious local minima in nonconvex low rank problems: A unified geometric analysis},
  author={Ge, Rong and Jin, Chi and Zheng, Yi},
  booktitle={International Conference on Machine Learning},
  pages={1233--1242},
  year={2017},
  organization={PMLR}
}

@article{wang2023solving,
  title={Solving low-rank semidefinite programs via manifold optimization},
  author={Wang, Jie and Hu, Liangbing},
  journal={arXiv preprint arXiv:2303.01722},
  year={2023}
}

@article{tang2024feasible,
  title={A feasible method for general convex low-rank sdp problems},
  author={Tang, Tianyun and Toh, Kim-Chuan},
  journal={SIAM Journal on Optimization},
  volume={34},
  number={3},
  pages={2169--2200},
  year={2024},
  publisher={SIAM}
}

@article{bezanson2017julia,
  title={Julia: A fresh approach to numerical computing},
  author={Bezanson, Jeff and Edelman, Alan and Karpinski, Stefan and Shah, Viral B},
  journal={SIAM review},
  volume={59},
  number={1},
  pages={65--98},
  year={2017},
  publisher={SIAM}
}

@article{besard2018effective,
  title={Effective extensible programming: unleashing Julia on GPUs},
  author={Besard, Tim and Foket, Christophe and De Sutter, Bjorn},
  journal={IEEE Transactions on Parallel and Distributed Systems},
  volume={30},
  number={4},
  pages={827--841},
  year={2018},
  publisher={IEEE}
}

@book{boumal2023introduction,
  title={An introduction to optimization on smooth manifolds},
  author={Boumal, Nicolas},
  year={2023},
  publisher={Cambridge University Press}
}

@article{tang2024solving,
  title={Solving graph equipartition SDPs on an algebraic variety},
  author={Tang, Tianyun and Toh, Kim-Chuan},
  journal={Mathematical Programming},
  volume={204},
  number={1},
  pages={299--347},
  year={2024},
  publisher={Springer}
}

@article{tang2024feasible_qp,
  title={A feasible method for solving an SDP relaxation of the quadratic knapsack problem},
  author={Tang, Tianyun and Toh, Kim-Chuan},
  journal={Mathematics of Operations Research},
  volume={49},
  number={1},
  pages={19--39},
  year={2024},
  publisher={INFORMS}
}

@article{erdogdu2022convergence,
  title={Convergence rate of block-coordinate maximization Burer--Monteiro method for solving large SDPs},
  author={Erdogdu, Murat A and Ozdaglar, Asuman and Parrilo, Pablo A and Vanli, Nuri Denizcan},
  journal={Mathematical Programming},
  volume={195},
  number={1},
  pages={243--281},
  year={2022},
  publisher={Springer}
}

@article{boumal2015riemannian,
  title={A Riemannian low-rank method for optimization over semidefinite matrices with block-diagonal constraints},
  author={Boumal, Nicolas},
  journal={arXiv preprint arXiv:1506.00575},
  year={2015}
}

@inproceedings{absil2010optimization,
  title={Optimization on manifolds: Methods and applications},
  author={Absil, P-A and Mahony, Robert and Sepulchre, Rodolphe},
  booktitle={Recent Advances in Optimization and its Applications in Engineering: The 14th Belgian-French-German Conference on Optimization},
  pages={125--144},
  year={2010},
  organization={Springer}
}

@article{boumal2014manopt,
  title={Manopt, a Matlab toolbox for optimization on manifolds},
  author={Boumal, Nicolas and Mishra, Bamdev and Absil, P-A and Sepulchre, Rodolphe},
  journal={The Journal of Machine Learning Research},
  volume={15},
  number={1},
  pages={1455--1459},
  year={2014},
  publisher={JMLR. org}
}

@article{candes2012exact,
  title={Exact matrix completion via convex optimization},
  author={Candes, Emmanuel and Recht, Benjamin},
  journal={Communications of the ACM},
  volume={55},
  number={6},
  pages={111--119},
  year={2012},
  publisher={ACM New York, NY, USA}
}

@article{allen2017linear,
  title={Linear convergence of a frank-wolfe type algorithm over trace-norm balls},
  author={Allen-Zhu, Zeyuan and Hazan, Elad and Hu, Wei and Li, Yuanzhi},
  journal={Advances in neural information processing systems},
  volume={30},
  year={2017}
}

@article{davis2011university,
  title={The University of Florida sparse matrix collection},
  author={Davis, Timothy A and Hu, Yifan},
  journal={ACM Transactions on Mathematical Software (TOMS)},
  volume={38},
  number={1},
  pages={1--25},
  year={2011},
  publisher={ACM New York, NY, USA}
}

@article{burer2003nonlinear,
  title={A nonlinear programming algorithm for solving semidefinite programs via low-rank factorization},
  author={Burer, Samuel and Monteiro, Renato DC},
  journal={Mathematical programming},
  volume={95},
  number={2},
  pages={329--357},
  year={2003},
  publisher={Springer}
}

@article{burer2005local,
  title={Local minima and convergence in low-rank semidefinite programming},
  author={Burer, Samuel and Monteiro, Renato DC},
  journal={Mathematical programming},
  volume={103},
  number={3},
  pages={427--444},
  year={2005},
  publisher={Springer}
}

@article{nayakkankuppam1999conditioning,
  title={Conditioning of semidefinite programs.},
  author={Nayakkankuppam, Madhu V and Overton, Michael L},
  journal={Mathematical programming},
  volume={85},
  number={3},
  year={1999}
}

@article{wang2023decomposition,
  title={A decomposition augmented lagrangian method for low-rank semidefinite programming},
  author={Wang, Yifei and Deng, Kangkang and Liu, Haoyang and Wen, Zaiwen},
  journal={SIAM Journal on Optimization},
  volume={33},
  number={3},
  pages={1361--1390},
  year={2023},
  publisher={SIAM}
}

@article{ding2021simplicity,
  title={On the simplicity and conditioning of low rank semidefinite programs},
  author={Ding, Lijun and Udell, Madeleine},
  journal={SIAM Journal on Optimization},
  volume={31},
  number={4},
  pages={2614--2637},
  year={2021},
  publisher={SIAM}
}

@article{burer2006computational,
  title={Computational enhancements in low-rank semidefinite programming},
  author={Burer, Samuel and Choi, Changhui},
  journal={Optimisation Methods and Software},
  volume={21},
  number={3},
  pages={493--512},
  year={2006},
  publisher={Taylor \& Francis}
}

@article{zhao2010newton,
  title={A Newton-CG augmented Lagrangian method for semidefinite programming},
  author={Zhao, Xin-Yuan and Sun, Defeng and Toh, Kim-Chuan},
  journal={SIAM Journal on Optimization},
  volume={20},
  number={4},
  pages={1737--1765},
  year={2010},
  publisher={SIAM}
}

@article{yang2015sdpnal+,
  title={SDPNAL+: a majorized semismooth Newton-CG augmented Lagrangian method for semidefinite programming with nonnegative constraints},
  author={Yang, Liuqin and Sun, Defeng and Toh, Kim-Chuan},
  journal={Mathematical Programming Computation},
  volume={7},
  number={3},
  pages={331--366},
  year={2015},
  publisher={Springer}
}

@article{sun2020sdpnal+,
  title={SDPNAL+: A Matlab software for semidefinite programming with bound constraints (version 1.0)},
  author={Sun, Defeng and Toh, Kim-Chuan and Yuan, Yancheng and Zhao, Xin-Yuan},
  journal={Optimization Methods and Software},
  volume={35},
  number={1},
  pages={87--115},
  year={2020},
  publisher={Taylor \& Francis}
}

@article{yurtsever2021scalable,
  title={Scalable semidefinite programming},
  author={Yurtsever, Alp and Tropp, Joel A and Fercoq, Olivier and Udell, Madeleine and Cevher, Volkan},
  journal={SIAM Journal on Mathematics of Data Science},
  volume={3},
  number={1},
  pages={171--200},
  year={2021},
  publisher={SIAM}
}

@inproceedings{yurtsever2018conditional,
  title={A conditional gradient framework for composite convex minimization with applications to semidefinite programming},
  author={Yurtsever, Alp and Fercoq, Olivier and Locatello, Francesco and Cevher, Volkan},
  booktitle={International conference on machine learning},
  pages={5727--5736},
  year={2018},
  organization={PMLR}
}

@inproceedings{yurtsever2019conditional,
  title={A conditional-gradient-based augmented Lagrangian framework},
  author={Yurtsever, Alp and Fercoq, Olivier and Cevher, Volkan},
  booktitle={International Conference on Machine Learning},
  pages={7272--7281},
  year={2019},
  organization={PMLR}
}

@article{tropp2017practical,
  title={Practical sketching algorithms for low-rank matrix approximation},
  author={Tropp, Joel A and Yurtsever, Alp and Udell, Madeleine and Cevher, Volkan},
  journal={SIAM Journal on Matrix Analysis and Applications},
  volume={38},
  number={4},
  pages={1454--1485},
  year={2017},
  publisher={SIAM}
}

@article{ding2021optimal,
  title={An optimal-storage approach to semidefinite programming using approximate complementarity},
  author={Ding, Lijun and Yurtsever, Alp and Cevher, Volkan and Tropp, Joel A and Udell, Madeleine},
  journal={SIAM Journal on Optimization},
  volume={31},
  number={4},
  pages={2695--2725},
  year={2021},
  publisher={SIAM}
}

@article{han2024accelerating,
  title={Accelerating Low-Rank Factorization-Based Semidefinite Programming Algorithms on GPU},
  author={Han, Qiushi and Lin, Zhenwei and Liu, Hanwen and Chen, Caihua and Deng, Qi and Ge, Dongdong and Ye, Yinyu},
  journal={arXiv preprint arXiv:2407.15049},
  year={2024}
}

@article{han2024low,
  title={A Low-Rank ADMM Splitting Approach for Semidefinite Programming},
  author={Han, Qiushi and Li, Chenxi and Lin, Zhenwei and Chen, Caihua and Deng, Qi and Ge, Dongdong and Liu, Huikang and Ye, Yinyu},
  journal={arXiv preprint arXiv:2403.09133},
  year={2024}
}

@article{toh1999sdpt3,
  title={SDPT3—a MATLAB software package for semidefinite programming, version 1.3},
  author={Toh, Kim-Chuan and Todd, Michael J and T{\"u}t{\"u}nc{\"u}, Reha H},
  journal={Optimization methods and software},
  volume={11},
  number={1-4},
  pages={545--581},
  year={1999},
  publisher={Taylor \& Francis}
}

@article{tutuncu2003solving,
  title={Solving semidefinite-quadratic-linear programs using SDPT3},
  author={T{\"u}t{\"u}nc{\"u}, Reha H and Toh, Kim-Chuan and Todd, Michael J},
  journal={Mathematical programming},
  volume={95},
  pages={189--217},
  year={2003},
  publisher={Springer}
}

@article{sturm1999using,
  title={Using SeDuMi 1.02, a MATLAB toolbox for optimization over symmetric cones},
  author={Sturm, Jos F},
  journal={Optimization methods and software},
  volume={11},
  number={1-4},
  pages={625--653},
  year={1999},
  publisher={Taylor \& Francis}
}

@article{borchers1999csdp,
  title={CSDP, AC library for semidefinite programming},
  author={Borchers, Brian},
  journal={Optimization methods and Software},
  volume={11},
  number={1-4},
  pages={613--623},
  year={1999},
  publisher={Taylor \& Francis}
}

@article{goulart2024clarabel,
  title={Clarabel: An interior-point solver for conic programs with quadratic objectives},
  author={Goulart, Paul J and Chen, Yuwen},
  journal={arXiv preprint arXiv:2405.12762},
  year={2024}
}

@article{chen2024cuclarabel,
  title={CuClarabel: GPU Acceleration for a Conic Optimization Solver},
  author={Chen, Yuwen and Tse, Danny and Nobel, Parth and Goulart, Paul and Boyd, Stephen},
  journal={arXiv preprint arXiv:2412.19027},
  year={2024}
}

@article{o2016conic,
  title={Conic optimization via operator splitting and homogeneous self-dual embedding},
  author={O’donoghue, Brendan and Chu, Eric and Parikh, Neal and Boyd, Stephen},
  journal={Journal of Optimization Theory and Applications},
  volume={169},
  pages={1042--1068},
  year={2016},
  publisher={Springer}
}

@article{alizadeh1997complementarity,
  title={Complementarity and nondegeneracy in semidefinite programming},
  author={Alizadeh, Farid and Haeberly, Jean-Pierre A and Overton, Michael L},
  journal={Mathematical programming},
  volume={77},
  number={1},
  pages={111--128},
  year={1997},
  publisher={Springer}
}

@article{garstka2021cosmo,
  title={COSMO: A conic operator splitting method for convex conic problems},
  author={Garstka, Michael and Cannon, Mark and Goulart, Paul},
  journal={Journal of Optimization Theory and Applications},
  volume={190},
  number={3},
  pages={779--810},
  year={2021},
  publisher={Springer}
}

@article{o2021operator,
  title={Operator splitting for a homogeneous embedding of the linear complementarity problem},
  author={O'Donoghue, Brendan},
  journal={SIAM Journal on Optimization},
  volume={31},
  number={3},
  pages={1999--2023},
  year={2021},
  publisher={SIAM}
}

@article{souto2022exploiting,
  title={Exploiting low-rank structure in semidefinite programming by approximate operator splitting},
  author={Souto, Mario and Garcia, Joaquim D and Veiga, {\'A}lvaro},
  journal={Optimization},
  volume={71},
  number={1},
  pages={117--144},
  year={2022},
  publisher={Taylor \& Francis}
}

@article{waldspurger2020rank,
  title={Rank optimality for the Burer--Monteiro factorization},
  author={Waldspurger, Irene and Waters, Alden},
  journal={SIAM journal on Optimization},
  volume={30},
  number={3},
  pages={2577--2602},
  year={2020},
  publisher={SIAM}
}

@article{barvinok1995problems,
  title={Problems of distance geometry and convex properties of quadratic maps},
  author={Barvinok, Alexander I.},
  journal={Discrete \& Computational Geometry},
  volume={13},
  pages={189--202},
  year={1995},
  publisher={Springer}
}

@misc{ding2024sharpnesswellconditioningnonsmoothconvex,
      title={Sharpness and well-conditioning of nonsmooth convex formulations in statistical signal recovery}, 
      author={Lijun Ding and Alex L. Wang},
      year={2023},
      eprint={2307.06873},
      archivePrefix={arXiv},
      primaryClass={math.OC},
      url={https://arxiv.org/abs/2307.06873v1}, 
}

@article{pataki1998rank,
  title={On the rank of extreme matrices in semidefinite programs and the multiplicity of optimal eigenvalues},
  author={Pataki, G{\'a}bor},
  journal={Mathematics of operations research},
  volume={23},
  number={2},
  pages={339--358},
  year={1998},
  publisher={INFORMS}
}

@article{journee2010low,
  title={Low-rank optimization on the cone of positive semidefinite matrices},
  author={Journ{\'e}e, Michel and Bach, Francis and Absil, P-A and Sepulchre, Rodolphe},
  journal={SIAM Journal on Optimization},
  volume={20},
  number={5},
  pages={2327--2351},
  year={2010},
  publisher={SIAM}
}

@article{boumal2016non,
  title={The non-convex Burer-Monteiro approach works on smooth semidefinite programs},
  author={Boumal, Nicolas and Voroninski, Vlad and Bandeira, Afonso},
  journal={Advances in Neural Information Processing Systems},
  volume={29},
  year={2016}
}

@article{helmberg2000spectral,
  title={A spectral bundle method for semidefinite programming},
  author={Helmberg, Christoph and Rendl, Franz},
  journal={SIAM Journal on Optimization},
  volume={10},
  number={3},
  pages={673--696},
  year={2000},
  publisher={SIAM}
}

@article{ding2023revisiting,
  title={Revisiting spectral bundle methods: Primal-dual (sub) linear convergence rates},
  author={Ding, Lijun and Grimmer, Benjamin},
  journal={SIAM Journal on Optimization},
  volume={33},
  number={2},
  pages={1305--1332},
  year={2023},
  publisher={SIAM}
}

@article{lu2023cupdlp,
  title={cuPDLP. jl: A GPU implementation of restarted primal-dual hybrid gradient for linear programming in Julia},
  author={Lu, Haihao and Yang, Jinwen},
  journal={arXiv preprint arXiv:2311.12180},
  year={2023}
}

@article{lu2023cupdlpc,
  title={cuPDLP-C: A Strengthened Implementation of cuPDLP for Linear Programming by C language},
  author={Lu, Haihao and Yang, Jinwen and Hu, Haodong and Huangfu, Qi and Liu, Jinsong and Liu, Tianhao and Ye, Yinyu and Zhang, Chuwen and Ge, Dongdong},
  journal={arXiv preprint arXiv:2312.14832},
  year={2023}
}

@article{lu2023practical,
  title={A practical and optimal first-order method for large-scale convex quadratic programming},
  author={Lu, Haihao and Yang, Jinwen},
  journal={arXiv preprint arXiv:2311.07710},
  year={2023}
}

@article{chen2017efficient,
  title={An efficient inexact symmetric Gauss--Seidel based majorized ADMM for high-dimensional convex composite conic programming},
  author={Chen, Liang and Sun, Defeng and Toh, Kim-Chuan},
  journal={Mathematical Programming},
  volume={161},
  pages={237--270},
  year={2017},
  publisher={Springer}
}

@article{huang2024restarted,
  title={Restarted primal-dual hybrid conjugate gradient method for large-scale quadratic programming},
  author={Huang, Yicheng and Zhang, Wanyu and Li, Hongpei and Ge, Dongdong and Liu, Huikang and Ye, Yinyu},
  journal={arXiv preprint arXiv:2405.16160},
  year={2024}
}

@manual{mosek,
   author = "MOSEK ApS",
   title = "MOSEK Optimization Suite 11.0.20.",
   year = 2025,
   url = "https://docs.mosek.com/latest/intro/index.html"
}

@manual{copt,
   author = "Cardinal Operations",
   title = "Cardinal Optimizer (COPT) User Guide.",
   year = 2025,
   url = "https://guide.coap.online/copt/en-doc/"
}

@article{hestenes1969multiplier,
  title={Multiplier and gradient methods},
  author={Hestenes, Magnus R},
  journal={Journal of optimization theory and applications},
  volume={4},
  number={5},
  pages={303--320},
  year={1969},
  publisher={Springer}
}

@article{powell1969method,
  title={A method for nonlinear constraints in minimization problems},
  author={Powell, Michael JD},
  journal={Optimization},
  pages={283--298},
  year={1969},
  publisher={Academic Press}
}

@article{rockafellar1976augmented,
  title={Augmented Lagrangians and applications of the proximal point algorithm in convex programming},
  author={Rockafellar, R Tyrrell},
  journal={Mathematics of operations research},
  volume={1},
  number={2},
  pages={97--116},
  year={1976},
  publisher={INFORMS}
}

@article{rockafellar1973multiplier,
  title={The multiplier method of Hestenes and Powell applied to convex programming},
  author={Rockafellar, R Tyrell},
  journal={Journal of Optimization Theory and applications},
  volume={12},
  number={6},
  pages={555--562},
  year={1973},
  publisher={Springer}
}

@article{chen2015fast,
  title={Fast low-rank estimation by projected gradient descent: General statistical and algorithmic guarantees},
  author={Chen, Yudong and Wainwright, Martin J},
  journal={arXiv preprint arXiv:1509.03025},
  year={2015}
}

@inproceedings{bhojanapalli2018smoothed,
  title={Smoothed analysis for low-rank solutions to semidefinite programs in quadratic penalty form},
  author={Bhojanapalli, Srinadh and Boumal, Nicolas and Jain, Prateek and Netrapalli, Praneeth},
  booktitle={Conference on learning theory},
  pages={3243--3270},
  year={2018},
  organization={PMLR}
}

@article{yi2016fast,
  title={Fast algorithms for robust PCA via gradient descent},
  author={Yi, Xinyang and Park, Dohyung and Chen, Yudong and Caramanis, Constantine},
  journal={Advances in neural information processing systems},
  volume={29},
  year={2016}
}

@article{xiong2023over,
  title={How over-parameterization slows down gradient descent in matrix sensing: The curses of symmetry and initialization},
  author={Xiong, Nuoya and Ding, Lijun and Du, Simon S},
  journal={arXiv preprint arXiv:2310.01769},
  year={2023}
}

@article{du2018algorithmic,
  title={Algorithmic regularization in learning deep homogeneous models: Layers are automatically balanced},
  author={Du, Simon S and Hu, Wei and Lee, Jason D},
  journal={Advances in neural information processing systems},
  volume={31},
  year={2018}
}

@article{recht2010guaranteed,
  title={Guaranteed minimum-rank solutions of linear matrix equations via nuclear norm minimization},
  author={Recht, Benjamin and Fazel, Maryam and Parrilo, Pablo A},
  journal={SIAM review},
  volume={52},
  number={3},
  pages={471--501},
  year={2010},
  publisher={SIAM}
}

@article{yu2015useful,
  title={A useful variant of the Davis--Kahan theorem for statisticians},
  author={Yu, Yi and Wang, Tengyao and Samworth, Richard J},
  journal={Biometrika},
  volume={102},
  number={2},
  pages={315--323},
  year={2015},
  publisher={Oxford University Press}
}

@article{necoara2019linear,
  title={Linear convergence of first order methods for non-strongly convex optimization},
  author={Necoara, Ion and Nesterov, Yu and Glineur, Francois},
  journal={Mathematical Programming},
  volume={175},
  pages={69--107},
  year={2019},
  publisher={Springer}
}

@book{nesterov2018lectures,
  title={Lectures on convex optimization},
  author={Nesterov, Yurii},
  volume={137},
  publisher={Springer}
}

@article{luo1998superlinear,
  title={Superlinear convergence of a symmetric primal-dual path following algorithm for semidefinite programming},
  author={Luo, Zhi-Quan and Sturm, Jos F and Zhang, Shuzhong},
  journal={SIAM Journal on Optimization},
  volume={8},
  number={1},
  pages={59--81},
  year={1998},
  publisher={SIAM}
}

@article{cui2019r,
  title={On the R-superlinear convergence of the KKT residuals generated by the augmented Lagrangian method for convex composite conic programming},
  author={Cui, Ying and Sun, Defeng and Toh, Kim-Chuan},
  journal={Mathematical Programming},
  volume={178},
  pages={381--415},
  year={2019},
  publisher={Springer}
}

@article{drusvyatskiy2011generic,
  title={Generic nondegeneracy in convex optimization},
  author={Drusvyatskiy, Dmitriy and Lewis, Adrian S},
  journal={Proceedings of the American Mathematical Society},
  pages={2519--2527},
  year={2011},
  publisher={JSTOR}
}

@article{xu2021iteration,
  title={Iteration complexity of inexact augmented Lagrangian methods for constrained convex programming},
  author={Xu, Yangyang},
  journal={Mathematical Programming},
  volume={185},
  pages={199--244},
  year={2021},
  publisher={Springer}
}

@article{liao2024inexact,
  title={Inexact augmented Lagrangian methods for conic optimization: Quadratic growth and linear convergence},
  author={Liao, Feng-Yi and Ding, Lijun and Zheng, Yang},
  journal={Advances in Neural Information Processing Systems},
  volume={37},
  pages={41013--41050},
  year={2024}
}

@article{lan2016iteration,
  title={Iteration-complexity of first-order augmented Lagrangian methods for convex programming},
  author={Lan, Guanghui and Monteiro, Renato DC},
  journal={Mathematical Programming},
  volume={155},
  number={1},
  pages={511--547},
  year={2016},
  publisher={Springer}
}

@article{nedelcu2014computational,
  title={Computational complexity of inexact gradient augmented Lagrangian methods: application to constrained MPC},
  author={Nedelcu, Valentin and Necoara, Ion and Tran-Dinh, Quoc},
  journal={SIAM Journal on Control and Optimization},
  volume={52},
  number={5},
  pages={3109--3134},
  year={2014},
  publisher={SIAM}
}

@article{liu2019nonergodic,
  title={On the nonergodic convergence rate of an inexact augmented Lagrangian framework for composite convex programming},
  author={Liu, Ya-Feng and Liu, Xin and Ma, Shiqian},
  journal={Mathematics of Operations Research},
  volume={44},
  number={2},
  pages={632--650},
  year={2019},
  publisher={INFORMS}
}

@article{lu2023iteration,
  title={Iteration-complexity of first-order augmented Lagrangian methods for convex conic programming},
  author={Lu, Zhaosong and Zhou, Zirui},
  journal={SIAM journal on optimization},
  volume={33},
  number={2},
  pages={1159--1190},
  year={2023},
  publisher={SIAM}
}

@article{rockafellar1973dual,
  title={A dual approach to solving nonlinear programming problems by unconstrained optimization},
  author={Rockafellar, R Tyrrell},
  journal={Mathematical programming},
  volume={5},
  number={1},
  pages={354--373},
  year={1973},
  publisher={Springer}
}

@article{alizadeh1995interior,
  title={Interior point methods in semidefinite programming with applications to combinatorial optimization},
  author={Alizadeh, Farid},
  journal={SIAM journal on Optimization},
  volume={5},
  number={1},
  pages={13--51},
  year={1995},
  publisher={SIAM}
}

@article{helmberg1996interior,
  title={An interior-point method for semidefinite programming},
  author={Helmberg, Christoph and Rendl, Franz and Vanderbei, Robert J and Wolkowicz, Henry},
  journal={SIAM Journal on optimization},
  volume={6},
  number={2},
  pages={342--361},
  year={1996},
  publisher={SIAM}
}

@article{vandenberghe1996semidefinite,
  title={Semidefinite programming},
  author={Vandenberghe, Lieven and Boyd, Stephen},
  journal={SIAM review},
  volume={38},
  number={1},
  pages={49--95},
  year={1996},
  publisher={SIAM}
}

@article{todd1998nesterov,
  title={On the Nesterov--Todd Direction in Semidefinite Programming},
  author={Todd, Michael J and Toh, Kim-Chuan and T{\"u}t{\"u}nc{\"u}, Reha H},
  journal={SIAM Journal on Optimization},
  volume={8},
  number={3},
  pages={769--796},
  year={1998},
  publisher={SIAM}
}

@book{wolkowicz2012handbook,
  title={Handbook of semidefinite programming: theory, algorithms, and applications},
  author={Wolkowicz, Henry and Saigal, Romesh and Vandenberghe, Lieven},
  volume={27},
  year={2012},
  publisher={Springer Science \& Business Media}
}

@incollection{goemans2000combinatorial,
  title={Combinatorial optimization},
  author={Goemans, Michel and Rendl, Franz},
  booktitle={Handbook of Semidefinite Programming: Theory, Algorithms, and Applications},
  pages={343--360},
  year={2000},
  publisher={Springer}
}

@phdthesis{helmberg2000semidefinite,
  title={Semidefinite programming for combinatorial optimization},
  author={Helmberg, Christoph},
  year={2000}
}

@incollection{lovasz2003semidefinite,
  title={Semidefinite programs and combinatorial optimization},
  author={Lov{\'a}sz, L{\'a}szl{\'o}},
  booktitle={Recent advances in algorithms and combinatorics},
  pages={137--194},
  year={2003},
  publisher={Springer}
}

@article{benson2000solving,
  title={Solving large-scale sparse semidefinite programs for combinatorial optimization},
  author={Benson, Steven J and Ye, Yinyu and Zhang, Xiong},
  journal={SIAM Journal on Optimization},
  volume={10},
  number={2},
  pages={443--461},
  year={2000},
  publisher={SIAM}
}

@article{todd2001semidefinite,
  title={Semidefinite optimization},
  author={Todd, Michael J},
  journal={Acta Numerica},
  volume={10},
  pages={515--560},
  year={2001},
  publisher={Cambridge University Press}
}

@article{lavaei2011zero,
  title={Zero duality gap in optimal power flow problem},
  author={Lavaei, Javad and Low, Steven H},
  journal={IEEE Transactions on Power systems},
  volume={27},
  number={1},
  pages={92--107},
  year={2011},
  publisher={IEEE}
}

@article{low2014convex1,
  title={Convex relaxation of optimal power flow—Part I: Formulations and equivalence},
  author={Low, Steven H},
  journal={IEEE Transactions on Control of Network Systems},
  volume={1},
  number={1},
  pages={15--27},
  year={2014},
  publisher={IEEE}
}

@article{low2014convex2,
  title={Convex relaxation of optimal power flow—Part II: Exactness},
  author={Low, Steven H},
  journal={IEEE Transactions on Control of Network Systems},
  volume={1},
  number={2},
  pages={177--189},
  year={2014},
  publisher={IEEE}
}

@article{madani2014convex,
  title={Convex relaxation for optimal power flow problem: Mesh networks},
  author={Madani, Ramtin and Sojoudi, Somayeh and Lavaei, Javad},
  journal={IEEE Transactions on Power Systems},
  volume={30},
  number={1},
  pages={199--211},
  year={2014},
  publisher={IEEE}
}

@article{majumdar2020recent,
  title={Recent scalability improvements for semidefinite programming with applications in machine learning, control, and robotics},
  author={Majumdar, Anirudha and Hall, Georgina and Ahmadi, Amir Ali},
  journal={Annual Review of Control, Robotics, and Autonomous Systems},
  volume={3},
  number={1},
  pages={331--360},
  year={2020},
  publisher={Annual Reviews}
}

@inproceedings{ding2020spectral,
  title={Spectral Frank-Wolfe algorithm: Strict complementarity and linear convergence},
  author={Ding, Lijun and Fei, Yingjie and Xu, Qiantong and Yang, Chengrun},
  booktitle={International conference on machine learning},
  pages={2535--2544},
  year={2020},
  organization={PMLR}
}

@article{garber2025low,
  title={Low-rank extragradient methods for scalable semidefinite optimization},
  author={Garber, Dan and Kaplan, Atara},
  journal={Operations Research Letters},
  volume={60},
  pages={107230},
  year={2025},
  publisher={Elsevier}
}

@article{martin2023guarantees,
  title={Guarantees for data-driven control of nonlinear systems using semidefinite programming: A survey},
  author={Martin, Tim and Sch{\"o}n, Thomas B and Allg{\"o}wer, Frank},
  journal={Annual Reviews in Control},
  volume={56},
  pages={100911},
  year={2023},
  publisher={Elsevier}
}

@book{skrzypczyk2023semidefinite,
  title={Semidefinite programming in quantum information science},
  author={Skrzypczyk, Paul and Cavalcanti, Daniel},
  year={2023},
  publisher={IOP Publishing}
}

@article{cavalcanti2016quantum,
  title={Quantum steering: a review with focus on semidefinite programming},
  author={Cavalcanti, Daniel and Skrzypczyk, Paul},
  journal={Reports on Progress in Physics},
  volume={80},
  number={2},
  pages={024001},
  year={2016},
  publisher={IOP Publishing}
}

@article{mazziotti2011large,
  title={Large-scale semidefinite programming for many-electron quantum mechanics},
  author={Mazziotti, David A},
  journal={Physical review letters},
  volume={106},
  number={8},
  pages={083001},
  year={2011},
  publisher={APS}
}

@article{goemans1995improved,
  title={Improved approximation algorithms for maximum cut and satisfiability problems using semidefinite programming},
  author={Goemans, Michel X and Williamson, David P},
  journal={Journal of the ACM (JACM)},
  volume={42},
  number={6},
  pages={1115--1145},
  year={1995},
  publisher={ACM New York, NY, USA}
}

@article{lemon2016low,
  title={Low-rank semidefinite programming: Theory and applications},
  author={Lemon, Alex and So, Anthony Man-Cho and Ye, Yinyu and others},
  journal={Foundations and Trends{\textregistered} in Optimization},
  volume={2},
  number={1-2},
  pages={1--156},
  year={2016},
  publisher={Now Publishers, Inc.}
}

@article{korpelevich1977extragradient,
  title={Extragradient method for finding saddle points and other problems},
  author={Korpelevich, Galina M},
  journal={Matekon},
  volume={13},
  number={4},
  pages={35--49},
  year={1977},
  publisher={ME SHARPE INC 80 BUSINESS PARK DR, ARMONK, NY 10504}
}

@article{rebjock2024fast,
  title={Fast convergence to non-isolated minima: four equivalent conditions for c 2 functions},
  author={Rebjock, Quentin and Boumal, Nicolas},
  journal={Mathematical Programming},
  pages={1--49},
  year={2024},
  publisher={Springer}
}

@article{wang2025accelerated,
  title={Accelerated first-order methods for a class of semidefinite programs},
  author={Wang, Alex L and K{\i}l{\i}n{\c{c}}-Karzan, Fatma},
  journal={Mathematical Programming},
  volume={209},
  number={1},
  pages={503--556},
  year={2025},
  publisher={Springer}
}

@article{gabay1976dual,
  title={A dual algorithm for the solution of nonlinear variational problems via finite element approximation},
  author={Gabay, Daniel and Mercier, Bertrand},
  journal={Computers \& mathematics with applications},
  volume={2},
  number={1},
  pages={17--40},
  year={1976},
  publisher={Elsevier}
}

@article{chen2024hpr,
  title={HPR-LP: An implementation of an HPR method for solving linear programming},
  author={Chen, Kaihuang and Sun, Defeng and Yuan, Yancheng and Zhang, Guojun and Zhao, Xinyuan},
  journal={arXiv preprint arXiv:2408.12179},
  year={2024}
}

@article{boumal2020deterministic,
  title={Deterministic Guarantees for Burer-Monteiro Factorizations of Smooth Semidefinite Programs},
  author={Boumal, Nicolas and Voroninski, Vladislav and Bandeira, Afonso S},
  journal={Communications on Pure and Applied Mathematics},
  volume={73},
  number={3},
  pages={581--608},
  year={2020},
  publisher={Wiley Online Library}
}

@article{o2022burer,
  title={The Burer-Monteiro SDP method can fail even above the Barvinok-Pataki bound},
  author={O'Carroll, Liam and Srinivas, Vaidehi and Vijayaraghavan, Aravindan},
  journal={Advances in Neural Information Processing Systems},
  volume={35},
  pages={31254--31264},
  year={2022}
}

@article{chen2018harnessing,
  title={Harnessing structures in big data via guaranteed low-rank matrix estimation: Recent theory and fast algorithms via convex and nonconvex optimization},
  author={Chen, Yudong and Chi, Yuejie},
  journal={IEEE Signal Processing Magazine},
  volume={35},
  number={4},
  pages={14--31},
  year={2018},
  publisher={IEEE}
}

@article{zhang2021preconditioned,
  title={Preconditioned gradient descent for over-parameterized nonconvex matrix factorization},
  author={Zhang, Jialun and Fattahi, Salar and Zhang, Richard Y},
  journal={Advances in Neural Information Processing Systems},
  volume={34},
  pages={5985--5996},
  year={2021}
}

@article{zhuo2024computational,
  title={On the computational and statistical complexity of over-parameterized matrix sensing},
  author={Zhuo, Jiacheng and Kwon, Jeongyeol and Ho, Nhat and Caramanis, Constantine},
  journal={Journal of Machine Learning Research},
  volume={25},
  number={169},
  pages={1--47},
  year={2024}
}

@article{shin2024accelerating,
  title={Accelerating optimal power flow with {GPU}s: {SIMD} abstraction of nonlinear programs and condensed-space interior-point methods},
  author={Shin, Sungho and Anitescu, Mihai and Pacaud, Fran{\c{c}}ois},
  journal={Electric Power Systems Research},
  volume={236},
  pages={110651},
  year={2024},
  publisher={Elsevier}
}

@article{lin2025pdcs,
  title={PDCS: A Primal-Dual Large-Scale Conic Programming Solver with GPU Enhancements},
  author={Lin, Zhenwei and Xiong, Zikai and Ge, Dongdong and Ye, Yinyu},
  journal={arXiv preprint arXiv:2505.00311},
  year={2025}
}

@article{wu2000thick,
  title={Thick-restart Lanczos method for large symmetric eigenvalue problems},
  author={Wu, Kesheng and Simon, Horst},
  journal={SIAM Journal on Matrix Analysis and Applications},
  volume={22},
  number={2},
  pages={602--616},
  year={2000},
  publisher={SIAM}
}

\appendix

\section{Additional related work} \label{sec: add_rel_work}
\textbf{Additional related work on Low-rank structure of SDP.} Under suitable structural assumptions on the problem data, numerous theoretical results show that $\rstar$ can be constant concerning $m$ and $n$, see, for example, \cite{lemon2016low,candes2012exact,recht2010guaranteed,chen2015fast,chen2018harnessing,ding2021simplicity} among others. 

\textbf{Burer-Monteiro factorization and Riemannian optimization.} \textcolor{black}{When the feasible region of the BM formulation of the SDP forms a Riemannian manifold, a situation that occurs in many applications \cite{boumal2016non,boumal2020deterministic},  the techniques of Riemannian optimization \cite{absil2010optimization, boumal2016non,boumal2023introduction,boumal2014manopt} can be employed. Pioneered in \cite{journee2010low}, a series of work has explored the Riemannian optimization for solving SDP problems, including \cite{boumal2015riemannian,erdogdu2022convergence,tang2024solving,tang2024feasible_qp,tang2024feasible}, ranging from particular constraints to more generic constraints. Remarkably, the recent work \cite{tang2024feasible} addresses generic constraints and presents a local linear convergence rate result for a Riemannian gradient method under some assumptions on the structure of the original SDP and the algorithm iterates.
}

\textbf{Theoretical guarantees of general ALM.}
ALM was first introduced in~\cite{hestenes1969multiplier,powell1969method} and later in the seminal work~\cite{rockafellar1976augmented} for convex programming and is not limited to SDP. A strong connection between ALM and the proximal point method (PPM) is established in \cite{rockafellar1976augmented}.
Specifically, the dual iterates generated by ALM coincide with the proximal updates of the dual function in PPM. Under mild assumptions, global convergence of ALM is guaranteed when each subproblem is solved exactly, with convergence rates that are often linear~\cite{rockafellar1973multiplier,rockafellar1976augmented,cui2019r} under certain regular conditions. The convergence includes both the primal and dual sequences to optimal solutions, along with the asymptotic satisfaction of the KKT conditions. However, in practical large-scale settings, solving subproblems exactly can be computationally prohibitive. To address this, a substantial body of work has studied the behavior of inexact ALM, where subproblems are solved approximately under controlled error conditions, while still ensuring global convergence~\cite{rockafellar1973dual,rockafellar1976augmented,lan2016iteration,nedelcu2014computational,liu2019nonergodic,xu2021iteration,lu2023iteration,liao2024inexact}. \textcolor{black}{Typical analysis of the subproblem complexity relies on the worst-case complexity, e.g., the $O(\frac{1}{\sqrt{\epsilon}})$ complexity of accelerated gradient method for smooth convex optimization \cite{liu2019nonergodic,xu2021iteration}.
}

\textbf{SDP solvers.} Several scalable semidefinite programming solvers have been developed based on the augmented Lagrangian method and/or Burer-Monteiro factorization. Here we provide a brief discussion on a few of the most relevant solvers:
\begin{itemize}
    \item {SDPLR~\cite{burer2003nonlinear,burer2005local,burer2006computational}.} SDPLR is an augmented Lagrangian solver designed to efficiently handle large-scale semidefinite programs by combining the Burer–Monteiro factorization with several computational enhancements. ALM–BM subproblems are solved using limited-memory BFGS (L-BFGS) methods, with step sizes determined via exact line search. Additionally, the solver employs a dynamic rank update strategy based on LU factorizations to adaptively adjust the factorization rank during optimization.
    
    \item {SDPNAL+~\cite{zhao2010newton,yang2015sdpnal+,sun2020sdpnal+}.} SDPNAL+ is an SDP solver that builds on the augmented Lagrangian method and incorporates advanced second-order techniques. It applies a majorized semi-smooth Newton-CG method~\cite{zhao2010newton} to efficiently solve the inner ALM subproblems, enabling rapid convergence even in the presence of degeneracy or ill-conditioning. The solver operates in two phases: an inexact symmetric Gauss-Seidel-based ADMM~\cite{chen2017efficient} phase for warm starting, followed by a refinement phase with Newton-CG-based ALM. SDPNAL+ has demonstrated strong numerical performance across a wide range of SDP problems.
    
    \item {SketchyCGAL~\cite{yurtsever2021scalable}.} CGAL~\cite{yurtsever2018conditional,yurtsever2019conditional} is a first-order method designed to solve large-scale semidefinite programs within an augmented Lagrangian framework. It combines conditional gradient updates with dual ascent and leverages approximate eigenvector computations to maintain scalability. Building on this, a scalable SDP solver, dubbed SketchyCGAL, further introduces randomized sketching~\cite{tropp2017practical, ding2021optimal} to compress the iterates, enabling near-optimal low-rank approximations with significantly reduced storage and computation requirements. Together, these methods provide provably convergent and memory-efficient solvers for large-scale SDPs. \textcolor{black}{It is also shown that the method theoretically and empirically converges at a sublinear rate \cite{yurtsever2021scalable,tang2024feasible,monteiro2024low}.
    }
    
    \item {cuLoRADS~\cite{han2024accelerating}.} cuLoRADS is a recent GPU-accelerated solver for large-scale semidefinite programs, which combines the Burer–Monteiro factorization with a two-phase approach. The first phase employs an augmented Lagrangian method to solve the low-rank formulation, leveraging its robustness for early-stage optimization and warm starting. Once sufficient progress is made in reducing primal infeasibility, the algorithm switches to a second phase using an ADMM-based matrix-splitting strategy, dubbed LoRADS~\cite{han2024low}, for faster convergence. By exploiting GPU-parallelizable operations and designing efficient computation and memory patterns, cuLoRADS achieves significant scalability, solving several SDPs with matrix dimensions in the hundreds of millions in a matter of minutes. \textcolor{black}{Note that the second phase is based on an algorithmic framework different from ALM-BM and is considered the main ingredient of the method.}
\end{itemize}
In addition to ALM-based solvers, there exist many mature SDP solvers based on interior-point methods, including MOSEK~\cite{mosek}, COPT~\cite{copt}, SDPT3~\cite{toh1999sdpt3,tutuncu2003solving}, SeDuMi~\cite{sturm1999using}, CSDP~\cite{borchers1999csdp}, and Clarabel~\cite{goulart2024clarabel,chen2024cuclarabel}. There are also efficient solvers based on operator splitting techniques, such as SCS~\cite{o2016conic,o2021operator}, COSMO~\cite{garstka2021cosmo}, and ProxSDP~\cite{souto2022exploiting}.

\textbf{Other GPU-based solvers.}  Recently, there has been a rapidly growing trend in developing GPU-based solvers for mathematical programming, driven by the substantial computational power and parallelism offered by modern GPUs. Notable examples include solvers based on first-order methods such as \cite{lu2023cupdlp,lu2023cupdlpc,chen2024hpr} for linear programming, \cite{lu2023practical,huang2024restarted} for quadratic programming, \cite{lin2025pdcs,kang2024fast} for conic programming, as well as IPM-based solvers such as \cite{chen2024cuclarabel} for conic programming and \cite{shin2024accelerating} for nonlinear programming.

\section{Lemmas and remaining proofs for Section \ref{sec: analysis_AL-P}}\label{sec: pf_lem_auxiliary}

We shall prove Lemma \ref{lem: auxiliary} and Lemma \ref{lem: X_yrho_X_star_diff_dual_diff} in this section. We start with a lemma ensuring the map $\Amap$ is injective when restricting to certain subspace.

\begin{lem}\label{lem: Z_Z_star_A_U_inj}
Suppose the assumptions of Theorem \ref{thm: ALmin_primal_simple_qg} hold. Suppose a symmetric matrix $Z$ satisfies $\fronorm{Z - Z(\ysol)}\leq \frac{\sigma_{\min>0}(Z(\ysol))\sigma_{\min}(\Amap_{\Vstar})}{40\sqrt{2}\opnorm{\Amap}}$. Let $U\in \mathbb{R}^{n\times \rstar}$ be any matrix with orthonormal columns corresponding to the eigenspace of $Z$ associated with its smallest $\rstar$ eigenvalues. Then the map $\Amap_U: \symMat{\rstar} \rightarrow \mathbb{R}^m$ with $\Amap_U(S) = \Amap(USU^\top)$ is injective, and satisfies 
$\opnorm{\mathcal{A}_{\Vstar} - \mathcal{A}_U} \leq \frac{4\sqrt{2}\opnorm{\Amap}\fronorm{Z(\ysol)-Z}}{\sigma_{\min>0}(Z(\ysol))}$ and $\sigma_{\min}(\Amap_U)\geq \frac{9}{10}\sigma_{\min}(\Amap_{\Vstar})$.
\end{lem}
\begin{proof}
Recall $\rstar = \rank(\Xsol)$ and the definition of $U\in \real^{n\times \rstar}$, a matrix
consists of orthonormal eigenvectors of $Z(\zyp)$ that correspond to the smallest $\rstar$ eigenvalues of $Z(\zyp)$. Note that for any orthonormal $O\in \RR^{\rstar}$, there holds the equality
$
 \sigma_{\min}(\Amap_{U}) = \sigma_{\min}(\Amap_{UO})\ .
$
Let $\bar{O} \in \arg\min_{OO^\top = I_{\rstar}} \fronorm{\Vstar - UO}$. In the following, we use $U$ rather than $U\bar{O}$ to save some notations.

To show $\sigma_{\min}(\Amap_U)\geq \frac{9}{10}\sigma_{\min}(\Amap_{\Vstar})$, the main argument is that using $\sigma_{\min}(\Amap_{\Vstar})>0$, thanks to Lemma \ref{lem: kkt_like.qg}, and show $\Amap_V$ is close to $\Amap_{\Vstar}$ if $Z$ is close to $Z(\ysol)$.
Indeed, we have the following derivation for any $S\in \symMat{\rstar}$ with $\fronorm{S}=1$:
\begin{equation}\label{eq: AVAVstar}
    \begin{aligned}
\twonorm{(\Amap_{\Vstar}- \Amap_{U})(S)} & = 
\twonorm{(\Amap(\Vstar S\Vstar^\top - USU^\top)} \\ 
& \overset{(a)}{=}\twonorm{\Amap((\Vstar -U)S\Vstar^\top + US(U-\Vstar)^\top)} \\
& \leq \opnorm{\Amap}\left( \fronorm{(\Vstar -U)S\Vstar^\top} + \fronorm{US(U-\Vstar)^\top} \right)\\
& \overset{(b)}{\leq} 2\opnorm{\Amap}\opnorm{U-\Vstar}\ .
    \end{aligned}
\end{equation}
In the step $(a)$, we add and subtract $US\Vstar^\top$. In the step $(b)$, we use the fact that $\Vstar$ and $U$ have orthonormal columns and $\fronorm{S}=1$. To bound the distance $\opnorm{U-\Vstar}$, we use the Davis-Kahan Theorem \cite[Theorem 2]{yu2015useful} in the following step $(a)$ and our assumption on $\fronorm{Z(\ysol)-Z}$ in the step $(b)$:
\begin{equation}\label{eq: distVVstar}
    \opnorm{U-\Vstar} \overset{(a)}{\leq} 2\sqrt{2}\frac{\fronorm{Z(\ysol)-Z}}{\sigma_{\min>0}(Z(\ysol))} 
    \overset{(b)}{\leq}
    \frac{\sigma_{\min}(\Amap_{\Vstar})}{20\opnorm{\Amap}}
    \ . 
\end{equation}
Thus, combining \eqref{eq: AVAVstar} and \eqref{eq: distVVstar}, we see that $\opnorm{\mathcal{A}_{\Vstar} - \mathcal{A}_U} \leq \frac{4\sqrt{2}\opnorm{\Amap}\fronorm{Z(\ysol)-Z}}{\sigma_{\min>0}(Z(\ysol))}$ and $\sigma_{\min}(\Amap_{U})\geq \frac{9}{10}\sigma_{\min}(\Amap_{\Vstar})$.
\end{proof}
Next, we prove Lemma \ref{lem: auxiliary} based on the previous lemma. 
\begin{proof}[Proof of Lemma \ref{lem: auxiliary}]
From Proposition \ref{prop: existence_ALmin_P_D_sd}, we know \eqref{eq: ALmin} and \eqref{eq: ALminD} have optimal primal solutions and a unique dual solution, respectively.
%
We prove \eqref{eq: Zmin_nonzero_better_bound} first. From Proposition \ref{prop: existence_ALmin_P_D_sd}, we know \eqref{eq: ALmin} and \eqref{eq: ALminD} satisfy strong duality and indeed have optimal primal solutions and a unique dual solution, respectively. Using \eqref{eq: zypysol} in Lemma \ref{lem: zypyysol} in the following step $(a)$, we have that
\begin{equation}\label{eq: Zmin_nonzero} 
\fronorm{Z(\zyp) - Z(\ysol)} \leq \opnorm{\Amap}\twonorm{\zyp - \ysol}\frac{1}{10}\sigma_{\min >0} (Z(\ysol))\ .
\end{equation}
Hence, combining \eqref{eq: Zmin_nonzero} with Weyl's inequality, we see \eqref{eq: Zmin_nonzero_better_bound} holds. In particular, this reveals that 
\begin{equation}\label{eq: Zyprank_dualrank}
    \rank(Z(z_{y,\rho}))\geq \rank(Z(\ysol))=n-\rstar.
\end{equation}

Furthermore, due to \eqref{eq: DSlater's}, we can apply Proposition \ref{prop: existence_ALmin_P_D_sd}, and know the complementarity \eqref{eq: ALmin.KKT.cs}, $Z(\zyp) \Xyp =0$ holds for $(\Xyp, \zyp)$, which implies that  
\begin{equation}\label{eq: range_Xyrho_nullspace_Z_zyrho}
\range(\Xyp) \subset \nullspace(Z(z_{y,\rho})).
\end{equation}
Also, we have the following
\begin{equation}
\begin{aligned}\label{eq: subset_Z_zyrho_range_V}
&\dim(\nullspace(Z(z_{y,\rho}))) 
\overset{(a)}{=} n-\rank(Z(z_{y,\rho})) \overset{(b)}{\leq} \rstar =\dim(\range(V)) \\
\overset{(c)}{\implies}&\;
\nullspace(Z(z_{y,\rho}))\subset \range(V)
\end{aligned}
\end{equation}
where $(a)$ is due to the rank-nullity theorem, and $(b)$ is due to \eqref{eq: Zyprank_dualrank}. The step $(c)$ is due to $Z(z_{y,\rho})\succeq 0$ and the definition of $V$. The relationship \eqref{eq: null_space_V_range_Xyp_range} holds by combining \eqref{eq: subset_Z_zyrho_range_V} and \eqref{eq: range_Xyrho_nullspace_Z_zyrho}.


Lastly, we shall prove \eqref{eq: A_V_y_zrho}. Note that
\[
\fronorm{Z(\zyp)-Z(\ysol)}\overset{(a)}{\leq} \opnorm{\Amap}\twonorm{\zyp-\ysol} \overset{(b)}{\leq}  \frac{\sigma_{\min>0}(Z(\ysol))\sigma_{\min}(\Amap_{\Vstar})}{40\sqrt{2}\opnorm{\Amap}},\]
where $(a)$ is due to the definition of $Z(\zyp)$ and $Z(\ysol)$ and the step $(b)$ is based on the definition of $c_{\text{aux}}$. Thus, Lemma \ref{lem: Z_Z_star_A_U_inj} and the inequality $(a)$ and $(b)$ above imply \eqref{eq: A_V_y_zrho}.
\end{proof}

\begin{rem}[A weak primal rank bound of \eqref{eq: ALmin}]\label{rem: weaker_primal_rank_bound}
    The proof of  Lemma \ref{lem: auxiliary} reveals that $\sigma_{\rank(Z(\ysol)}(Z(\zyp))\geq \frac{9}{10} \sigma_{\min>0}(Z(\ysol)) >0$ holds so long as \eqref{eq: DSlater's} holds and $\twonorm{z_{y,\rho}-\ysol}\leq \frac{\sigma_{\min>0}(Z(\ysol))}{10\opnorm{\Amap}}$ for a dual optimal vector $\ysol$. Using the dual nonexpansiveness lemma, Lemma \ref{lem: zypyysol}, we have the following rank bounds for any optimal solution pair $(\Xyp,\zyp)$ to \eqref{eq: ALmin} and \eqref{eq: ALminD} under \eqref{eq: DSlater's} and  $\twonorm{y-\ysol}\leq \frac{\sigma_{\min>0}(Z(\ysol))}{10\opnorm{\Amap}}$ 
    even if primal simplicity fails:
    \begin{align}
       \rank(Z(\zyp))& \geq \rank(Z(\ysol)) \label{eq: Zyprank_dualrank_remark},\\
       \rank(\Xyp) &\overset{(a)}{\leq} n - \rank(Z(\zyp)) \leq n- \rank(Z(\ysol))\ . \label{eq: Xyprank_dualrank}
    \end{align}
    where the step $(a)$ is due to the complementarity condition \eqref{eq: ALmin.KKT.cs} of $X_{y,\rho}$ and $Z(z_{y,\rho})$.
\end{rem}

\begin{proof}[Proof of Lemma \ref{lem: X_yrho_X_star_diff_dual_diff}]
Using Lemma \ref{lem: zypyysol} and $\twonorm{y-\ysol}\leq c_{\text{aux}}$, we see Lemma \ref{lem: auxiliary} is activated. Using the proof of \eqref{eq: SS_star} (before the step $(b)$ there) in the proof of Theorem \ref{thm: ALmin_primal_simple_qg}, we know that 
\begin{equation}\label{eq: SS_star_diff_dual_diff}
\begin{aligned}
    \fronorm{S_{y,\rho}-S_\star}
& \leq 
\frac{\frac{4\sqrt{2}\opnorm{\Amap}^2\fronorm{S_\star}}{\sigma_{\min>0}(Z(\ysol))}\twonorm{\ysol-\zyp}+ \frac{1}{\rho} \twonorm{y-\zyp}}{\frac{9}{10}\sigma_{\min}(\mathcal{A}_{\Vstar})}
\overset{(a)}{\leq} 
\frac{\frac{4\sqrt{2}\opnorm{\Amap}^2\fronorm{S_\star}}{\sigma_{\min>0}(Z(\ysol))}\twonorm{\ysol-y}+ \frac{1}{\rho} \twonorm{y-\ysol}}{\frac{9}{10}\sigma_{\min}(\mathcal{A}_{\Vstar})}, 
\end{aligned}
\end{equation}
where the step $(a)$ is due to Lemma \ref{lem: zypyysol}. Due to the choice of $c_{\text{aux}}$ and $c_{\rho}$, the above also implies that 
\begin{equation}\label{eq: S_bound}
    \opnorm{S_{y,\rho}}\leq 2\opnorm{S_\star}. 
\end{equation}
Next, we can bound the difference $\fronorm{\Xyp - \Xsol}$ as follows 
\begin{equation}
\begin{aligned}
\fronorm{\Xyp - \Xsol} &= \fronorm{VS_{y,\rho}V^\top - \Vstar S_{\star} V_\star^\top} \\
& \leq \fronorm{VS_{y,\rho}V^\top - VS_{y,\rho}\Vstar^\top} +  \fronorm{VS_{y,\rho}\Vstar^\top - V S_\star\Vstar^\top} 
+ \fronorm{VS_\star \Vstar^\top - \Vstar S_\star\Vstar^\top}. 
\end{aligned}
\end{equation}
Combining the above with \eqref{eq: S_bound}, \eqref{eq: SS_star_diff_dual_diff} and \eqref{eq: distVVstar} in the proof of Lemma \ref{lem: auxiliary}, we obtain the desired bound \eqref{eq: Xyp_Xsol_dual_diff_bound}.
\end{proof}
\section{Lemmas for quadratic growth}\label{sec:lem}
Our first lemma is drawn from \cite[Lemma 9]{ding2020spectral}. Although it may appear technical, it is a crucial step in proving our quadratic growth result. 
	\begin{lem}\cite[Lemma 9]{ding2020spectral}\label{lem: SpecW}
    Suppose $Y\in\symMat{\dm}$ with eigenvalues
		$\lambda_{1}(Y)\geq\dots\geq\lambda_{\dm}(Y)$, and $\lambda_{\dm-r}(Y)-\lambda_{\dm-r+1}(Y)\geq\delta$.
Let  $V_{Y,r}\in\real^{\dm\times r}$
		formed by the last $r$  orthonormal eigenvectors $v_{\dm-r+1},\dots v_{\dm}$ of $Y$.
Define $\face_r(Y)=\left\{ V_{Y,r}SV_{Y,r}^{\top}\mid S\succeq0,\tr(S)=1\right\} .$
		Then for any $X\in\symMat{\dm}$ with $\tr(X)=1,X\succeq0$, there
		is some $W\in\face_r(Y)$ such that 
		\[
		\inprd{X-W}Y\geq\frac{\delta}{2}\fronorm{X-W}^{2}.
		\]
	\end{lem}
	
	\begin{rem}
		Note that as long as $\range(V)=\range(V_{Y,r})$ for some matrix
		$V\in\real^{\dm\times r}$ with orthonormal columns, the two sets $\face _r(Y)$
		and  $\left\{ VSV^{\top}\mid S\succeq0,\tr(S)=1\right\} $ are the same.
	\end{rem}

Next, we use the previous lemma to establish a quantitative version of a system of equality and inequalities, which generalizes \cite[Lemma 12]{ding2020spectral} as it does not require $\Amap(X) =b \implies \tr(X) = 1$. 
\begin{lem}
		\label{lem: kkt_like.qg} 
        Suppose the following system admits
		a unique solution $\bar{X} \in \symMat{\dm}$ with rank $\bar{r}:$ 
		\begin{equation}
		\inprd{\bar{Z}}{X}=0,\quad \Amap X=b,\quad\text{and\ensuremath{\quad X\succeq0},}\label{eq: linearmatrixsystem}
		\end{equation}
		for a $\bar{Z}\succeq0$ such that $\rank(\bar{Z})+\rank(\bar{X})=\dm$,
		a linear map $\Amap:\symMat{\dm}\rightarrow\real^{m}$, and a
		vector $b\in\real^{m}$. Let $V\in\mathbb{R}^{n\times\bar{r}}$ be a matrix with orthonormal
		columns represent the eigenspace of $\bar{X}$
		for positive eigenvalues.
		Then we have the following.
        \begin{enumerate}
            \item Injectivity of $\mathcal{A}_V$: $\sigma_{\min}(\Amap_V)>0$
            \item Quadratic error bound: for any $X\succeq 0$, we have 
		\begin{equation}\label{eq: X-xsolZXAX-Axsol}
		\begin{aligned}
		\fronorm{X-\bar{X}}^{2}\;\leq\tr(X)\left(4+8\frac{\sigma_{\max}^2(\Amap)}{\sigma_{\min}^2(\Amap_{V})}\right)\frac{\inprd{\bar{Z}}X}{\lambda_{n-\bar{r}}(\bar{Z})}
		+\frac{4}{\sigma_{\min}^{2}(\Amap_{V})}\twonorm{\Amap(X)-b}^{2}.
		\end{aligned}
		\end{equation}
        \end{enumerate} 
\end{lem}
	
	\begin{proof}
We first note the matrix $V$ is also a representation of the null space of the $\bar{Z}.$ Indeed, because $\inprd{\bar{Z}}{\bar{X}}=0$ and both matrices are PSD, we know $\bar{Z}\bar{X}=0$ and, hence,
$\range(\bar{X}) \subset \nullspace(\bar{Z}).$
Combining this with $\rank(\bar{Z}) + \rank(\bar{X}) = n$, we know $\nullspace(\bar{Z})= \range(V) = \range(\bar{X})$.
        
Next, we verify that the linear map $\Amap_{V}: 
        \symMat{\bar{r}}\rightarrow \real^{m}$
       with $ \Amap_V(S) =  \Amap(VSV^\top)$ is injective. 
		Suppose not, then there is some nonzero $S_{0}\in\symMat{\bar{r}}$
		such that $\Amap_{V}(S_{0})=0$. But this means that  $X_\alpha = V(\alpha S_{0}+\bar{S})V^{\top}$ also
		satisfies the system (\ref{eq: linearmatrixsystem}) for all small
		enough $\alpha$ ($\inprd{\bar{Z}}{X_\alpha}=0$ as $\range(V) = \nullspace(Z)$), which contradicts to the assumption that $\bar{X}$ is a unique solution to  (\ref{eq: linearmatrixsystem}). Hence $\sigma_{\min}(\Amap_{V})=\min_{\fronorm S=1}\twonorm{\Amap_{V}(S)}>0$, and we have that for any $S\in\symMat{\bar{r}}$
\begin{equation}\label{eq:firstinequalityqudraticgrowthlemma}
		\begin{aligned} 
		\fronorm{VSV^{\top}-\bar{X}}\leq&\frac{1}{\sigma_{\min}(\Amap_{V})}\twonorm{\Amap(VSV^{\top})-\Amap(\bar{X})}
		=&\frac{1}{\sigma_{\min}(\Amap_{V})}\twonorm{\Amap(VSV^{\top})-b}.
		\end{aligned}
		\end{equation}

Using Lemma \ref{lem: SpecW},
		we know for any $X\succeq 0$ and $X\not=0$, there is some $\bar{W}=VSV^{\top}\in\face_{\bar{r}}(\bar{Z})$ such that
        \begin{align}
	&\inprd{\frac{X}{\tr(X)}}{\bar{Z}}\overset{(a)}{=}\inprd{\frac{X}{\tr(X)}-\bar{W}}{\bar{Z}}\geq\frac{\lambda_{n-\bar{r}}(\bar{Z})}{2}\fronorm{\frac{X}{\tr(X)}-\bar{W}}^{2}\\
\implies &        
\tr(X)\inprd{{X}}{\bar{Z}}= \tr(X) \inprd{{X}-\bar{W}}{\bar{Z}}\geq\frac{\lambda_{n-\bar{r}}(\bar{Z})}{2}\fronorm{{X}-\tr(X)\bar{W}}^{2}\label{eq:quadraticgrowthlemmaZero}
        \end{align}
		where step $(a)$ is because $\lambda_{n-\bar{r}+1}(\bar{Z})=\dots=\lambda_{\dm}(\bar{Z})=0$.

Let $W = \tr(X)\bar{W}$. We can bound $\fronorm{X-\bar{X}}^{2}$ by 
		\begin{align}\label{eq:secondinequalityqudraticgrowthlemma-revised}
		\fronorm{X-\bar{X}}^{2} & \overset{(a)}{\leq}2\fronorm{X-W}^{2}+2\fronorm{W-\bar{X}}^{2}
		\overset{(b)}{\leq}2\fronorm{X-W}^{2}+\frac{2}{\sigma_{\min}^{2}(\Amap_{V})}\twonorm{\Amap(W)-b}^{2}.
		\end{align}
		Here we use triangle inequality and basic inequality $(a+c)^{2}\leq2a^{2}+2c^{2}$
		for any real $a,c$ in step $(a)$. In step $(b)$, we use (\ref{eq:firstinequalityqudraticgrowthlemma}). 
		
		We can further bound the term $\twonorm{\Amap(W)-b}$ by 
		\begin{equation}\label{eq:thirdinequalityqudraticgrowthlemma-revised}
		\begin{aligned}
		\twonorm{\Amap(W)-b} &=\twonorm{\Amap(W-X)+\Amap(X)-b}
		&\leq\twonorm{\Amap(W-X)}+\twonorm{\Amap(X)-b}.
		\end{aligned} 
		\end{equation}
		Now combining (\ref{eq:secondinequalityqudraticgrowthlemma-revised}),
		(\ref{eq:thirdinequalityqudraticgrowthlemma-revised}) and $(a+c)^{2}\leq2a^{2}+2c^{2}$
		for any $a,c\in\real$ in the following step $(a)$, we see 
		\begin{equation}\label{eq:fourthinequalityqudraticgrowthlemma-revised} 
		\begin{aligned}
		\fronorm{X-\bar{X}}^{2} & \overset{(a)}{\leq}2\fronorm{X-W}^{2}+\frac{4\twonorm{\Amap(W-X)}^{2}}{\sigma_{\min}^{2}(\Amap_{V})}
		+\frac{4}{\sigma_{\min}^{2}(\Amap_{V})}\twonorm{\Amap(X)-b}^{2}\\
		& \leq\left(2+4\frac{\sigma_{\max}^{2}(\Amap)}{\sigma_{\min}^{2}(\Amap_{V})}\right)\fronorm{X-W}^{2}
		+\frac{4}{\sigma_{\min}^{2}(\Amap_{V})}\twonorm{\Amap(X)-b}^{2}.\nonumber 
		\end{aligned}
		\end{equation} 
		Finally using (\ref{eq:quadraticgrowthlemmaZero}) to bound $\fronorm{X-W}$,
		we reached the inequality we want to prove:
		\begin{equation*} 
		\begin{aligned}
		\fronorm{X-\bar{X}}^{2} &\leq \tr(X)\left(4+8\frac{\sigma_{\max}^{2}(\Amap)}{\sigma_{\min}^{2}(\Amap_{V})}\right)\frac{\inprd{\bar{Z}}X}{\lambda_{n-\bar{r}}(\bar{Z})}
		+\frac{4}{\sigma_{\min}^{2}(\Amap_{V})}\twonorm{\Amap(X)-b}^{2}.
		\end{aligned} 
		\end{equation*}
	\end{proof}

A quick consequence of Lemma \ref{lem: kkt_like.qg} is the following quadratic growth lemma of \eqref{eq: sdp.p}. Note that for any $X\succeq 0$, we have $\tr(X) = \nucnorm{X}$.

\begin{lem}\label{lem: sdp.p.qg}
Suppose the assumptions of Theorem \ref{thm: ALmin_primal_simple_qg} hold. Then, for any $B>0$, we have the following inequality for any $X\succeq 0$ with $\nucnorm{X}\leq B$: 
\begin{equation}\label{eq: sdp.p.X-Xstar.qg}
		\begin{aligned}	
        \fronorm{X-\Xsol}^{2}\;\leq 
        a_0 (\inprd{C}{X} -\psol) + a_1 \twonorm{\Amap X - b},
	\end{aligned}
	\end{equation}
\end{lem}
where  $\psol = \inprd{C}{\Xsol}$, $a_0 = B\left(4+8\frac{\sigma_{\max}^2(\Amap)}{\sigma_{\min}^2(\Amap_{\Vstar})}\right)\frac{1}{\sigma_{n-\rstar}(Z(\ysol))}$, and 
$a_1 = a_0 \twonorm{\ysol} + \frac{4\sigma_{\max}(\Amap)(B+\fronorm{\Xsol})}{\sigma_{\min}^{2}(\Amap_{\Vstar})}$.
\begin{proof}
From Lemma \ref{lem: kkt_like.qg}, we know the following inequality holds for any $X\succeq 0$: 
	\begin{equation}\label{eq: X-Xstar_kkt.qg}
		\begin{aligned}	
        \fronorm{X-\Xsol}^{2}\;\leq\tr(X)\left(4+8\frac{\sigma_{\max}^2(\Amap)}{\sigma_{\min}^2(\Amap_{\Vstar})}\right)
        \frac{\inprd{Z(\ysol)}{X}}{\lambda_{n-\rstar}(Z(\ysol))}
		+\frac{4}{\sigma_{\min}^{2}(\Amap_{\Vstar})}\twonorm{\Amap(X)-b}^{2}.
	\end{aligned}
	\end{equation}
We bound $\inprd{Z(\ysol)}{X}$ as follows: 
\begin{align}\label{eq: sdp.ZX.SbpLIF}
\inprd{Z(\ysol)}{X} = 
\inprd{C-\Amap^*(\ysol)}{X}  = 
\inprd{C}{X} - \inprd{b}{\ysol} + 
\inprd{\ysol}{b - \Amap X}  
 \overset{(a)}{\leq} \inprd{C}{X} -\psol + \twonorm{\ysol}\twonorm{\Amap X - b}.
\end{align}
In the step $(a)$, we use the strong duality for \eqref{eq: sdp.p} and \eqref{eq: sdp.d}. We bound $\twonorm{\Amap X-b}$ as follows: 
\begin{align}\label{eq: sdp.AXb.bd}
    \twonorm{\Amap X-b}\leq \sigma_{\max} (\Amap)\fronorm{X-\Xsol} \overset{(a)}{\leq} \sigma_{\max} (\Amap)(B+\fronorm{\Xsol}),
\end{align}
where the step $(a)$ is due to the triangle inequality and $\fronorm{X}\leq \nucnorm{X}\leq B$. 

Combining pieces \eqref{eq: X-Xstar_kkt.qg}, \eqref{eq: sdp.ZX.SbpLIF}, and \eqref{eq: sdp.AXb.bd}, we see \eqref{eq: sdp.p.X-Xstar.qg} holds.
\end{proof}

The following lemma gives a quadratic growth result for the more general problem \eqref{eq: f_g_A}. 
\begin{lem}\label{lem: f_g_A_C_QG}
   Consider the problem \eqref{eq: f_g_A} has $f$ of the form \eqref{eq: f_g_C_form}, i.e., 
    \[
    f(X) = g(\Amap X) + \inprd{C}{X}
    \]
    where $g:\mathbb{R}^m \rightarrow \real$ is an $\alpha_g$-strongly and $L_g$ smooth convex function, the map $\Amap: \symMat{\dm}\rightarrow \mathbb{R}^m$ is linear, and $C\in \symMat{\dm}$. Suppose there is a unique rank $\bar{r}$ primal solution $X_0$ to \eqref{eq: f_g_A} exists with $V\in \mathbb{R}^{n\times \bar{r}}$ of orthonormal columns representing the range of $X_0$. And suppose that strict complementarity holds: 
    \[
    \rank(\nabla f(X_0)) + \rank(X_0) =n.
    \]
    Then the problem \eqref{eq: f_g_A} satisfies $(\gamma,n,B)$-QG for any $B>0$ with $\gamma$ specified in \eqref{eq: f_g_A_fXX_0_gamma} (depending on $B$).
\end{lem}
\begin{proof}
Let $X$ be any feasible point for \eqref{eq: f_g_A}, we have the following derivation for $f(X)  - f(X_0)$:
\begin{equation}
\begin{aligned}\label{eq: f_g_A_fXX_0}
    f(X) - f(X_0) 
    & = g(\Amap X) - g(\Amap X_0) + \inprd{C}{X-X_0} \\
    &  \overset{(a)}{\geq} \inprd{\Amap^* \nabla g (\Amap X_0)}{X-X_0} + \frac{\alpha_g}{2} \twonorm{\Amap (X-X_0)}^2 + \inprd{C}{X-X_0}\\
    & \overset{(b)}{=} \inprd{\nabla f(X_0)}{X-X_0} + 
    \frac{\alpha_g}{2} \twonorm{\Amap (X-X_0)}^2\\
    & \overset{(c)}{=}\inprd{\nabla f(X_0)}{X} + \frac{\alpha_g}{2} \twonorm{\Amap (X-X_0)}^2\\
    & \overset{(d)}{\geq} 0.
\end{aligned}
\end{equation}
In the step $(a)$, we use the assumption that $g$ is $\alpha_g$ strongly convex. In the step $(b)$, we use $\nabla f(X_0) = \Amap^* \nabla g(\Amap X_0) +C $ by the chain rule. In the step $(c)$, we use the KKT condition \eqref{eq: f_g_A_kkt}, which holds thanks to the assumption that $X_0$ exists and that the primal Slater's condition is satisfied for \eqref{eq: f_g_A} (as $I\succ 0$). The last step $(d)$ is because  $\inprd{\nabla f(X_0)}{X}\geq 0$ as $\nabla f(X_0)\succeq 0$ due to \eqref{eq: f_g_A_kkt}. 
Thus, we see that $X$ is optimal if and only if the following condition holds: 
\begin{equation}
    \inprd{\nabla f(X_0)}{X}=0, \quad \Amap X = \Amap X_0, \quad X\succeq 0.
\end{equation}
Our uniqueness assumption on $X_0$ shows the above system has a unique solution. Due to strict complementarity, we can apply Lemma \ref{lem: kkt_like.qg} to the above system and conclude that 
\begin{equation}\label{eq: f_g_A_XX_0}
    \fronorm{X-X_0}^2 \leq \tr(X) \left(
    4 +8 \frac{\sigma_{\max}^2(\Amap)}{\sigma_{\min} ^2(\Amap_V)}
    \right) \frac{\inprd{\nabla f(X_0)}{X}}{\lambda_{n-\bar{r}}(\nabla f(X_0))}
    + \frac{4}{\sigma_{\min}^2 (\Amap_V)} \twonorm{\Amap(X -X_0)}^2.
\end{equation}
Combining \eqref{eq: f_g_A_fXX_0} and \eqref{eq: f_g_A_XX_0}, we have that for any $X\succeq 0$ with $\nucnorm{X}\leq B$, 
\begin{equation}
    f(X)-f(X_0)\geq \gamma \fronorm{X-X_0}^2,
\end{equation}
where $\gamma$ is of the following form:
\begin{equation}\label{eq: f_g_A_fXX_0_gamma}
    \gamma =\min\left\{ \frac{\lambda_{n-\bar{r}}(\nabla f(X_0))}{
    B\left(
    4 +8 \frac{\sigma_{\max}^2(\Amap)}{\sigma_{\min}^2 (\Amap_V)}
    \right)}, 
    \frac{\alpha_g\sigma_{\min}^2 (\Amap_V)} {8}
    \right\}. 
\end{equation}
\end{proof}

\section{Sensitivity of Theorem \ref{thm: ALmin_primal_simple_qg} regarding primal simplicity}
\label{sec: failure_ps}
In this section, we show that the two core components of primal simplicity, namely, the primal uniqueness and strict complementarity, are, in a sense, necessary to Theorem \ref{thm: ALmin_primal_simple_qg}. Specifically, we show that if some of the assumptions fail, then some desired properties of Theorem \ref{thm: ALmin_primal_simple_qg} also fail. 

\subsection{Failure of primal uniqueness and diminishing quadratic growth}
In this section, we show that the failure of the primal uniqueness assumption, which may seem non-critical to the rank stability, the strict complementarity property, and the quadratic growth of the subproblem \eqref{eq: ALmin}, may surprisingly cause unstable solution rank and have a diminishing quadratic growth constant as $y$ gets close to a strict complementary $\ysol$. 

More specifically, our example \eqref{eq: sdp.p} has multiple primal solutions and one of them has rank $2$, and satisfies strict complementarity. Then, there is a sequence of $y$ approaching a strictly complementary dual optimal solution $\ysol$ and the quadratic growth condition of \eqref{eq: ALmin} diminishes as $y$ approaches the strictly complementary dual optimal solution and the rank is always $1$ unless $y=\ysol$. 

Consider the SDP \eqref{eq: sdp.p} with the following problem data: 
\begin{equation}
C = 0,\quad 
    A_1 = \begin{bmatrix}
        1 & 0 & 0\\
        0 & 1 & 0 \\
        0 & 0 & 0
    \end{bmatrix},\quad A_2 = \begin{bmatrix}
        0 & 0 & 0 \\
        0 & 0 & 0 \\
        0 & 0 & 1
    \end{bmatrix},\quad A_3 = \begin{bmatrix}
        0 & 0 & 1 \\
        0 & 0 & 0 \\
        1 & 0 & 0
\end{bmatrix},\quad \text{and}\quad b = \begin{bmatrix}
        1 \\ 0 \\0
    \end{bmatrix}.
\end{equation}
Note that $X_\star = \diag(\frac{1}{2},\frac{1}{2},0)$ and $y_\star = [0,1,0]^\top$ form an optimal primal-dual pair. And they satisfy strict complementarity. The primal optimal solution set is given by 
\[
\mathcal{X} = \{X\in \symMat{3}_+\mid X = \diag(X_1,0), \;X_1\in \symMat{2}_+, \; \tr(X_1)=1\}
\]

Let us fix $\rho=1$ in \eqref{eq: AL} for simplicity. For any $\epsilon\in(0,1)$, consider $y = [-\epsilon, -1,-\delta]$ in \eqref{eq: ALmin} with $0<\delta^2 \leq \epsilon$ to be chosen so that $Z(y) \succeq 0$. Let $X = X_{y,1}$ be any primal optimal solution to \eqref{eq: ALmin}. From the
KKT condition \eqref{eq: ALmin.KKT} and Proposition \ref{prop: existence_ALmin_P_D_sd}, the dual optimal slack matrix $Z(z_{y,1})\succeq 0$ is of the following form:
\begin{equation}\label{eq: Zyrho_nonunique_primal_solution}
Z(z_{y,1})=\begin{bmatrix}
\epsilon + ( X_{11}+X_{22} -1)  & 0 & X_{13}+\delta \\ 
0 & \epsilon+ (X_{11}+X_{22} -1) & 0 \\
\delta+  X_{13} &  0 & 1 +  X_{33}
\end{bmatrix}\succeq 0
\end{equation}

We next make our choices of the entries of $X_{y,1}$ and $\delta$. Our choices shall make $Z(z_{y,1})$ and $X_{y,1}$ indeed dual and primal optimal, which can be verified using the KKT conditions of the augmented Lagrangian, i.e., \eqref{eq: ALmin.KKT}. We shall omit the verification details. We make the following choices of $X=X_{y,1}$ in terms of its $X_{11}$, $X_{13}$, and $X_{33}$:
\begin{equation}\label{eq: nonunique_X_y_1_choice}
X_{11} = 1 - (1-\eta)\epsilon, \quad 
X_{33} = \frac{\sqrt{1+4\eta\epsilon(1-(1-\eta)\epsilon)}-1}{2}, \quad \text{and}\quad X_{13} = \sqrt{X_{11}X_{33}}.
\end{equation}
The small constant $\eta>0$ will be determined later. We set the remaining entries of $X_{y,1}$ to be $0$. We make the following choice of $\delta$:
\begin{equation}
    \delta = 
    - \left(\sqrt{\eta\epsilon(1+X_{33})} + X_{13}\right)
 =   - \left(\sqrt{\eta\epsilon(1+X_{33})} + \sqrt{X_{33}X_{11}}\right).
\end{equation}
By picking a number $\eta$ small enough and requiring $\epsilon \in (0,c)$ for some small numerical constant $c>0$, we see that 
\begin{equation}\label{eq: X_11_33_delta_range}
X_{11} \in (0.992,1.005),\quad  X_{33} \in (0.991\eta \epsilon,\eta\epsilon), \quad
    \delta \in (-2.01 \sqrt{\eta \epsilon}, -1.99 \sqrt{\eta \epsilon}), \quad \text{and}\quad 0<\delta^2 \leq \epsilon.
\end{equation}
Our choices are complete. The above choices make sure $X_{y,1}$ and $Z(z_{y,1})$ are primal and dual optimal to \eqref{eq: ALmin}. We remark that (i) $Z(z_{y,1})_{11} = Z(z_{y,1})_{22}=\eta \epsilon>0$, (ii) $Z(z_{y,1})$ has rank $2$, and (iii) $X_{y,1}$ has rank $1$. Note that (ii) implies any primal optimal solution $\tilde{X}$ to \eqref{eq: ALmin} has rank no more than $1$ with $\range(\tilde{X})\subset \nullspace(Z(z_{y,1})) = \range(X_{y,1})$ by complementarity \eqref{eq: ALmin.KKT.cs}. Hence $\tilde{X} = \tilde{c}X_{y,1}$ for some constant $\tilde{c}\geq 0$. Since $\Amap(\tilde{X}) =\Amap(X_{y,1})$  by Lemma \ref{lem: AL_same_image} and $\Amap(X_{y,1})\not=0$, we see that  $\tilde{c}=1$ and $X_{y,1}$ given above is the unique primal optimal solution to \eqref{eq: ALmin}.

Now consider the following $X_{\xi}$ for any $\xi>0$: 
\begin{equation} 
X_\xi = \begin{bmatrix}
    X_{11} & \sqrt{X_{11}}\xi & X_{13} \\ 
    \xi \sqrt{X_{11}} & \xi^2 & \xi \sqrt{X_{33}} \\ 
    X_{13} & \sqrt{X_{33}}\xi & X_{33}
\end{bmatrix}. 
\end{equation}
Note that $ X_\xi \succeq 0$ as $X_\xi = FF^\top$ where $F = [\sqrt{X_{11}},\; \xi, \; \sqrt{X_{33}}]^\top $. 
The objective difference in terms of augmented Lagrangian \eqref{eq: AL} is the following:
\begin{equation}\label{eq: nonunique_L_diff}
    \mathcal{L}_{1}(X_\xi,y) - 
    \mathcal{L}_{1}(X_{y,1},y) 
    \overset{(a)}{=}
    \inprd{Z(z_{y,1})}{X_\xi - X_{y,1}} + \frac{1}{2}\twonorm{\Amap(X_{\xi} - X_{y,1}}^2 = 
    \xi^2 \eta \epsilon + 
    \frac{1}{2} \xi^4.
\end{equation}
where the step $(a)$ is from \eqref{eq: Al_diff}. The square of the distance of $X_\xi$ to $X_{y,1}$ is the following:
\begin{equation}\label{eq: nonunique_distance}
\fronorm{X_\xi - X_{y,1}}^2 = \xi^4 + 2\xi^2 (X_{33} +X_{11}).
\end{equation}
By comparing \eqref{eq: nonunique_L_diff} and \eqref{eq: nonunique_distance} and using  $X_{33} +X_{11} \in (0.99,1.01)$ from \eqref{eq: nonunique_X_y_1_choice} and the smallness of $\epsilon$, we see the quadratic growth constant is no more than $\eta \epsilon$, which is diminishing as $\epsilon$ goes to $0$.  

For general $\rho>0$, we set $y = \rho[-\epsilon,-1,-\delta]^\top$ with the previous choice of $\delta>0$ and $X_{y,\rho}= X_{y,1}$. The quadratic growth constant is then $\rho \eta\epsilon/3$,  which is also diminishing as $\epsilon$ goes to $0$.  

\subsection{Failure of strict complementarity and high-rank solutions.}
In this example, we present a series of SDPs that satisfy the primal and dual nondegeneracy \cite[Definition 5 and 8]{alizadeh1997complementarity} with rank-$1$ primal solutions, which ensures the primal and dual uniqueness \cite[Theorem 6 and 9]{alizadeh1997complementarity}. It should also imply the quadratic growth of the dual problem \eqref{eq: sdp.d} using arguments in \cite[Proposition 2.1]{zhao2010newton}. Note that the quadratic growth of \eqref{eq: sdp.d} ensures that ALM methods converge linearly according to the classical result \cite{rockafellar1976augmented}. However, for this example, the strict complementarity is not satisfied, and, critically, we show that there is a sequence of feasible $y_\epsilon$ converging to the (unique) dual optimal solution of \eqref{eq: sdp.d}, leading to a high-rank solution of \eqref{eq: ALmin}. Thus, strict complementarity is important for our results in Theorem \ref{thm: ALmin_primal_simple_qg} while other common regularity conditions, such as primal and dual nondegeneracies, cannot guarantee them.

For an arbitrary integer $r\geq 2$, consider $n=m=\frac{r(r+1)}{2}$. Let $k(i,j) = \frac{(2r-i+2)(i-1)}{2} + j-i+1$ for any $1\leq i\leq j\leq r$ and $E_r = \sum_{i=2}^r e_{k(i,i)}e_1^\top +e_1 e_{k(i,i)}^\top.$ We make the following choices of the cost and constraint matrices of \eqref{eq: sdp.p}:
\begin{equation}
\begin{aligned}
    C      &= \diag(0_r,\mathbf{1}_{n-r})\ , 
    \quad \quad 
      b_{ij} = 
    \begin{cases}
        1 , & i=j=1 \;\text{or}\;i=j=r \\ 
        0, &\text{otherwise}
    \end{cases}, \quad \text{and}\\
A_{ij} &  = \begin{cases}
        \frac{e_1 e_j^\top + e_j e_1^\top }{2}  \ , \; 1\leq  j \leq r ,\\
        \frac{e_i e_j^\top + e_j e_i^\top }{2} + 
        \frac{e_1 e_{k(i,j)}^\top + e_{k(i,j)}e_1^\top}{2} \ , \; 1< i\leq j \leq r \; \text{and}\; 
        (i,j)\not = (r,r)\\
    \diag(\mathbf{1}_r,0_{n-r}) + E_r, i=j= r.
    \end{cases}.
\end{aligned}    
\end{equation}
By verifying the definition of primal nondegeneracy and dual nondegeneracy in \cite{alizadeh1997complementarity}, we can show that the primal optimal solution is uniquely $\Xsol = e_1e_1^\top$ and the dual optimal solution is also uniquely $\ysol = 0$ with $Z(\ysol) = C$. Note that the strict complementarity does not hold in this case.

For simplicity, we set $\rho=1$. Our discussion on removing the requirement on $\rho =1 $ in Remark \ref{rem: choice_rho_high_rank_localness} also applies here. 
For any $\epsilon \in (0,\frac{1}{3(r-1)})$, we set $y_\epsilon$ in \eqref{eq: ALmin} as the following:
\begin{equation}
    y_{\epsilon,ij} =\begin{cases}
        -2(r-1)\epsilon, & i=j=1 \\
        \epsilon, & i=j, \;i\in[2,r-1]\\
        -(r-1)\epsilon, & i=j=r\\
        0, &\text{otherwise}
    \end{cases}\ . 
\end{equation}
One can check that the dual slack $Z(y_\epsilon)\succeq 0$ as $\epsilon < \frac{1}{3(r-1)}$. 
The primal optimal solution to \eqref{eq: ALmin} with $y_\epsilon$ and $\rho=1$ is given by 
\begin{equation}
    [X_{\rho,y_\epsilon}]_{ij} = \begin{cases}
        1-2(r-1)\epsilon & i=j=1\\
        \epsilon, &  i=j, \;i\in[2,r]\\
        0, & \text{otherwise}
    \end{cases} \ .
\end{equation}
Its optimality can be verified by checking the KKT condition with dual optimal slack $Z_{\rho, y_\epsilon} = C$.
The matrix $X_{\rho,y_\epsilon}$ is clearly of rank $r$ which is much larger than $1$.

\section{Examples of quadratic growth} \label{eq: f_g_A_QG_example}
Our first example of quadratic growth of \eqref{eq: f_g_A} consists of problems satisfying strong convexity.

\begin{exam}[Strong convexity]
If the function $f$ is $\alpha-$strongly convex for an $\alpha>0$, i.e., for any $X,Y$, we have 
\begin{equation}
    f(X) \geq f(Y) + \inprd{\nabla f(Y)}{X-Y} + \frac{\alpha}{2}\fronorm{X-Y}^2\ ,
\end{equation}
then the problem satisfies $\alpha$-QG. Indeed, by taking $Y=X_0$ in the above inequality, we have 
\begin{equation}
    f(X) \geq f(X_0) + \inprd{\nabla f(X_0)}{X-X_0} + \frac{\alpha}{2}\fronorm{X-X_0}^2 \overset{(a)}{\geq} 
    f(X_0) + \frac{\alpha}{2}\fronorm{X-X_0}^2\ ,
\end{equation}
where the step $(a)$ is due to the KKT condition \eqref{eq: f_g_A_kkt}. Strong convexity is satisfied when the function $f$ is a sum of a convex function and a quadratic term of the form $\fronorm{X-X^\text{ref}}^2$, where $X^{\text{ref}}$ is a reference point depending on the underlying application. In machine learning applications, one typically has $X^{\text{ref}}=0$, and the quadratic term is often referred to as Tikhonov regularization. In many optimization subroutines for solving \eqref{eq: sdp.p}, e.g., the projection step in ADMM \cite{gabay1976dual} and extragradient method (EGM) \cite{korpelevich1977extragradient}, we also have this quadratic term, and each subroutine is about optimizing a strongly convex function over the set of positive semidefinite matrices. 
\end{exam}

Our next example consists of problems satisfying strict complementarity. To better understand strict complementarity in this setting, let us state the KKT conditions of \eqref{eq: f_g_A}. Assuming an optimal solution $X_0$ to \eqref{eq: f_g_A} exists, the KKT condition of \eqref{eq: f_g_A}  is satisfied and is the following:
\begin{subequations}\label{eq: f_g_A_kkt}
    \begin{align}
        \nabla f(X_0) X_0 &= 0\ ,\label{f_g_A_kkt_cs}\\
        \nabla f(X_0)\succeq 0, \quad X_0&\succeq 0\ .\label{f_g_A_kkt_feasible}
    \end{align}
\end{subequations}
Let us now introduce the second example.

\begin{exam}[Strict complementarity: \eqref{eq: ALmin} and beyond] For \eqref{eq: f_g_A}, strict complementarity holds if the KKT condition holds for an $X_0$ and 
\begin{equation}\label{eq: f_g_A_sc}
  \rank(\nabla f(X_0)) + \rank(X_0) = n \overset{(a)}{\iff} \range(X_0) = \nullspace(\nabla f(X_0))\ ,
\end{equation}
where $(a)$ is due to \eqref{f_g_A_kkt_cs}. For the augmented Lagrangian subproblem \eqref{eq: ALmin}, the above definition of strict complementarity coincides with our earlier definition \eqref{eq: ALmin.sc}. 

From Theorem \ref{thm: ALmin_primal_simple_qg}, Problem \eqref{eq: ALmin} satisfies the quadratic growth condition for all $y$ close to a strictly complementary $\ysol$. Specifically, by letting $B=2\nucnorm{X_\star}$, we have the parameter $\gamma$ specified in \eqref{eq: ALmin.gamma_formula} in Theorem \ref{thm: ALmin_primal_simple_qg}. We see such $\mathcal{L}_\rho(X,y)$ satisfies $(\gamma,\,n,\,B)$-QG. 

In general, consider function $f$ of the form: 
\begin{equation}\label{eq: f_g_C_form}
f(X):= g(\Amap X) + \inprd{C}{X}\ , 
\end{equation} 
where $g:\mathbb{R}^m \rightarrow \mathbb{R}$ is an $\alpha_g$-strongly and $L_g$-smooth convex function, the map $\Amap: \symMat{\dm} \rightarrow \mathbb{R}^{m}$ is linear, and $C\in \symMat{\dm}$ is a symmetric matrix. The above form includes \eqref{eq: ALmin} as a special case. 

According to Lemma \ref{lem: f_g_A_C_QG} in the appendix, if there is a unique primal solution $X_0$ and strict complementarity holds in the sense of \eqref{eq: f_g_A_sc}, then Problem \eqref{eq: f_g_C_form} satisfies $(\gamma,\,n,\,B)$-QG for any $B>0$ with a $\gamma>0$ specified in \eqref{eq: f_g_A_fXX_0_gamma} (depending on $B$). Lastly, we note that the strict complementarity and uniqueness of the primal solution are satisfied for generic $C$ so long as the primal solution exists \cite[Corollary 3.5]{drusvyatskiy2011generic}. 
\end{exam}

Our last example is an extension of strong convexity. 

\begin{exam}[Restricted strong convexity]
A function $f$ satisfies $(\alpha, r)$ restricted strong convexity if for any $X,Y$ with rank no more than $r$, we have 
\begin{equation}\label{eq: f_g_A_rsc}
f(X) \geq f(Y) + \inprd{\nabla f(Y)}{X-Y} + \frac{\alpha}{2} \fronorm{X-Y}^2\ .
\end{equation}
Provided $\rank(X_0)\leq r$, Problem \eqref{eq: f_g_A} satisfies $(\alpha,\,r, \,+\infty)$-QG. 
The restricted strong convexity condition relaxes the strong convexity condition by only requiring the inequality \eqref{eq: f_g_A_rsc} to hold over low rank matrices. Such a condition is satisfied (or approximately satisfied) in many statistical signal processing applications, e.g., low-rank matrix recovery problems described in \cite{recht2010guaranteed,chen2015fast,chen2018harnessing,ding2021simplicity}, which have drawn significant attention in recent years due to their connections to modern neural networks \cite{du2018algorithmic,xiong2023over}.
\end{exam}

\section{Sublinear convergence of gradient descent for the ALM-BM subproblem}\label{sec: sublinear_convergence}

We show that for some $f$ in \eqref{eq: f_g_A} that is strongly convex, once $r>\rank(X_0)$ in \eqref{eq: f_g_A.BM}, then \eqref{eq: GD} converges at a sublinear rate rather than at a linear rate, no matter how local the initialization is. In fact, this example is a special case of \eqref{eq: ALmin} for a particular choice of $\Amap$, $b$, and $C$ in \eqref{eq: sdp.p}.

Consider the following $f$:
\begin{equation}
    f(X) = \frac{1}{4}\fronorm{X - \diag(1,0_{n-1})}^2.
\end{equation}
The function $f$ admits a unique rank one PSD minimizer $X_0 = \diag(1,0_{n-1})$. Note that the above objective function is the augmented Lagrangian with $y=0$ for \eqref{eq: sdp.p}  with the following choice of $\Amap$, $b$, and $C$:
\begin{equation}
    \Amap = \mathcal{I}, \quad b = \diag(1,0_{n-1}),\quad \text{and}\quad C =0.
\end{equation}

Next, consider \eqref{eq: f_g_A.BM} with any $r>1$:
\begin{equation}
    h(F) = f(FF^\top) = \frac{1}{4}\fronorm{FF^\top - \diag(1,0_{n-1})}^2, \quad \text{where}\quad F\in \real^{n\times r}.
\end{equation}
Let us apply \eqref{eq: GD} to the above problem with the following initialization:
\begin{equation}\label{eq: gd_sublinear_init}
    F_1 = \begin{bmatrix}
        1 & 0 & 0_{1\times {r-2}}\\ 
        0 & \epsilon & 0 _{1\times {r-2}}\\
        0_{(n-2)\times 1} &   0_{(n-2)\times 1}& 0 _{(n-2)\times {r-2}}
    \end{bmatrix}.
\end{equation}
Namely, $F_1$ is only nonzero with its upper left $2\times 2$ block being $\begin{bmatrix}
1 & 0 \\
0 & \epsilon
\end{bmatrix}$ for an arbitrary $\epsilon\in (0,\frac{1}{2})$. Note that $\fronorm{F_1F_1^\top -X_0}= \epsilon^2$. Hence, the initialization can be arbitrarily close to $X_0$. We shall prove that for \eqref{eq: GD} with any initialization $F_1$ of the form \eqref{eq: gd_sublinear_init} and any step size $\eta \in (0,1)$, the iterate $F_{t+1}$ satisfies that   
\begin{equation}\label{eq: F_tX_0_sublinear_exmp}
     \frac{1}{t \eta + \frac{1}{\epsilon^2}}\geq  \fronorm{F_{t+1}F_{t+1}^\top -X_0}\geq \frac{1}{4t \eta + \frac{1}{\epsilon^2}}.
\end{equation}
Hence, the convergence rate is sublinear. 

With the above setup, it can be verified that the iterate $F_t$ of \eqref{eq: GD} with any stepsize is of the following form: 
\begin{equation}
    F_t =  \begin{bmatrix}
        1 & 0 & 0_{1\times {r-2}}\\ 
        0 & a_t & 0 _{1\times {r-2}}\\
        0_{(n-2)\times 1} &   0_{(n-2)\times 1}& 0 _{(n-2)\times {r-2}}
    \end{bmatrix},
\end{equation}
where the entry $a_t$ admits the following recursive relation: 
\begin{equation}\label{eq: at_rr}
a_{t+1} = a_t - \eta a_t^3 = (1-\eta a_t^2)a_t. 
\end{equation}
With $\epsilon\in (0,\frac{1}{2})$ and $\eta\in (0,1)$, it is easily verified that $a_t\in (0,\epsilon)$ always. Let us show how $a_t$ evolves. By letting $b_t = a_t^2$, from \eqref{eq: at_rr}, we find that 
\begin{equation}
    \begin{aligned}
     & b_{t+1} = b_t (1-2\eta b_t + \eta^2 b_t^2) 
\iff  \frac{1}{b_{t+1}}  = \frac{1}{b_t} + \frac{2\eta -\eta^2 b_t}{1-2\eta b_t + \eta^2 b_t^2}. 
\overset{(a)}{\implies} \frac{1}{b_t} + \eta \leq 
\frac{1}{b_{t+1}} \leq \frac{1}{b_t} + 4\eta \\
&\implies \frac{1}{\epsilon^2} + t\eta = \frac{1}{b_1} + t\eta
\leq \frac{1}{b_{t+1}} \leq \frac{1}{b_1} + 4t\eta = \frac{1}{\epsilon^2} + 4t\eta 
\implies  \frac{1}{t \eta + \frac{1}{\epsilon^2}}\geq b_{t+1}\geq \frac{1}{4t \eta + \frac{1}{\epsilon^2}}.
    \end{aligned}
\end{equation}
In the step $(a)$, we use that $\eta >\eta^2b_t >0$ and $1-2\eta b_t + \eta^2 b_t^2 \in (\frac{1}{2},1)$ as $\eta\in (0,1)$, $\epsilon\in (0,\frac{1}{2})$, and $b_t \in (0,\epsilon^2)$. By noting that $\fronorm{F_{t}F_t^\top - X_0} = b_t$, we see \eqref{eq: F_tX_0_sublinear_exmp} holds and the proof is complete.

\section{Lemmas and remaining proofs for Section \ref{sec: overparametrized ALG-BM}}\label{sec: appendix_over_param}

We first present the lemma on the properties of the optimization problem 
$\min_{OO^\top = I_r, O\in \mathbb{R}^{r\times r}} \fronorm{F - \bar{F}O}^2$.
\begin{lem}\label{lem: Delta_barF}
Given any $\bar{F}\in \mathbb{R}^{n\times r}$ and $F\in \mathbb{R}^{n\times r}$,  let $\bar{O} = \arg\min_{OO^\top = I_r, O\in \mathbb{R}^{r\times r}} \fronorm{F - \bar{F}O}^2$ and $\Delta = F - \bar{F}\bar{O}$. Then we have 
(i) $\Delta^\top \bar{F} \bar{O}= (\bar{F}\bar{O})^\top \Delta$, 
(ii) for any $G\in \mathbb{R}^{n\times r}$ such that $GG^\top = \bar{F}\bar{F}^\top$, there is an $O \in \mathbb{R}^{r\times r}$ such that $OO^\top =I$ and $G= \bar{F}O$, and 
(iii) $\bar{F}\bar{O} \in \arg\min_{GG^\top  = \bar{F}\bar{F}^\top}\fronorm{F - G}$.
\end{lem}
\begin{proof}
For the first item, note that for the optimization problem $\min_{OO^\top = I_r} \fronorm{F - \bar{F}O}^2$,  it is standard to verify that linear independence constraint qualification is satisfied (by considering the lower triangular part of the matrix) and the minimizer exists due to compactness of the feasible region and continuity of the objective. The conclusion then follows from the first order optimality condition. We omit the details.

For the second item, let $X = \bar{F}\bar{F}^\top$ and $\bar{r} = \rank(X)$. Let the eigenvalue decomposition of $X = U\Sigma^2 U^\top$ where $U\in \mathbb{R}^{n\times \bar{r}}$ has orthonormal columns and $\Sigma \in \mathbb{R}^{\bar{r}\times\bar{r}}$ is a diagonal matrix with positive diagonal. Note that $V_{\bar{F}} =  \bar{F} ^\top U \Sigma^{-1} \in \mathbb{R}^{r\times \bar{r}}$ satisfies that $V_{\bar{F}}^\top V_{\bar{F}} = I_{\bar{r}}$, i.e., it has orthonormal columns, and $U\Sigma V_{\bar{F}} ^\top = U \Sigma \Sigma^{-1} U^\top \bar{F} = UU^\top \bar{F} = \bar{F}$ where we use the fact that $UU^\top$ is the orthogonal projection of $\range(X) = \range(\bar{F})$. Furthermore, we can complete the basis of the columns of $V_{\bar{F}}$ and obtain an orthonormal $\bar{V}_{\bar{F}} \in \mathbb{R}^{r\times r}$ whose first $\bar{r}$ columns are $V_{\bar{F}}$. Then, the matrix $\bar{F} = U \bar{\Sigma}\bar{V}_{\bar{F}}^\top$ where $\bar{\Sigma} = \begin{bmatrix}
    \bar{\Sigma } & 0 \\
    0 & 0_{(r-\bar{r}) \times (r-\bar{r})}
\end{bmatrix}.$ Similarly, we can have $G = U \bar{\Sigma} \bar{V}_G^\top$ where $\bar{V}_G \in \mathbb{R}^{r\times r}$ is orthonormal with first $\bar{r}$ columns being $V_G = G^\top U \Sigma^{-1}$. We can now let $O= \bar{V}_{\bar{F}} \bar{V}_{G}^\top$. 

The last item is immediate given the second item as it implies $\{\bar{F}O\mid OO^\top = I_r\} = \{G \in \mathbb{R}^{n\times r} \mid GG^\top = \bar{F}\bar{F}^\top \}$.
\end{proof}

Our next lemma upper bounds the distance on the factors in terms of the distance in the original space for possibly rank-deficient matrices.  
\begin{lem}\label{lem: over_para_F_bar_F_distance_FF_Fbar_Fbar_relation}
Consider any $\bar{F},F\in \mathbb{R}^{n\times r}$ with $\bar{r} = \rank(\bar{F})$ and $r\geq \bar{r}$.
Let $F_\pi \in \arg\min_{GG^\top = \bar{F}\bar{F}^\top}\fronorm{F-G}$.
Then 
\begin{equation}
\fronorm{F-F_\pi} \leq 
\sqrt{\nucnorm{FF^\top -\bar{F}\bar{F}^\top}} + 
\frac{3}{2\sigma_{\min>0}(\bar{F})}\fronorm{FF^\top -\bar{F}\bar{F}^\top}
\end{equation}
\end{lem}
\begin{proof}
Denote the SVD of $F=U\Sigma V^\top$ where $U\in \mathbb{R}^{n\times r}$ has orthonormal columns, and $V\in \mathbb{R}^{r\times r}$ is orthonormal, and $\Sigma\in \mathbb{R}^{r\times r}$ is a diagonal matrix with nonnegative diagonals. 
Let $F_{\bar{r}}= U\Sigma_{\bar{r}}V^\top$ where $\Sigma_{\bar{r}}$ keeps all the entries of $\Sigma$ but changes the smallest $r-\bar{r}$ diagonals to $0$. Similarly, let the SVD of $\bar{F} = \bar{U}\bar{\Sigma}\bar{V}^\top$ where $\bar{\Sigma} \in \mathbb{R}^{\bar{r}\times r}$ is diagonal and $\bar{U}\in \mathbb{R}^{n\times \bar{r}}$ and $\bar{V}\in \mathbb{R}^{r\times r}$ both have orthonormal columns. 
Denote 
$G_{\bar{r},\pi,0} = 
\arg\min_{G =\bar{U}\bar{\Sigma}\diag(O,I_{r-\bar{r}}),\;OO^\top = I_{\bar{r}},O\in \mathbb{R}^{\bar{r}\times \bar{r}}} \fronorm{U\Sigma_{\bar{r}}-G}$. Here, the matrix  $\diag(O,I_{r-\bar{r}}) = \begin{bmatrix}
O & 0 \\
0 & I_{r-\bar{r}}
\end{bmatrix}$. Let $\Pi_{\bar{r}}: \mathbb{R}^{n\times r} \rightarrow \mathbb{R}^{n\times \bar{r}}$ be the operator of taking the first $\bar{r}$ columns. Finally, let $F_{\bar{r},\pi,0} = G_{\bar{r},\pi,0}V^\top$. 
\footnote{The complicated setup is trying to apply 
\cite[Lemma 6]{ge2017no} in the following step $(d)$ in \eqref{eq: fourth_order_growth_mf} as it requires both matrices to be of size $n\times \bar{r}$.
}
With the above setup, we have
\begin{equation}
\begin{aligned}\label{eq: fourth_order_growth_mf}
\fronorm{F - F_\pi} 
& 
\overset{(a)}{\leq} 
\fronorm{F - F_{\bar{r},\pi,0}} 
\leq \fronorm{F-F_{\bar{r}}} + 
\fronorm{F_{\bar{r}}-F_{\bar{r},\pi,0}} \\ 
& 
\overset{(b)} = \fronorm{F-F_{\bar{r}}} + 
\fronorm{U\Sigma_{\bar{r}}- G_{\bar{r},\pi,0}}\\
&
\overset{(c)} = \fronorm{F-F_{\bar{r}}} + 
\fronorm{\Pi_{\bar{r}}(U\Sigma_{\bar{r}})- \Pi_{\bar{r}}(G_{\bar{r},\pi,0})}\\
& 
\overset{(d)}{\leq} 
\sqrt{\nucnorm{FF^\top -F_{\bar{r}}F_{\bar{r}}^\top}} + 
\frac{3}{4\sigma_{\min>0}(\bar{F})}\fronorm{F_{\bar{r}}F_{\bar{r}}^\top -\bar{F}\bar{F}^\top}\\
& 
\overset{(e)}{\leq} \sqrt{\nucnorm{FF^\top -\bar{F}\bar{F}^\top}} + 
\frac{3}{4\sigma_{\min>0}(\bar{F})}\fronorm{F_{\bar{r}}F_{\bar{r}}^\top -\bar{F}\bar{F}^\top}\\
&\overset{(f)}{\leq} \sqrt{\nucnorm{FF^\top -\bar{F}\bar{F}^\top}} + 
\frac{3}{2\sigma_{\min>0}(\bar{F})}\fronorm{FF^\top-\bar{F}\bar{F}^\top},\\
\end{aligned}
\end{equation}
where the step $(a)$ is due to the optimality of $F_\pi$. In the step $(b)$, we use the invariance property of Frobenius norm under orthogonal transformation and the definition of $F_{\bar{r}}$ and $F_{\bar{r},\pi,0}$. The step $(c)$ is because both $U\Sigma_{\bar{r}}$ and $G_{\bar{r},\pi,0}$ have their last $r-\bar{r}$ columns being zero. 
The step $(d)$ uses that 
$\Pi_{\bar{r}}(U\Sigma_{\bar{r}})\Pi_{\bar{r}}(U\Sigma_{\bar{r}})^\top=F_{\bar{r}}F_{\bar{r}}^\top$, 
$\Pi_{\bar{r}}(G_{\bar{r},\pi,0})\Pi_{\bar{r}}(G_{\bar{r},\pi,0})^\top = \bar{F} \bar{F}^\top$, and \cite[Lemma 6]{ge2017no}. The step $(e)$ is because $F_{\bar{r}}F_{\bar{r}}^\top$ is the best rank-$\bar{r}$ approximation of $FF^\top$ in terms of nuclear norm. The last step $(f)$ is due to the following inequality: 
$
\fronorm{F_{\bar{r}}F_{\bar{r}}^\top -\bar{F}\bar{F}^\top} \leq 
\fronorm{FF^\top -\bar{F}\bar{F}^\top} 
+ 
\fronorm{F_{\bar{r}}F_{\bar{r}}^\top-FF^\top} \overset{(a)}{\leq} 
2\fronorm{FF^\top -\bar{F}\bar{F}^\top}
$, where the step $(a)$ is because $F_{\bar{r}}F_{\bar{r}}^\top$ is the best rank-$\bar{r}$ approximation of $FF^\top$.
\end{proof}

We now prove the Lipschitz property in Lemma \ref{lem: local_lips_grad_hessian}. 
\begin{proof}
The Lipschitzness of the gradient is a direct consequence of the local smoothness part of Lemma \ref{thm: f_g_A} and that (i) $\nabla_X\mathcal{L}_\rho$ has Lipschitz constant $\rho \opnorm{\Amap}^2$, and (ii) $\nabla_X\mathcal{L}_\rho(\Xyp)=Z(\zyp)$ due to the KKT condition \eqref{eq: ALmin.KKT}.\footnote{Although Lemma 5 was stated in the exact-rank setting, its local smoothness part only uses $\fronorm{\Delta}\leq \frac{1}{4}\sqrt{\sigma_{\min>0}(\Xyp)}$
, and therefore applies verbatim for $r\geq \rstar$.}
To show the Lipschitzness of the Hessian, note that for any $H\in \mathbb{R}^{n\times r}$, the quadratic form of the Hessian of $\bar{\mathcal{L}}_\rho$ \footnote{We omit the dependence on $y$ here as $y$ is fixed.} at any $F$ is $\nabla^2 \bar{\mathcal{L}}_\rho(F) [H,H] = 
2\inprd{C-\Amap^*y+\rho\Amap^*(\Amap(FF^\top)-b)}{HH^\top}+ 
\rho \fronorm{\Amap(FH^\top + HF^\top)}^2$. Define $\Delta = F-\bar{F}$. We may decompose the difference $\nabla^2 \bar{\mathcal{L}}_\rho(F) [H,H] - 
\nabla^2 \bar{\mathcal{L}}_\rho(\bar{F}) [H,H]$
in the following ways: 
\begin{equation*}
\begin{aligned}
& \nabla^2 \bar{\mathcal{L}}_\rho(F) [H,H]
-
\nabla^2 \bar{\mathcal{L}}_\rho(\bar{F})[H,H] \\
=& 
2\inprd{\Amap^*\Amap(FF^\top -\bar{F}\bar{F}^\top)}{HH^\top} + \rho 
\left(
\fronorm{\Amap(FH^\top + HF^\top}^2
-\fronorm{\Amap(\bar{F}H^\top + H\bar{F}^\top)}^2\right)\\
= &  2\inprd{\Amap^*\Amap(F\Delta^\top +
\Delta F^\top +\Delta \Delta^\top)}{HH^\top} + 
\rho \inprd{FH^\top +HF^\top - (\bar{F}H^\top + H\bar{F}^\top)}{\Amap^*\Amap(FH^\top +HF^\top)} \\ 
&+ \rho \inprd{FH^\top +HF^\top -( \bar{F}H^\top + H\bar{F}^\top)}{\Amap^*\Amap(\bar{F}H^\top +H\bar{F}^\top)}.
\end{aligned}
\end{equation*}
It is then standard to bound the three terms using the properties of trace inner product, Cauchy-Schwarz inequality, and the property of Frobenius norms. We omit the details. 
\end{proof}

Finally, we prove Lemma \ref{lem: Hessian_pd_critical_direction} showing the Hessian is positive definite for points around the solution set along  certain directions. 
\begin{proof}[Proof of Lemma \ref{lem: Hessian_pd_critical_direction}]
We have the following estimate:
\begin{equation}
    \begin{aligned}
\nabla^2 \Lf(F,y)[\Delta,\Delta] 
=& 
\nabla^2 \Lf(\bar{F},y)[\Delta,\Delta] + (\nabla^2 \Lf(F,y)- \nabla^2 \Lf(\bar{F},y))[\Delta,\Delta]\\
\overset{(a)}{\geq} &
2\sigma_{AL} \fronorm{\Delta}^2 - 
L_{\rho,2}\fronorm{\Delta}^2\fronorm{F-\bar{F}}
\overset{(b)}{\geq} \sigma_{AL} \fronorm{\Delta}^2.
\end{aligned}
\end{equation}
In the step $(a)$, to bound the first term, we use Lemma \ref{lem: pd_hessian_AL_along_critical_directions}, which is activated due to that $\nabla ^2 \Lf (\bar{F},y)[\Delta,\Delta] = 
2\inprd{Z(\zyp)}{\Delta\Delta^\top}+ 
\rho \twonorm{\Amap(\Delta \bar{F}^\top + \bar{F} \Delta^\top) }^2
=2\mathcal{Q}_{Z(\zyp),\bar{F}}(\Delta)$ and that $(\Delta,\bar{F}, Z(\zyp))\in \mathcal{H}$ for 
$\mathcal{Q}$ and $\mathcal{H}$ 
defined in Lemma \ref{lem: pd_hessian_AL_along_critical_directions}, due to the bound on $\twonorm{y-\ysol}$. 
To bound the second term, we use the Lipschitzness of the Hessian in Lemma \ref{lem: local_lips_grad_hessian}. The last step $(b)$ is due to the bound on $\fronorm{F-\bar{F}}$.
\end{proof}

\end{document}